\documentclass[a4j,10pt,showkeys]{article}

\usepackage{amsmath,amsthm,amssymb}
\usepackage{geometry}
\usepackage{hyperref}
\usepackage[capitalize]{cleveref}
\usepackage{cancel}
\usepackage{mathtools}
\usepackage{bm}
\usepackage{color}
\usepackage{mathrsfs}

\newtheorem{theo}{Theorem}[section]
\newtheorem{lemm}[theo]{Lemma}
\newtheorem{prop}[theo]{Proposition}
\newtheorem{cor}[theo]{Corollary}

\theoremstyle{definition}
\newtheorem{defi}[theo]{Definition}
\newtheorem{rem}[theo]{Remark}

\newtheorem{exam}[theo]{Example}

\newcommand{\bP}{\mathbb{P}}
\newcommand{\bE}{\mathbb{E}}
\newcommand{\bF}{\mathbb{F}}
\newcommand{\bR}{\mathbb{R}}
\newcommand{\bN}{\mathbb{N}}

\newcommand{\cF}{\mathcal{F}}

\newcommand{\cB}{\mathcal{B}}
\newcommand{\cM}{\mathcal{M}}

\newcommand{\cK}{\mathcal{K}}
\newcommand{\cJ}{\mathcal{J}}
\newcommand{\cV}{\mathcal{V}}

\newcommand{\fF}{\mathfrak{F}}
\newcommand{\fW}{\mathfrak{W}}
\newcommand{\bfF}{\overline{\mathfrak{F}}}
\newcommand{\bfW}{\overline{\mathfrak{W}}}
\newcommand{\rd}{\,\mathrm{d}}
\newcommand{\op}{\mathrm{op}}

\newcommand{\ep}{\varepsilon}

\newcommand{\1}{\mbox{\rm{1}}\hspace{-0.25em}\mbox{\rm{l}}}

\newcommand{\CD}{^\mathrm{C}\!D^\alpha_{0+}}
\newcommand{\knorm}{|\hspace{-0.2mm}|\hspace{-0.2mm}|}
\newcommand{\Bknorm}{\Big|\hspace{-0.2mm}\Big|\hspace{-0.2mm}\Big|}
\newcommand{\bstar}{\overset{\leftarrow}{\star}}
\newcommand{\bast}{\overset{\leftarrow}{\ast}}

\providecommand{\keywords}[1]{\textbf{Keywords:} #1}

\makeatletter

\@addtoreset{equation}{section}
\makeatother
\def\widebar{\accentset{{\cc@style\underline{\mskip10mu}}}}
\numberwithin{equation}{section}
\def\rnum#1{\expandafter{\romannumeral #1}} 
\def\Rnum#1{\uppercase\expandafter{\romannumeral #1}}

\allowdisplaybreaks

\title{Variation of constants formulae for forward and backward\\stochastic Volterra integral equations}
\author{
	Yushi Hamaguchi\thanks{
	Graduate School of Engineering Science, Department of Systems Innovation,
	Osaka University.
	1-3, Machikaneyama, Toyonaka, Osaka, Japan.
	\href{mailto:hmgch2950@gmail.com}{hmgch2950@gmail.com}}
}

\begin{document}
\maketitle

\begin{abstract}
In this paper, we provide variation of constants formulae for linear (forward) stochastic Volterra integral equations (SVIEs, for short) and linear Type-II backward stochastic Volterra integral equations (BSVIEs, for short) in the usual It\^{o}'s framework. For these purposes, we define suitable classes of stochastic Volterra kernels and introduce new notions of the products of adapted $L^2$-processes. Observing the algebraic properties of the products, we obtain the variation of constants formulae by means of the corresponding resolvent. Our framework includes SVIEs with singular kernels such as fractional stochastic differential equations. Also, our results can be applied to general classes of SVIEs and BSVIEs with infinitely many iterated stochastic integrals. The duality principle between generalized SVIEs and generalized BSVIEs is also proved.
\end{abstract}
\keywords
Stochastic Volterra integral equations; backward stochastic Volterra integral equations; variation of constants formulae.



\tableofcontents


\section{Introduction}\label{section_Intro}

In this paper, we provide variation of constants formulae for linear stochastic Volterra integral equations (SVIEs, for short) and linear backward stochastic Volterra integral equations (BSVIEs, for short). First of all, consider the following SVIE:
\begin{equation}\label{Intro_eq_SVIE}
	X(t)=\varphi(t)+\int^t_0j(t,s)X(s)\rd s+\int^t_0k(t,s)X(s)\rd W(s),\ t\in(0,T),
\end{equation}
where $W$ is a one-dimensional Brownian motion, $j$ and $k$ are given deterministic functions called the kernels, and $\varphi$ is a given adapted process called the free term. If the kernels $j(t,s)$ and $k(t,s)$ do not depend on the first time-parameter $t$, and if the free term is of the form $\varphi(t)=x_0+\int^t_0b(s)\rd s+\int^t_0\sigma(s)\rd W(s)$ for some constant $x_0$ and adapted processes $b$ and $\sigma$, then SVIE~\eqref{Intro_eq_SVIE} is reduced to a linear stochastic differential equation (SDE, for short)
\begin{equation*}
	\begin{cases}
	\mathrm{d}X(t)=\{j(t)X(t)+b(t)\}\rd t+\{k(t)X(t)+\sigma(t)\}\rd W(t),\ t\in(0,T),\\
	X(0)=x_0.
	\end{cases}
\end{equation*}
More importantly, SVIE~\eqref{Intro_eq_SVIE} includes a class of fractional SDEs of the following form:
\begin{equation*}
	\begin{cases}\displaystyle
	\CD X(t)=j(t)X(t)+b(t)+\{k(t)X(t)+\sigma(t)\}\frac{\mathrm{d}W(t)}{\mathrm{d}t},\ t\in(0,T),\\
	X(0)=x_0,
	\end{cases}
\end{equation*}
where $\CD$ denotes the Caputo fractional differential operator of order $\alpha\in(\frac{1}{2},1]$. Indeed, by the definition, an adapted process $X$ is called a solution to the above fractional SDE if it satisfies the integral equation
\begin{equation*}
	X(t)=x_0+\frac{1}{\Gamma(\alpha)}\int^t_0(t-s)^{\alpha-1}\{j(s)X(s)+b(s)\}\rd s+\frac{1}{\Gamma(\alpha)}\int^t_0(t-s)^{\alpha-1}\{k(s)X(s)+\sigma(s)\}\rd W(s),\ t\in(0,T),
\end{equation*}
where $\Gamma(\alpha)$ is the Gamma function (see \cite{WaXuKl16}). Thus, the above fractional SDE can be seen as SVIE~\eqref{Intro_eq_SVIE} with the singular kernels 
\begin{equation*}
	j(t,s)=\frac{1}{\Gamma(\alpha)}(t-s)^{\alpha-1}j(s),\ k(t,s)=\frac{1}{\Gamma(\alpha)}(t-s)^{\alpha-1}k(s),
\end{equation*}
and the free term
\begin{equation*}
	\varphi(t)=x_0+\frac{1}{\Gamma(\alpha)}\int^t_0(t-s)^{\alpha-1}b(s)\rd s+\frac{1}{\Gamma(\alpha)}\int^t_0(t-s)^{\alpha-1}\sigma(s)\rd W(s).
\end{equation*}
Fractional differential systems are suitable tools to describe the dynamics of systems with memory effects and hereditary properties. There are many applications of fractional calculus in a variety of research fields including mathematical finance, physics, chemistry, biology, and other applied sciences. For detailed accounts of theory and applications of fractional calculus, see for example \cite{DaBa10,Di07,Ra10,SaKiMa87} and the references cited therein.

It is well-known that a variation of constants formula for deterministic Volterra integral equations is an important tool in many qualitative theory such as the linear perturbation theory (see Chapter 11 of \cite{GrLoSt90}), the invariant manifold theory (see \cite{DiGi84}), and the linear-quadratic optimal control theory (see \cite{HaLiYo22}). In the first half of this paper, we extend the variation of constants formula to the stochastic case \eqref{Intro_eq_SVIE}.

A difficulty to solve SVIE~\eqref{Intro_eq_SVIE} comes from the term $\int^t_0k(t,s)X(s)\rd W(s)$. In order to treat the stochastic convolution term, Berger and Mizel~\cite{BeMi82} introduced an extended stochastic integral for anticipated integrands, and gave a variation of constants formula in terms of the anticipated stochastic integral. Also, {\O}ksendal and Zhang~\cite{OkZh93,OkZh96} studied SVIE~\eqref{Intro_eq_SVIE}, where the free term $\varphi$ is not necessarily adapted, and the stochastic integral is taken in the Skorohod sense. They showed variation of constants formulae for \eqref{Intro_eq_SVIE} in extended frameworks of the functional process approach and the generalized white noise functional (Hida distribution) approach. Unfortunately, the frameworks of \cite{BeMi82,OkZh93,OkZh96} cannot be applied to fractional SDEs due to some strong regularity and boundedness assumptions on the kernels $j$ and $k$. For other studies on (nonlinear) SVIEs in the framework of anticipated calculus, see Berger and Mizel~\cite{BeMi80,BeMi80+}, Pardoux and Protter~\cite{PaPr90} and Al\`{o}s and Nualart~\cite{AlNu97}.

In this paper, we provide another viewpoint to solve SVIE~\eqref{Intro_eq_SVIE} in the classical It\^{o}'s framework without the use of anticipated stochastic integrals. Our idea is to regard SVIE~\eqref{Intro_eq_SVIE} as an algebraic equation in the space of adapted $L^2$-processes of the following form:
\begin{equation}\label{Intro_eq_SVIEalg}
	X=\varphi+J\ast X+K\star X.
\end{equation}
Here, the $\ast$-product $J\ast X$ with $J=j$ corresponds to the convolution $\int^t_0j(t,s)X(s)\rd s$ in terms of the Lebesgue integral, and the $\star$-product $K\star X$ with $K(t)=\int^t_0k(t,s)\rd W(s)$ corresponds to the stochastic convolution $\int^t_0k(t,s)X(s)\rd W(s)$ in terms of the stochastic integral. The $\star$-product for general adapted processes is defined by means of the Wiener--It\^{o} chaos expansion (see \cref{star_defi_starprod}), and it has a representation by infinitely many iterated stochastic integrals in the classical It\^{o}'s sense (see \cref{star_prop_integral}). We introduce suitable classes of stochastic Volterra kernels $J$ and $K$ and show that these spaces become Banach algebras with respect to the $\ast$- and the $\star$-products, respectively (see \cref{star_prop_Banachalg} and \cref{ast_prop_Banachalg}). Then, we show that the solution to the linear equation \eqref{Intro_eq_SVIEalg} is given by the variation of constants formula
\begin{equation*}
	X=\varphi+Q\ast\varphi+R\star\varphi,
\end{equation*}
where $Q$ and $R$ are resolvents of $J$ and $K$ in terms of the $\ast$- and the $\star$-products (see \cref{fVCF_theo_VCF}). The main advantages of our frameworks are three fold. First, we do not need any anticipated calculus in the above expressions, and all stochastic integrals are understood in the usual It\^{o}'s sense. Second, our framework includes linear SVIEs with singular kernels such as fractional SDEs; in \cref{fVCF_exam_FBS}, we apply our general results to a time-fractional version of the Black--Scholes SDE and give an explicit expression of the solution. Third, we can apply our analysis to a more general class of SVIEs with (infinitely many) iterated stochastic integrals (see equation \eqref{SVIE_eq_GSVIE}), which is beyond the framework of \cite{BeMi82,OkZh93,OkZh96}. This class of SVIEs includes fractional SDEs with delay and ``noisy memory'' (see Example~\ref{SVIE_exam_FSDEnoisy}).

Next, consider the following equations:
\begin{equation}\label{Intro_eq_BSVIE1}
	Y(t)=\tilde{\psi}(t)+\int^T_t\{j(s,t)^\top Y(s)+k(s,t)^\top Z(t,s)\}\rd s-\int^T_tZ(t,s)\rd W(s),\ t\in(0,T),
\end{equation}
and
\begin{equation}\label{Intro_eq_BSVIE2}
	Y(t)=\tilde{\psi}(t)+\int^T_t\{j(s,t)^\top Y(s)+k(s,t)^\top Z(s,t)\}\rd s-\int^T_tZ(t,s)\rd W(s),\ t\in(0,T),
\end{equation}
where $j$ and $k$ are given deterministic kernels, and $\tilde{\psi}$ is a given (not necessarily adapted) stochastic process. Equations \eqref{Intro_eq_BSVIE1} and \eqref{Intro_eq_BSVIE2} are called linear Type-I and Type-II backward stochastic Volterra integral equations (BSVIEs, for short), respectively. The difference between Type-I and Type-II BSVIEs is the role of the time-parameters of $Z(t,s)$ and $Z(s,t)$ in the drivers. On the one hand, Type-I BSVIE \eqref{Intro_eq_BSVIE1} is an equation for a pair $(Y(\cdot),Z(\cdot,\cdot))$ with $Z(t,s)$ defined on $0<t<s<T$, and we call it an adapted solution to \eqref{Intro_eq_BSVIE1} if $Y$ is adapted, $Z(t,\cdot)=(Z(t,s))_{s\in(t,T)}$ is adapted for each $t\in(0,T)$, and the equality \eqref{Intro_eq_BSVIE1} holds for a.e.\ $t\in(0,T)$, a.s. Nonlinear Type-I BSVIEs were introduced by Lin~\cite{Li02} and Yong~\cite{Yo06,Yo08}. Recently, many authors have applied Type-I BSVIEs to stochastic control and mathematical finance where the time-inconsistency is taken into account. For example, time-inconsistent dynamic risk measures and time-inconsistent recursive utilities can be modeled by the solutions to Type-I BSVIEs (see \cite{Yo07,KrOv17,Ag19,WaSuYo21}). Time-inconsistent stochastic control problems related to Type-I BSVIEs were studied by Wang and Yong~\cite{WaYo21} and the author~\cite{Ha21}. On the other hand, Type-II BSVIE~\eqref{Intro_eq_BSVIE2} is an equation for $(Y(\cdot),Z(\cdot,\cdot))$ with $Z(t,s)$ defined on $(t,s)\in(0,T)^2$, and we call it an adapted M-solution to \eqref{Intro_eq_BSVIE2} if $Y$ is adapted, $Z(t,\cdot)=(Z(t,s))_{s\in(0,T)}$ is adapted for each $t\in(0,T)$, and the equality \eqref{Intro_eq_BSVIE2}, together with the relation $Y(t)=\bE[Y(t)]+\int^t_0Z(t,s)\rd W(s)$ hold for a.e.\ $t\in(0,T)$, a.s. Type-II BSVIEs and the concept of adapted M-solutions were first introduced by Yong~\cite{Yo08} and applied to the duality principle appearing in stochastic control problems for (forward) SVIEs. Since then, Type-II BSVIEs have been found important to study stochastic control problems of SVIEs (see \cite{ChYo07,ShWaYo13,ShWaYo15,WaTZh17,WaT18,WaT20,ShWeXi20,Ha21+,Ha21++}). Specifically, Chen and Yong~\cite{ChYo07} and Wang~\cite{WaT18} showed that the optimal controls for linear-quadratic control problems of SVIEs are characterized by using the adapted M-solutions of linear Type-II BSVIEs. See also our previous papers \cite{Ha21+,Ha21++} for the counterparts in the infinite horizon setting.

For linear Type-I BSVIEs, Hu and {\O}ksendal~\cite{HuOk19}, Wang, Yong and Zhang~\cite{WaYoZh20}, Ren, Coulibaly and Aman~\cite{ReCoAm21} and the author \cite{Ha21+} obtained explicit formulae of the solutions in different settings. However, due to the difficulty of the appearance of the term $Z(s,t)$ in the driver, explicit solutions of linear Type-II BSVIEs have not been obtained in the literature, except for some concrete examples. In the latter half of this paper, we focus on linear Type-II BSVIEs. We provide a variation of constants formula for a general class of linear Type-II BSVIEs and give explicit expressions of the adapted M-solutions. For this purpose, we introduce new notions of the products which we call the backward $\star$-product and the backward $\ast$-product for adapted processes (see \cref{bstar_defi_bstarprod} and \cref{bast_defi_bastprod}). Then, regarding a generalized class of linear Type-II BSVIEs as a simple algebraic equation including (iterated) martingale representation operations, we obtain the variation of constants formula (see \cref{bVCF_theo_VCF}). Furthermore, we provide a duality principle between linear SVIEs and linear Type-II BSVIEs with infinitely many iterated stochastic integrals, which extends the result of Yong~\cite{Yo08} to a more general framework (see \cref{bVCF_theo_duality}).

The established variation of constants formulae and the duality principle in this paper are natural extensions of the counterparts of deterministic Volterra integral equations (see \cite{GrLoSt90}) to the stochastic one. Our results have potential applications to many qualitative theories such as the linear perturbation theory, the invariant manifold theory, and the linear-quadratic optimal control theory for SVIEs. We leave these applications as the future research.

We summarize the contributions of this paper.
\begin{itemize}
\item
We introduce novel kinds of products for adapted $L^2$-processes and show their fundamental properties.
\item
We introduce and study the corresponding resolvents, which play crucial roles for solving SVIEs and BSVIEs.
\item
We prove variation of constants formulae for SVIEs and BSVIEs, which provide explicit expressions of the Wiener--It\^{o} chaos expansions of the solutions.
\item
We provide useful sufficient conditions for the existence of the resolvent with respect to the $\star$-product.
\end{itemize}
This paper is organized as follows. In \cref{section_Pre}, we recall the Wiener--It\^{o} chaos expansion and introduce the notion of deterministic Volterra kernels, which play fundamental roles in our study. \cref{section_SVIE} is concerned with linear SVIEs. In \cref{subsection_star} and \cref{subsection_ast}, we introduce the notion of the $\star$- and the $\ast$-products, respectively. Resolvents with respect to these products are studied in \cref{subsection_res}. Then, in \cref{subsection_fVCF}, we show the variation of constants formula for (generalized) SVIEs, together with an application to the fractional Black--Scholes SDE. \cref{section_BSVIE} is concerned with linear BSVIEs. Introducing the notions of the backward $\star$- and the backward $\ast$-products in \cref{subsection_bstar} and \cref{subsection_bast}, respectively, we show in \cref{subsection_bVCF} the variation of constants formula for (generalized) BSVIEs. Also, we show the duality principle between SVIEs and BSVIEs in the general framework. Lastly, in \cref{section_exist}, we study the existence of the resolvent with respect to the $\star$-product.

\subsection*{Notations}

$(\Omega,\cF,\bP)$ is a complete probability space, and $W$ is a one-dimensional Brownian motion on $(\Omega,\cF,\bP)$. For each $s\geq0$, $\bF^s=(\cF^s_t)_{t\geq s}$ denotes the $\bP$-augmentation of the filtration generated by $(W(t)-W(s))_{t\geq s}$. When $s=0$, we denote $\bF^0=(\cF^0_t)_{t\geq0}$ by $\bF=(\cF_t)_{t\geq0}$. $\bE[\cdot]$ denotes the expectation. For each $t\geq0$, $\bE_t[\cdot]:=\bE[\cdot|\cF_t]$ is the conditional expectation given by $\cF_t$.

For each $d_1,d_2\in\bN$, we denote the space of $(d_1\times d_2)$-matrices by $\bR^{d_1\times d_2}$. We define $\bR^{d_1}:=\bR^{d_1 \times 1}$, that is, each element of $\bR^{d_1}$ is understood as a column vector. For each $A\in\bR^{d_1\times d_2}$, $A^\top$ denotes the transpose of $A$. $|A|:=\sqrt{\mathrm{tr}(A A^\top)}$ denotes the Frobenius norm, and $|A|_\op$ denotes the operator norm of $A$ as an linear operator from $\bR^{d_2}$ to $\bR^{d_1}$ with respect to the Euclidean norms. For each $d\in\bN$, $I_d\in\bR^{d\times d}$ is the identity matrix. For each $n\in\bN$ and $0\leq S<T<\infty$, we denote by $\Delta_n(S,T)$ the set of $n$-tuples $(t_1,t_2,\mathalpha{\dots},t_n)$ such that $T>t_1>t_2>\mathalpha{\cdots}>t_n>S$. For each $t=(t_1,\mathalpha{\dots},t_n)\in\Delta_n(S,T)$, we sometimes use the notations $\max t:=t_1$ and $\min t:=t_n$. We note that $\Delta_1(S,T)=(S,T)$. Also, we define $\bN_0:=\bN\cup\{0\}$.

For each $0\leq S<T<\infty$, $d_1,d_2\in\bN$ and $n\in\bN$, we define the following spaces:
\begin{itemize}
\item
$L^2(\Delta_n(S,T);\bR^{d_1\times d_2})$: the set of $\bR^{d_1\times d_2}$-valued measurable (deterministic) functions $\xi$ on $\Delta_n(S,T)$ such that $\|\xi\|_{L^2(\Delta_n(S,T))}<\infty$, where
\begin{equation*}
	\|\xi\|_{L^2(\Delta_n(S,T))}:=\Bigl(\int^T_S\int^{t_1}_S\mathalpha{\cdots}\int^{t_{n-1}}_S|\xi(t_1,t_2,\mathalpha{\dots},t_n)|^2\rd t_n\mathalpha{\cdots}\rd t_1\Bigr)^{1/2}.
\end{equation*}
\item
$L^2_\cF(\Delta_n(S,T);\bR^{d_1\times d_2})$: the set of $\bR^{d_1\times d_2}$-valued $\cF\otimes\cB(\Delta_n(S,T))$-measurable random fields $\xi$ on $\Omega\times\Delta_n(S,T)$ such that $\|\xi\|_{L^2_\cF(\Delta_n(S,T))}<\infty$, where
\begin{equation*}
	\|\xi\|_{L^2_\cF(\Delta_n(S,T))}:=\bE\Bigl[\int^T_S\int^{t_1}_S\mathalpha{\cdots}\int^{t_{n-1}}_S|\xi(t_1,t_2,\mathalpha{\dots},t_n)|^2\rd t_n\mathalpha{\cdots}\rd t_1\Bigr]^{1/2}.
\end{equation*}
\item
$L^2_\bF(\Delta_n(S,T);\bR^{d_1\times d_2})$: the set of $\xi\in L^2_\cF(\Delta_n[S,T];\bR^{d_1\times d_2})$ such that $\xi(t_1,t_2,\mathalpha{\dots},t_n)$ is $\cF^S_{t_n}$-measurable for any $(t_1,t_2,\mathalpha{\dots},t_n)\in\Delta_n(S,T)$. We denote $\|\xi\|_{L^2_\bF(\Delta_n(S,T))}:=\|\xi\|_{L^2_\cF(\Delta_n(S,T))}$.
\item
$L^2_{\bF,*}(\Delta_n(S,T);\bR^{d_1\times d_2})$: the set of $\xi\in L^2_\cF(\Delta_n[S,T];\bR^{d_1\times d_2})$ such that $\xi(t_1,t_2,\mathalpha{\dots},t_n)$ is $\cF^{t_2}_{t_1}$-measurable for any $(t_1,t_2,\mathalpha{\dots},t_n)\in\Delta_n(S,T)$. We denote $\|\xi\|_{L^2_{\bF,*}(\Delta_n(S,T))}:=\|\xi\|_{L^2_\cF(\Delta_n(S,T))}$.
\end{itemize}

When $n=1$, we denote the spaces $L^2(\Delta_1(S,T);\bR^{d_1\times d_2})$, $L^2_\cF(\Delta_1(S,T);\bR^{d_1\times d_2})$ and $L^2_\bF(\Delta_1(S,T);\bR^{d_1\times d_2})$ by $L^2(S,T;\bR^{d_1\times d_2})$, $L^2_\cF(S,T;\bR^{d_1\times d_2})$ and $L^2_\bF(S,T;\bR^{d_1\times d_2})$, respectively.

Here we summarize the notations which we will introduce in the following sections. Let $0\leq S<T<\infty$ and $d,d_1,d_2,d_3\in\bN$ be fixed.
\begin{itemize}
\item
For each $n\in\bN_0$, $\fW_n:L^2(\Delta_{n+1}(S,T);\bR^{d_1\times d_2})\to L^2_\bF(S,T;\bR^{d_1\times d_2})$ and $\bfW_n:L^2(\Delta_{n+2}(S,T);\bR^{d_1\times d_2})\to L^2_\bF(\Delta_2(S,T);\bR^{d_1\times d_2})$ denote the $n$-th iterated stochastic integral operators. See \cref{chaos_defi_iterated}.
\item
For each $\xi\in L^2_\bF(S,T;\bR^{d_1\times d_2})$ and $\Xi\in L^2_\bF(\Delta_2(S,T);\bR^{d_1\times d_2})$, $\fF_n[\xi]$ and $\bfF_n[\Xi]$ denote the integrands appearing in the $n$-th Wiener--It\^{o} chaos expansions of $\xi$ and $\Xi$, and $\fW_n[\xi]:=\fW_n[\fF_n[\xi]]$, $\bfW_n[\Xi]:=\bfW_n[\bfF_n[\Xi]]$. See \cref{chaos_prop_chaos}.
\item
For each $n\in\bN_0$, $\cV_{n+1}(S,T;\bR^{d\times d})$ denotes the space of $(n+1)$-parameters deterministic Volterra kernels. See \cref{detkernel_defi_kernel}.
\item
For each $f:\Delta_{m+1}(S,T)\to\bR^{d_1\times d_2}$ and $g:\Delta_{n+1}(S,T)\to\bR^{d_2\times d_3}$ with $m,n\in\bN_0$, $f\triangleright g:\Delta_{m+n+1}(S,T)\to\bR^{d_1\times d_3}$ denotes the $\triangleright$-product. See \cref{detkernel_defi_triangleright}.
\item
$\cK_\bF(S,T;\bR^{d\times d})$ denotes the space of $\star$-Volterra kernels. See \cref{star_defi_starkernel}.
\item
For each $K\in\cK_\bF(S,T;\bR^{d\times d})$ and $\xi\in L^2_\bF(S,T;\bR^{d\times d_1})$, $K\star\xi\in L^2_\bF(S,T;\bR^{d\times d_1})$ denotes the $\star$-product. See \cref{star_defi_starprod}.
\item
For each $f\in L^2_\cF(\Delta_{n+2}(S,T);\bR^{d_1\times d_2})$ and $g\in L^2_\cF(\Delta_{m+2}(S,T);\bR^{d_2\times d_3})$ with $n,m\in\bN_0$, $f\ast g:\Omega\times\Delta_{n+m+2}(S,T)\to\bR^{d_1\times d_3}$ denote the $\ast$-products. See \cref{ast_defi_astprod}.
\item
$\cJ_\bF(S,T;\bR^{d\times d})$ denotes the space of $\ast$-Volterra kernels. See \cref{ast_defi_astkernel}.
\item
For each $n\in\bN_0$, $\cM_n:L^2_\bF(S,T;\bR^{d_1\times d_2})\to L^2_\bF(\Delta_{n+1}(S,T);\bR^{d_1\times d_2})$ denotes the $n$-th martingale representation operator. See \cref{BSVIE_defi_martop}.
\item
For each $K\in\cK_\bF(S,T;\bR^{d\times d})$ and $\xi\in L^2_\bF(S,T;\bR^{d\times d_1})$, $K\bstar\xi\in L^2_\bF(S,T;\bR^{d\times d_1})$ denotes the backward $\star$-product. See \cref{bstar_defi_bstarprod}.
\item
For each $J\in\cJ_\bF(S,T;\bR^{d\times d})$ and $\xi\in L^2_\bF(S,T;\bR^{d\times d_1})$, $J\bast\xi\in L^2_\bF(S,T;\bR^{d\times d_1})$ denotes the backward $\ast$-product. See \cref{bast_defi_bastprod}.
\end{itemize}

In the following sections, we fix $0\leq S<T<\infty$ and $d\in\bN$.


\section{Preliminaries}\label{section_Pre}

\subsection{The Wiener--It\^{o} chaos expansion}

For each $\xi\in L^2_\bF(\Delta_{n+1}(S,T);\bR^{d_1\times d_2})$ with $d_1,d_2\in\bN$ and $n\in\bN$, the iterated stochastic integral
\begin{equation*}
	\int^t_S\int^{t_1}_S\mathalpha{\cdots}\int^{t_{n-1}}_S\xi(t,t_1,\mathalpha{\dots},t_n)\rd W(t_n)\mathalpha{\cdots}\rd W(t_1),\ t\in(S,T),
\end{equation*}
is well-defined as an element of $L^2_\bF(S,T;\bR^{d_1\times d_2})$. Furthermore, the following isometry holds:
\begin{equation*}
	\bE\Bigl[\int^T_S\Bigl|\int^t_S\int^{t_1}_S\mathalpha{\cdots}\int^{t_{n-1}}_S\xi(t,t_1,\mathalpha{\dots},t_n)\rd W(t_n)\mathalpha{\cdots}\rd W(t_1)\Bigr|^2\rd t\Bigr]^{1/2}=\|\xi\|_{L^2_\bF(\Delta_{n+1}(S,T))}.
\end{equation*}
From the Winer--It\^{o} chaos expansion in It\^{o}~\cite{It51}, we see that each adapted process $\xi\in L^2_\bF(S,T;\bR^{d_1\times d_2})$ has an orthogonal expansion in terms of iterated stochastic integrals with deterministic integrands. We introduce the following notation.


\begin{defi}\label{chaos_defi_iterated}
Let $d_1,d_2\in\bN$.
\begin{itemize}
\item[(i)]
We define operators $\fW_n:L^2(\Delta_{n+1}(S,T);\bR^{d_1\times d_2})\to L^2_\bF(S,T;\bR^{d_1\times d_2})$, $n\in\bN_0$, by
\begin{equation*}
	\fW_0[f_0](t):=f_0(t),\ t\in(S,T),
\end{equation*}
for each $f_0\in L^2(S,T;\bR^{d_1\times d_2})$, and
\begin{equation*}
	\fW_n[f_n](t):=\int^t_S\int^{t_1}_S\mathalpha{\cdots}\int^{t_{n-1}}_Sf_n(t,t_1,\mathalpha{\dots},t_n)\rd W(t_n)\mathalpha{\cdots}\rd W(t_1),\ t\in(S,T),
\end{equation*}
for each $f_n\in L^2(\Delta_{n+1}(S,T);\bR^{d_1\times d_2})$ and $n\in\bN$.
\item[(ii)]
We define operators $\bfW_n:L^2(\Delta_{n+2}(S,T);\bR^{d_1\times d_2})\to L^2_{\bF,\ast}(\Delta_2(S,T);\bR^{d_1\times d_2})$, $n\in\bN_0$, by
\begin{equation*}
	\bfW_0[g_0](t,s):=g_0(t,s),\ (t,s)\in\Delta_2(S,T),
\end{equation*}
for each $g_0\in L^2(\Delta_2(S,T);\bR^{d_1\times d_2})$, and
\begin{equation*}
	\bfW_n[g_n](t,s):=\int^t_s\int^{t_1}_s\mathalpha{\cdots}\int^{t_{n-1}}_sg_n(t,t_1,\mathalpha{\dots},t_n,s)\rd W(t_n)\mathalpha{\cdots}\rd W(t_1),\ (t,s)\in\Delta_2(S,T),
\end{equation*}
for each $g_n\in L^2(\Delta_{n+2}(S,T);\bR^{d_1\times d_2})$ and $n\in\bN$.
\end{itemize}
\end{defi}

In the present framework, the Wiener--It\^{o} chaos expansion theorem in It\^{o}~\cite{It51} can be stated as follows.


\begin{prop}[The Wiener--It\^{o} chaos expansion for stochastic processes]\label{chaos_prop_chaos}
Let $d_1,d_2\in\bN$.
\begin{itemize}
\item[(i)]
The space $L^2_\bF(S,T;\bR^{d_1\times d_2})$ is decomposed into the infinite orthogonal sum:
\begin{equation*}
	L^2_\bF(S,T;\bR^{d_1\times d_2})=\bigoplus^\infty_{n=0}\fW_n[L^2(\Delta_{n+1}(S,T);\bR^{d_1\times d_2})].
\end{equation*}
More precisely, for any $\xi\in L^2_\bF(S,T;\bR^{d_1\times d_2})$, there exists a unique sequence $\{\fF_n[\xi]\}_{n\in\bN_0}$ with $\fF_n[\xi]\in L^2(\Delta_{n+1}[S,T];\bR^{d_1\times d_2})$ for each $n\in\bN_0$ such that
\begin{equation*}
	\sum^\infty_{n=0}\|\fF_n[\xi]\|^2_{L^2(\Delta_{n+1}(S,T))}<\infty
\end{equation*}
and
\begin{align*}
	\xi=\sum^\infty_{n=0}\fW_n[\fF_n[\xi]],
\end{align*}
where the infinite sum in the right-hand side converges in $L^2_\bF(S,T;\bR^{d_1\times d_2})$. Furthermore, the following isometry holds:
\begin{equation*}
	\|\xi\|^2_{L^2_\bF(S,T)}=\sum^\infty_{n=0}\|\fF_n[\xi]\|^2_{L^2(\Delta_{n+1}(S,T))}.
\end{equation*}
\item[(ii)]
The space $L^2_{\bF,\ast}(\Delta_2(S,T);\bR^{d_1\times d_2})$ is decomposed into the infinite orthogonal sum:
\begin{equation*}
	L^2_{\bF,\ast}(\Delta_2(S,T);\bR^{d_1\times d_2})=\bigoplus^\infty_{n=0}\bfW_n[L^2(\Delta_{n+2}(S,T);\bR^{d_1\times d_2})].
\end{equation*}
More precisely, for any $\Xi\in L^2_{\bF,\ast}(\Delta_2(S,T);\bR^{d_1\times d_2})$, there exists a unique sequence $\{\bfF_n[\Xi]\}_{n\in\bN_0}$ with $\bfF_n[\Xi]\in L^2(\Delta_{n+2}[S,T];\bR^{d_1\times d_2})$ for each $n\in\bN_0$ such that
\begin{equation*}
	\sum^\infty_{n=0}\|\bfF_n[\Xi]\|^2_{L^2(\Delta_{n+2}(S,T))}<\infty
\end{equation*}
and
\begin{align*}
	\Xi=\sum^\infty_{n=0}\bfW_n[\bfF_n[\Xi]],
\end{align*}
where the infinite sum in the right-hand side converges in $L^2_{\bF,\ast}(\Delta_2(S,T);\bR^{d_1\times d_2})$. Furthermore, the following isometry holds:
\begin{equation*}
	\|\Xi\|^2_{L^2_{\bF,\ast}(\Delta_2(S,T))}=\sum^\infty_{n=0}\|\bfF_n[\Xi]\|^2_{L^2(\Delta_{n+2}(S,T))}.
\end{equation*}
\end{itemize}
\end{prop}

We sometimes write
\begin{equation*}
	\fW_n[\xi]:=\fW_n[\fF_n[\xi]],\ \bfW_n[\Xi]:=\bfW_n[\bfF_n[\Xi]],
\end{equation*}
which are the projections of $\xi\in L^2_\bF(S,T;\bR^{d_1\times d_2})$ and $\Xi\in L^2_{\bF,\ast}(\Delta_2(S,T);\bR^{d_1\times d_2})$ on the $n$-th chaos for each $n\in\bN_0$, respectively.

\subsection{Deterministic Volterra kernels}

We introduce the spaces of deterministic Volterra kernels and investigate their fundamental properties.


\begin{defi}\label{detkernel_defi_kernel}
Define the spaces $\cV_{n+1}(S,T;\bR^{d\times d})$, $n\in\bN_0$, of \emph{$(n+1)$-parameters deterministic Volterra kernels} by
\begin{equation*}
	\cV_{n+1}(S,T;\bR^{d\times d}):=\{k:\Delta_{n+1}(S,T)\to\bR^{d\times d}\,|\,\text{$k$ is measurable and}\ \knorm k\knorm_{\cV_{n+1}(S,T)}<\infty\},
\end{equation*}
where $\knorm k\knorm_{\cV_1(S,T)}:=\underset{t\in(S,T)}{\mathrm{ess\,sup}}|k(t)|_\op$ for $n=0$, and
\begin{equation*}
	\knorm k\knorm_{\cV_{n+1}(S,T)}:=\underset{t\in(S,T)}{\mathrm{ess\,sup}}\Bigl(\int_{\Delta_n(t,T)}|k(s,t)|^2_\op\rd s\Bigr)^{1/2}
\end{equation*}
for $n\in\bN$.
Here and elsewhere, for each matrix $A\in\bR^{d\times d}$, $|A|_\op$ denotes the operator norm of $A$ as an linear operator on $\bR^d$. (Recall that $|A|:=\sqrt{\mathrm{tr}(AA^\top)}$ denotes the Frobenius norm of $A$.)
\end{defi}

The following lemma can be easily proved by the definition. We omit the proof.
	

\begin{lemm}\label{detkernel_lemm_kernel}
\begin{itemize}
\item[(i)]
For each $n\in\bN_0$, $(\cV_{n+1}(S,T;\bR^{d\times d}),\|\cdot\|_{\cV_{n+1}(S,T)})$ is a Banach space. Furthermore, for each $k\in\cV_{n+1}(S,T;\bR^{d\times d})$, it holds that $\|k\|_{L^2(\Delta_{n+1}(S,T))}\leq\sqrt{d(T-S)}\knorm k\knorm_{\cV_{n+1}(S,T)}$. In particular, $\cV_{n+1}(S,T;\bR^{d\times d})\subset L^2(\Delta_{n+1}(S,T);\bR^{d\times d})$.
\item[(ii)]
Let $k:\Delta_{n+1}(S,T)\to\bR^{d\times d}$ be a measurable map with $n\in\bN$. Assume that there exists a function $a\in L^2((0,T-S)^n;\bR)$ such that
\begin{equation*}
	|k(t_0,t_1,t_2,\mathalpha{\dots},t_n)|_\op\leq a(t_0-t_1,t_1-t_2,\mathalpha{\dots},t_{n-1}-t_n)
\end{equation*}
for a.e.\ $(t_0,t_1,t_2,\mathalpha{\dots},t_n)\in\Delta_{n+1}(S,T)$. Then $k$ is in $\cV_{n+1}(S,T;\bR^{d\times d})$, and it holds that $\knorm k\knorm_{\cV_{n+1}(S,T)}\leq\|a\|_{L^2((0,T-S)^n)}$.
\end{itemize}
\end{lemm}


\begin{rem}
From the above lemma, we see that $\cV_{n+1}(S,T;\bR^{d\times d})$ includes ($(n+1)$-parameters) singular kernels. For example, the fractional kernel $k(t,s)=\frac{1}{\Gamma(\alpha)}(t-s)^{\alpha-1}$ with $\alpha\in(\frac{1}{2},1]$ is in $\cV_2(S,T;\bR)$.
\end{rem}

We use the following notation frequently in this paper.

\begin{defi}\label{detkernel_defi_triangleright}
For each $f:\Delta_{m+1}(S,T)\to\bR^{d_1\times d_2}$ and $g:\Delta_{n+1}(S,T)\to\bR^{d_2\times d_3}$ with $m,n\in\bN$ and $d_1,d_2,d_3\in\bN$, we define the \emph{$\triangleright$-product} $f\triangleright g:\Delta_{m+n+1}(S,T)\to\bR^{d_1\times d_3}$ by
\begin{equation*}
	(f\triangleright g)(t_0,t_1,\mathalpha{\dots},t_{m+n}):=f(t_0,t_1,\mathalpha{\dots},t_m)g(t_m,t_{m+1},\mathalpha{\dots},t_{m+n})
\end{equation*}
for $(t_0,t_1,\mathalpha{\dots},t_{m+n})\in\Delta_{m+n+1}(S,T)$. When $m=0$, $f\triangleright g:\Delta_{n+1}(S,T)\to\bR^{d_1\times d_3}$ is defined by
\begin{equation*}
	(f\triangleright g)(t_0,t_1,\mathalpha{\dots},t_n):=f(t_0)g(t_0,t_1,\mathalpha{\dots},t_n)
\end{equation*}
for $(t_0,t_1,\mathalpha{\dots},t_n)\in\Delta_{n+1}(S,T)$. When $n=0$, $f\triangleright g:\Delta_{m+1}(S,T)\to\bR^{d_1\times d_3}$ is defined by
\begin{equation*}
	(f\triangleright g)(t_0,t_1,\mathalpha{\dots},t_m):=f(t_0,t_1,\mathalpha{\dots},t_m)g(t_m)
\end{equation*}
for $(t_0,t_1,\mathalpha{\dots},t_m)\in\Delta_{m+1}(S,T)$.
\end{defi}
	

\begin{lemm}\label{detkernel_lemm_kernel2}
\begin{itemize}
\item[(i)]
For each $k_1\in\cV_{m+1}(S,T;\bR^{d\times d})$ and $k_2\in\cV_{n+1}(S,T;\bR^{d\times d})$ with $m,n\in\bN_0$, the $\triangleright$-product $k_1\triangleright k_2$ is in $\cV_{m+n+1}(S,T;\bR^{d\times d})$ and satisfies
\begin{equation*}
	\knorm k_1\triangleright k_2\knorm_{\cV_{m+n+1}(S,T)}\leq\knorm k_1\knorm_{\cV_{m+1}(S,T)}\knorm k_2\knorm_{\cV_{n+1}(S,T)}.
\end{equation*}
\item[(ii)]
For each $k\in\cV_{m+1}(S,T;\bR^{d\times d})$ and $f\in L^2(\Delta_{n+1}(S,T);\bR^{d\times d_1})$ with $m,n\in\bN_0$ and $d_1\in\bN$, the $\triangleright$-product $k\triangleright f$ is in $L^2(\Delta_{m+n+1}(S,T);\bR^{d\times d_1})$ and satisfies
\begin{equation*}
	\|k\triangleright f\|_{L^2(\Delta_{m+n+1}(S,T))}\leq\knorm k\knorm_{\cV_{m+1}(S,T)}\|f\|_{L^2(\Delta_{n+1}(S,T))}.
\end{equation*}
\end{itemize}
\end{lemm}


\begin{proof}
\begin{itemize}
\item[(i)]
Noting the inequality $|AB|_\op\leq|A|_\op|B|_\op$ for $A,B\in\bR^{d\times d}$ and using the Fubini theorem, we see that, for a.e.\ $t\in(S,T)$,
\begin{align*}
	\int_{\Delta_{m+n}(t,T)}|(k_1\triangleright k_2)(s,t)|^2_\op\rd s&=\int_{\Delta_n(t,T)}\int_{\Delta_m(\max s_2,T)}|k_1(s_1,\max s_2)k_2(s_2,t)|^2_\op\rd s_1\rd s_2\\
	&\leq\int_{\Delta_n(t,T)}\int_{\Delta_m(\max s_2,T)}|k_1(s_1,\max s_2)|^2_\op|k_2(s_2,t)|^2_\op\rd s_1\rd s_2\\
	&=\int_{\Delta_n(t,T)}\Bigl\{\int_{\Delta_m(\max s_2,T)}|k_1(s_1,\max s_2)|^2_\op\rd s_1\Bigr\}|k_2(s_2,t)|^2_\op\rd s_2\\
	&\leq\knorm k_1\knorm^2_{\cV_{m+1}(S,T)}\int_{\Delta_n(t,T)}|k_2(s_2,t)|^2_\op\rd s_2\\
	&\leq\knorm k_1\knorm^2_{\cV_{m+1}(S,T)}\knorm k_2\knorm^2_{\cV_{n+1}(S,T)}<\infty.
\end{align*}
This implies that $k_1\triangleright k_2\in\cV_{m+n+1}(S,T;\bR^{d\times d})$ and $\knorm k_1\triangleright k_2\knorm_{\cV_{m+n+1}(S,T)}\leq\knorm k_1\knorm_{\cV_{m+1}(S,T)}\knorm k_2\knorm_{\cV_{n+1}(S,T)}$.
\item[(ii)]
Noting the inequality $|AB|\leq|A|_\op|B|$ for $A\in\bR^{d\times d}$ and $B\in\bR^{d\times d_1}$, we can prove the assertion (ii) by the same way as in (i).
\end{itemize}
\end{proof}


\begin{lemm}\label{detkernel_lemm_kernel3}
Let $k_n\in\cV_{n+1}(S,T;\bR^{d\times d})$ for each $n\in\bN_0$. Suppose that $\sum^\infty_{n=0}\knorm k_n\knorm_{\cV_{n+1}(S,T)}<\infty$. Then $\sum^\infty_{n=0}\|k_n\|^2_{L^2(\Delta_{n+1}(S,T))}<\infty$. Consequently, the infinite sum of the iterated stochastic integrals $\sum^\infty_{n=0}\fW_n[k_n]$ converges in $L^2_\bF(S,T;\bR^{d\times d})$.
\end{lemm}


\begin{proof}
By \cref{detkernel_lemm_kernel}, for each $n\in\bN$, $k_n$ is in $L^2(\Delta_{n+1}(S,T);\bR^{d\times d})$, and it holds that $\|k_n\|_{L^2(\Delta_{n+1}(S,T))}\leq\sqrt{d(T-S)}\knorm k_n\knorm_{\cV_{n+1}(S,T)}$. Furthermore, by the assumption, we have $\sup_{n\in\bN_0}\knorm k_n\knorm_{\cV_{n+1}(S,T)}<\infty$. Therefore, we have
\begin{equation*}
	\sum^\infty_{n=0}\|k_n\|^2_{L^2(\Delta_{n+1}(S,T))}\leq d(T-S)\sum^\infty_{n=0}\knorm k_n\knorm^2_{\cV_{n+1}(S,T)}\leq d(T-S)\sup_{n\in\bN_0}\knorm k_n\knorm_{\cV_{n+1}(S,T)}\sum^\infty_{n=0}\knorm k_n\knorm_{\cV_{n+1}(S,T)}<\infty.
\end{equation*}
This completes the proof.
\end{proof}


\section{Stochastic Volterra integral equations}\label{section_SVIE}

In this section, we investigate linear SVIE \eqref{Intro_eq_SVIE}. More generally, we consider SVIEs with infinitely many iterated stochastic integrals of the following form:
\begin{equation}\label{SVIE_eq_GSVIE}
	X(t)=\varphi(t)+\int^t_SJ(t,s)X(s)\rd s+\sum^\infty_{n=1}\int^t_S\int^{t_1}_S\mathalpha{\cdots}\int^{t_{n-1}}_Sk_n(t,t_1,\mathalpha{\dots},t_n)X(t_n)\rd W(t_n)\mathalpha{\cdots}\rd W(t_1),\ t\in(S,T),
\end{equation}
with suitable choices of kernels $J\in L^2_{\bF,\ast}(\Delta_2(S,T);\bR^{d\times d})$ and $k_n\in\cV_{n+1}(S,T;\bR^{d\times d})$, $n\in\bN$. In this paper, we call the above equation a \emph{generalized SVIE}. This class of SVIEs includes fractional SDEs with ``noisy memory'', as demonstrated in the following example.


\begin{exam}\label{SVIE_exam_FSDEnoisy}
Consider the following fractional equation of order $\alpha\in(\frac{1}{2},1]$:
\begin{equation}\label{SVIE_eq_FSDEnoisy}
	\begin{cases}\displaystyle
	\CD X(t)=j(t)X(t)+X_1(t)+X_2(t)+b(t)+\{k(t)X(t)+X_1(t)+X_2(t)+\sigma(t)\}\frac{\mathrm{d}W(t)}{\mathrm{d}t},\ t\in(0,T),\\\displaystyle
	X_1(t)=\int^t_0\ell_1(t,s)X(s)\rd s,\ X_2(t)=\int^t_0\ell_2(t,s)X(s)\rd W(s),\ t\in(0,T),\\\displaystyle
	X(0)=x_0,
	\end{cases}
\end{equation}
where $j(t),k(t),\ell_1(t,s),\ell_2(t,s)$ are $\bR^{d\times d}$-valued, deterministic and bounded coefficients, $b,\sigma\in L^2_\bF(0,T;\bR^d)$, and $x_0\in\bR^d$. This is a fractional SDE which has a delay $X_1(t)=\int^t_0\ell_1(t,s)X(s)\rd s$ and a ``noisy memory'' $X_2(t)=\int^t_0\ell_2(t,s)X(s)\rd W(s)$. The term noisy memory was first introduced by Dahl et al~\cite{DaMoOkRo16} in the classical SDEs framework. By the definition, an adapted process $X$ is said to be a solution to \eqref{SVIE_eq_FSDEnoisy} if it satisfies the integral equation
\begin{align*}
	&X(t)=x_0+\frac{1}{\Gamma(\alpha)}\int^t_0(t-s)^{\alpha-1}\Bigl\{j(s)X(s)+\int^s_0\ell_1(s,r)X(r)\rd r+\int^s_0\ell_2(s,r)X(r)\rd W(r)+b(s)\}\rd s\\
	&\hspace{1cm}+\frac{1}{\Gamma(\alpha)}\int^t_0(t-s)^{\alpha-1}\Bigl\{k(s)X(s)+\int^s_0\ell_1(s,r)X(r)\rd r+\int^s_0\ell_2(s,r)X(r)\rd W(r)+\sigma(s)\Bigr\}\rd W(s)
\end{align*}
for $t\in(0,T)$. Applying the stochastic Fubini theorem, we see that the above integral equation becomes
\begin{align*}
	&X(t)=x_0+\frac{1}{\Gamma(\alpha)}\int^t_0(t-s)^{\alpha-1}b(s)\rd s+\frac{1}{\Gamma(\alpha)}\int^t_0(t-s)^{\alpha-1}\sigma(s)\rd W(s)\\
	&\hspace{1cm}+\frac{1}{\Gamma(\alpha)}\int^t_0\Bigl\{(t-s)^{\alpha-1}j(s)+\int^t_s(t-r)^{\alpha-1}\ell_1(r,s)\rd r+\int^t_s(t-r)^{\alpha-1}\ell_1(r,s)\rd W(r)\Bigr\}X(s)\rd s\\
	&\hspace{1cm}+\frac{1}{\Gamma(\alpha)}\int^t_0\Bigl\{(t-t_1)^{\alpha-1}k(t_1)+\int^t_{t_1}(t-s)^{\alpha-1}\ell_2(s,t_1)\rd s\Bigr\}X(t_1)\rd W(t_1)\\
	&\hspace{1cm}+\frac{1}{\Gamma(\alpha)}\int^t_0\int^{t_1}_0(t-t_1)^{\alpha-1}\ell_2(t_1,t_2)X(t_2)\rd W(t_2)\rd W(t_1),\ t\in(0,T).
\end{align*}
Thus, the fractional SDE with noisy memory \eqref{SVIE_eq_FSDEnoisy} can be seen as a generalized SVIE~\eqref{SVIE_eq_GSVIE} with the free term
\begin{equation*}
	\varphi(t)=x_0+\frac{1}{\Gamma(\alpha)}\int^t_0(t-s)^{\alpha-1}b(s)\rd s+\frac{1}{\Gamma(\alpha)}\int^t_0(t-s)^{\alpha-1}\sigma(s)\rd W(s)
\end{equation*}
and the kernels
\begin{gather*}
	J(t,s)=\frac{1}{\Gamma(\alpha)}\Bigl\{(t-s)^{\alpha-1}j(s)+\int^t_s(t-r)^{\alpha-1}\ell_1(r,s)\rd r+\int^t_s(t-r)^{\alpha-1}\ell_1(r,s)\rd W(r)\Bigr\},\\
	k_1(t,t_1)=\frac{1}{\Gamma(\alpha)}\Bigl\{(t-t_1)^{\alpha-1}k(t_1)+\int^t_{t_1}(t-s)^{\alpha-1}\ell_2(s,t_1)\rd s\Bigr\},\\
	k_2(t,t_1,t_2)=\frac{1}{\Gamma(\alpha)}(t-t_1)^{\alpha-1}\ell_2(t_1,t_2),\ k_n=0,\ n\geq3.
\end{gather*}
Note that since $J(t,s)$ is $\cF^s_t$-measurable, $J$ is in $L^2_{\bF,\ast}(\Delta_2(0,T);\bR^{d\times d})$.
\end{exam}

In order to investigate generalized SVIE~\eqref{SVIE_eq_GSVIE}, we write it as an algebraic equation on $L^2_\bF(S,T;\bR^d)$ as follows:
\begin{equation}\label{SVIE_eq_SVIEalg}
	X=\varphi+J\ast X+K\star X,
\end{equation}
where $J\ast X$ corresponds to the convolution in terms of the Lebesgue integral, and $K\star X$ with $K=\sum^\infty_{n=1}\fW_n[k_n]$ corresponds to the infinite sum of the stochastic convolutions with respect to the iterated stochastic integrals.


\subsection{$\star$-product}\label{subsection_star}

In this subsection, we introduce a suitable class of $K$ and define the $\star$-product $K\star X$ appearing in equation \eqref{SVIE_eq_SVIEalg}.


\begin{defi}\label{star_defi_starkernel}
We define the space $\cK_\bF(S,T;\bR^{d\times d})$ of \emph{$\star$-Volterra kernels} by
\begin{equation*}
	\cK_\bF(S,T;\bR^{d\times d}):=\{K\in L^2_\bF(S,T;\bR^{d\times d})\,|\,\fF_n[K]\in\cV_{n+1}(S,T;\bR^{d\times d}),\ \forall\,n\in\bN_0,\ \knorm K\knorm_{\cK_\bF(S,T)}<\infty\},
\end{equation*}
where
\begin{equation*}
	\knorm K\knorm_{\cK_\bF(S,T)}:=\sum^\infty_{n=0}\knorm\fF_n[K]\knorm_{\cV_{n+1}(S,T)}.
\end{equation*}
\end{defi}


\begin{rem}
We emphasize that each $\star$-Volterra kernel $K\in\cK_\bF(S,T;\bR^{d\times d})$ is an $\bF^S$-adapted and square-integrable stochastic process (with one time-parameter), see \cref{detkernel_lemm_kernel3}. Intuitively speaking, the Volterra structure corresponds to the adaptedness of the process $K$. Since each $(\cV_{n+1}(S,T;\bR^{d\times d}),\knorm\cdot\knorm_{\cV_{n+1}(S,T)})$ is a Banach space, we see that $(\cK_\bF(S,T;\bR^{d\times d}),\knorm\cdot\knorm_{\cK_\bF(S,T)})$ is a Banach space.
\end{rem}


\begin{defi}\label{star_defi_starprod}
For each $\star$-Volterra kernel $K\in\cK_\bF(S,T;\bR^{d\times d})$ and $\xi\in L^2_\bF(S,T;\bR^{d\times d_1})$ with $d_1\in\bN$, we define the \emph{$\star$-product} $K\star\xi\in L^2_\bF(S,T;\bR^{d\times d_1})$ by the following Wiener--It\^{o} chaos expansion:
\begin{equation*}
	\fF_n[K\star\xi]:=\sum^n_{k=0}\fF_{n-k}[K]\triangleright\fF_k[\xi],\ n\in\bN_0.
\end{equation*}
\end{defi}

The following lemma shows that the $\star$-product is well-defined and makes the space $\cK_\bF(S,T;\bR^{d\times d})$ a Banach algebra.


\begin{prop}\label{star_prop_Banachalg}
$(\cK_\bF(S,T;\bR^{d\times d}),\knorm\cdot\knorm_{\cK_\bF(S,T)},\star)$ is a (real) unitary Banach algebra with the unit $I_d$, where $I_d\in\bR^{d\times d}$ denotes the identity matrix which can be viewed as an element of $\cK_\bF(S,T;\bR^{d\times d})$. Furthermore, $(L^2_\bF(S,T;\bR^d),\|\cdot\|_{L^2_\bF(S,T)})$ is a left Banach module over $\cK_\bF(S,T;\bR^{d\times d})$ with respect to the $\star$-product. In other words, $(\cK_\bF(S,T;\bR^{d\times d}),\knorm\cdot\knorm_{\cK_\bF(S,T)})$ is a Banach space, and for each $K,K_1,K_2,K_3\in\cK_\bF(S,T;\bR^{d\times d})$, $\xi,\xi_1,\xi_2\in L^2_\bF(S,T;\bR^d)$ and $\alpha\in\bR$, the following hold:
\begin{gather*}
	K_1\star K_2\in\cK_\bF(S,T;\bR^{d\times d})\ \text{and}\ \knorm K_1\star K_2\knorm_{\cK_\bF(S,T)}\leq\knorm K_1\knorm_{\cK_\bF(S,T)}\knorm K_2\knorm_{\cK_\bF(S,T)},\\
	(K_1\star K_2)\star K_3=K_1\star (K_2\star K_3),\\
	K_1\star(K_2+K_3)=K_1\star K_2+K_1\star K_3,\\
	(K_1+K_2)\star K_3=K_1\star K_3+K_2\star K_3,\\
	\alpha(K_1\star K_2)=(\alpha K_1)\star K_2=K_1\star(\alpha K_2),\\
	K\star\xi\in L^2_\bF(S,T;\bR^d)\ \text{and}\ \|K\star\xi\|_{L^2_\bF(S,T)}\leq\knorm K\knorm_{\cK_\bF(S,T)}\|\xi\|_{L^2_\bF(S,T)},\\
	(K_1\star K_2)\star\xi=K_1\star(K_2\star\xi),\\
	K\star(\xi_1+\xi_2)=K\star\xi_1+K\star\xi_2,\\
	(K_1+K_2)\star\xi=K_1\star\xi+K_2\star\xi,\\
	\alpha(K\star\xi)=(\alpha K)\star\xi=K\star(\alpha\xi),\\
	I_d\in\cK_\bF(S,T;\bR^{d\times d}),\ \knorm I_d\knorm_{\cK_\bF(S,T)}=1,\ K\star I_d=I_d\star K=K,\ \text{and}\ I_d\star\xi=\xi.
\end{gather*}
\end{prop}


\begin{proof}
Let $K,K_1,K_2,K_3\in\cK_\bF(S,T;\bR^{d\times d})$, $\xi,\xi_1,\xi_2\in L^2_\bF(S,T;\bR^d)$ and $\alpha\in\bR$. By the triangle inequality, \cref{detkernel_lemm_kernel2} and Young's convolution inequality, we have
\begin{align*}
	\sum^\infty_{n=0}\Bknorm\sum^n_{k=0}\fF_{n-k}[K_1]\triangleright\fF_k[K_2]\Bknorm_{\cV_{n+1}(S,T)}&\leq\sum^\infty_{n=0}\sum^n_{k=0}\knorm\fF_{n-k}[K_1]\triangleright\fF_k[K_2]\knorm_{\cV_{n+1}(S,T)}\\
	&\leq\sum^\infty_{n=0}\sum^n_{k=0}\knorm\fF_{n-k}[K_1]\knorm_{\cV_{n-k+1}(S,T)}\knorm\fF_k[K_2]\knorm_{\cV_{k+1}(S,T)}\\
	&\leq\sum^\infty_{n=0}\knorm\fF_n[K_1]\knorm_{\cV_{n+1}(S,T)}\sum^\infty_{n=0}\knorm\fF_n[K_2]\knorm_{\cV_{n+1}(S,T)}\\
	&=\knorm K_1\knorm_{\cK_\bF(S,T)}\knorm K_2\knorm_{\cK_\bF(S,T)}<\infty.
\end{align*}
Therefore, $K_1\star K_2\in\cK_\bF(S,T;\bR^{d\times d})$ is well-defined and satisfies
\begin{equation*}
	\knorm K_1\star K_2\knorm_{\cK_\bF(S,T)}\leq\knorm K_1\knorm_{\cK_\bF(S,T)}\knorm K_2\knorm_{\cK_\bF(S,T)}.
\end{equation*}
Similarly, we have
\begin{align*}
	\sum^\infty_{n=0}\Bigl\|\sum^n_{k=0}\fF_{n-k}[K]\triangleright\fF_k[\xi]\Bigr\|^2_{L^2(\Delta_{n+1}(S,T))}&\leq\sum^\infty_{n=0}\Bigl(\sum^n_{k=0}\|\fF_{n-k}[K]\triangleright\fF_k[\xi]\|_{L^2(\Delta_{n+1}(S,T))}\Bigr)^2\\
	&\leq\sum^\infty_{n=0}\Bigl(\sum^n_{k=0}\knorm\fF_{n-k}[K]\knorm_{\cV_{n-k+1}(S,T)}\|\fF_k[\xi]\|_{L^2(\Delta_{k+1}(S,T))}\Bigr)^2\\
	&\leq\Bigl(\sum^\infty_{n=0}\knorm\fF_n[K]\knorm_{\cV_{n+1}(S,T)}\Bigr)^2\sum^\infty_{n=0}\|\fF_n[\xi]\|^2_{L^2(\Delta_{n+1}(S,T))}\\
	&=\knorm K\knorm^2_{\cK_\bF(S,T)}\|\xi\|^2_{L^2_\bF(S,T)}<\infty.
\end{align*}
Therefore, $K\star\xi\in L^2_\bF(S,T;\bR^d)$ is well-defined and satisfies
\begin{equation*}
	\|K\star\xi\|_{L^2_\bF(S,T)}\leq\knorm K\knorm_{\cK_\bF(S,T)}\|\xi\|_{L^2_\bF(S,T)}. 
\end{equation*}
Furthermore, for each $n\in\bN_0$,
\begin{align*}
	\fF_n[(K_1\star K_2)\star K_3]&=\sum^n_{k=0}\fF_{n-k}[K_1\star K_2]\triangleright\fF_k[K_3]\\
	&=\sum^n_{k=0}\Bigl(\sum^{n-k}_{\ell=0}\fF_{n-k-\ell}[K_1]\triangleright\fF_\ell[K_2]\Bigr)\triangleright\fF_k[K_3]\\
	&=\sum^n_{k=0}\sum^n_{\ell=k}\fF_{n-\ell}[K_1]\triangleright(\fF_{\ell-k}[K_2]\triangleright\fF_k[K_3])\\
	&=\sum^n_{\ell=0}\fF_{n-\ell}[K_1]\triangleright\Bigl(\sum^\ell_{k=0}\fF_{\ell-k}[K_2]\triangleright\fF_k[K_3]\Bigr)\\
	&=\sum^n_{\ell=0}\fF_{n-\ell}[K_1]\triangleright\fF_\ell[K_2\star K_3]\\
	&=\fF_n[K_1\star(K_2\star K_3)],
\end{align*}
and hence $(K_1\star K_2)\star K_3=K_1\star(K_2\star K_3)$. Similarly, we can show that $(K_1\star K_2)\star\xi=K_1\star(K_2\star\xi)$. The remaining assertions are clear.
\end{proof}

Although the $\star$-product is defined by means of the Wiener--It\^{o} chaos expansion, it can be represented in terms of iterated stochastic integrals.


\begin{prop}\label{star_prop_integral}
For each $\star$-Volterra kernel $K\in\cK_\bF(S,T;\bR^{d\times d})$ and $\xi\in L^2_\bF(S,T;\bR^{d\times d_1})$ with $d_1\in\bN$, it holds that
\begin{equation}\label{star_eq_integral}
	(K\star\xi)(t)=\fF_0[K](t)\xi(t)+\sum^\infty_{n=1}\int^t_S\int^{t_1}_S\mathalpha{\cdots}\int^{t_{n-1}}_S\fF_n[K](t,t_1,\mathalpha{\dots},t_n)\xi(t_n)\rd W(t_n)\mathalpha{\cdots}\rd W(t_1),\ t\in(S,T),
\end{equation}
where the infinite sum in the right-hand side converges in $L^2_\bF(S,T;\bR^{d\times d_1})$.
\end{prop}


\begin{proof}
We define $\Phi^K_0[\xi](t):=\fF_0[K](t)\xi(t)$, $t\in(S,T)$, and
\begin{equation*}
	\Phi^K_n[\xi](t):=\int^t_S\int^{t_1}_S\mathalpha{\cdots}\int^{t_{n-1}}_S\fF_n[K](t,t_1,\mathalpha{\dots},t_n)\xi(t_n)\rd W(t_n)\mathalpha{\cdots}\rd W(t_1),\ t\in(S,T),
\end{equation*}
for each $n\in\bN$. Note that the integrand in the above iterated stochastic integral is in $L^2_\bF(\Delta_{n+1}(S,T);\bR^{d\times d_1})$, and thus $\Phi^K_n[\xi]\in L^2_\bF(S,T;\bR^{d\times d_1})$ is well-defined. By using the isometry of the stochastic integral, similar calculations as in \cref{detkernel_lemm_kernel2} show that $\|\Phi^K_n[\xi]\|_{L^2_\bF(S,T)}\leq\knorm\fF_n[K]\knorm_{\cV_{n+1}(S,T)}\|\xi\|_{L^2_\bF(S,T)}$ for each $n\in\bN_0$. Thus, each $\Phi^K_n$ is a bounded linear operator on $L^2_\bF(S,T;\bR^{d\times d_1})$ with the operator norm $\|\Phi^K_n\|_{\mathrm{op}}\leq\knorm\fF_n[K]\knorm_{\cV_{n+1}(S,T)}$. Since $\sum^\infty_{n=0}\knorm\fF_n[K]\knorm_{\cV_{n+1}(S,T)}=\knorm K\knorm_{\cK_\bF(S,T)}<\infty$, the infinite sum in the right-hand side of \eqref{star_eq_integral} converges in $L^2_\bF(S,T;\bR^{d\times d_1})$. Noting the continuity of the linear operator $\Phi^K_n$ and considering the Wiener--It\^{o} chaos expansion of $\xi\in L^2_\bF(S,T;\bR^{d\times d_1})$, we have
\begin{equation*}
	\Phi^K_n[\xi]=\Phi^K_n\Bigl[\sum^\infty_{m=0}\fW_m[\xi]\Bigr]=\sum^\infty_{m=0}\Phi^K_n[\fW_m[\xi]],\ n\in\bN_0,
\end{equation*}
where $\fW_m[\xi]:=\fW_m[\fF_m[\xi]]$ for each $m\in\bN_0$. Observe that, for each $n,m\in\bN$,
\begin{align*}
	&\Phi^K_n[\fW_m[\xi]](t)\\
	&=\int^t_S\int^{t_1}_S\mathalpha{\cdots}\int^{t_{n-1}}_S\fF_n[K](t,t_1,\mathalpha{\dots},t_n)\\*
	&\hspace{1cm}\times\Bigl(\int^{t_n}_S\int^{t_{n+1}}_S\mathalpha{\cdots}\int^{t_{n+m-1}}_S\fF_m[\xi](t_n,t_{n+1},\mathalpha{\dots},t_{n+m})\rd W(t_{n+m})\mathalpha{\cdots}\rd W(t_{n+1})\Bigr)\rd W(t_n)\mathalpha{\cdots}\rd W(t_1)\\
	&=\int^t_S\int^{t_1}_S\mathalpha{\cdots}\int^{t_{n+m-1}}_S(\fF_n[K]\triangleright\fF_m[\xi])(t,t_1,\mathalpha{\dots},t_{n+m})\rd W(t_{n+m})\mathalpha{\cdots}\rd W(t_1)\\
	&=\fW_{n+m}[\fF_n[K]\triangleright\fF_m[\xi]](t),\ t\in(S,T),
\end{align*}
and hence $\Phi^K_n[\fW_m[\xi]]=\fW_{n+m}[\fF_n[K]\triangleright\fF_m[\xi]]$ for each $n,m\in\bN$. Clearly, these relations hold when $n=0$ or $m=0$. Therefore, the right-hand side of \eqref{star_eq_integral} becomes
\begin{align*}
	\sum^\infty_{n=0}\sum^\infty_{m=0}\Phi^K_n[\fW_m[\xi]]&=\sum^\infty_{n=0}\sum^\infty_{m=0}\fW_{n+m}[\fF_n[K]\triangleright\fF_m[\xi]]\\
	&=\sum^\infty_{n=0}\sum^n_{k=0}\fW_n[\fF_{n-k}[K]\triangleright\fF_k[\xi]]\\
	&=\sum^\infty_{n=0}\fW_n\Bigl[\sum^n_{k=0}\fF_{n-k}[K]\triangleright\fF_k[\xi]\Bigr]\\
	&=\sum^\infty_{n=0}\fW_n[\fF_n[K\star\xi]]\\
	&=K\star\xi.
\end{align*}
This completes the proof.
\end{proof}


\subsection{$\ast$-product}\label{subsection_ast}

Next, we introduce a suitable class of $J$ and define the $\ast$-product $J\ast X$ appearing in equation \eqref{SVIE_eq_SVIEalg}.


\begin{defi}\label{ast_defi_astprod}
For each $j\in L^2_\cF(\Delta_{n+2}(S,T);\bR^{d_1\times d_2})$, $f\in L^2_\cF(S,T;\bR^{d_2\times d_3})$ with $n\in\bN_0$ and $d_1,d_2,d_3\in\bN$, we define the \emph{$\ast$-product} $j\ast f:\Omega\times\Delta_{n+1}(S,T)\to\bR^{d_1\times d_3}$ by
\begin{equation*}
	(j\ast f)(t_0,t_1,\mathalpha{\dots},t_n):=\int^{t_n}_Sj(t_0,t_1,\mathalpha{\dots},t_n,s)f(s)\rd s
\end{equation*}
for $(t_0,t_1,\mathalpha{\dots},t_n)\in\Delta_{n+1}(S,T)$. Also, for each $g\in L^2_\cF(\Delta_{m+2}(S,T);\bR^{d_2\times d_3})$ with $m\in\bN_0$, we define the $\ast$-product $j\ast g:\Omega\times\Delta_{n+m+2}(S,T)\to\bR^{d_1\times d_3}$ by
\begin{equation*}
	(j\ast g)(t_0,t_1,\mathalpha{\dots},t_{n+m+1}):=\int^{t_n}_{t_{n+1}}j(t_0,t_1,\mathalpha{\dots},t_n,s)g(s,t_{n+1},\mathalpha{\dots},t_{n+m+1})\rd s
\end{equation*}
for $(t_0,t_1,\mathalpha{\dots},t_{n+m+1})\in\Delta_{n+m+2}(S,T)$.
\end{defi}

Recall that $L^2_{\bF,\ast}(\Delta_2(S,T);\bR^{d_1\times d_2})$ is the space of $\Xi\in L^2_\cF(\Delta_2(S,T);\bR^{d_1\times d_2})$ with $\Xi(t,s)$ being $\cF^s_t$-measurable for each $(t,s)\in\Delta_2(S,T)$. Also, recall the definitions of $\bfW_n$ and $\bfF_n$ (see \cref{chaos_defi_iterated}).


\begin{lemm}\label{ast_lemm_chaos}
For each $\Xi\in  L^2_{\bF,\ast}(\Delta_2(S,T);\bR^{d_1\times d_2})$ and $\xi\in L^2_\bF(S,T;\bR^{d_2\times d_3})$ with $d_1,d_2,d_3\in\bN$, the $\ast$-product $\Xi\ast\xi$ is in $L^2_\bF(S,T;\bR^{d_1\times d_3})$, and the Wiener--It\^{o} chaos expansion satisfies
\begin{equation*}
	\fF_n[\Xi\ast\xi]=\sum^n_{k=0}\bfF_{n-k}[\Xi]\ast\fF_k[\xi],\ n\in\bN_0.
\end{equation*}
Furthermore, for each $\Xi_1\in L^2_{\bF,\ast}(\Delta_2(S,T);\bR^{d_1\times d_2})$ and $\Xi_2\in L^2_{\bF,\ast}(\Delta_2(S,T);\bR^{d_2\times d_3})$, the $\ast$-product $\Xi_1\ast \Xi_2$ is in $L^2_{\bF,\ast}(\Delta_2(S,T);\bR^{d_1\times d_3})$, and the Wiener--It\^{o} chaos expansion satisfies
\begin{equation*}
	\bfF_n[\Xi_1\ast \Xi_2]=\sum^n_{k=0}\bfF_{n-k}[\Xi_1]\ast\bfF_k[\Xi_2],\ n\in\bN_0.
\end{equation*}
\end{lemm}


\begin{proof}
We prove the first assertion. The second can be proved by the same way. Note that $\Xi(t,s)$ is $\cF^s_t$-measurable, and hence it is independent of $\cF_s$. By using Minkowski's inequality and H\"{o}lder's inequality, we have
\begin{align*}
	\bE\Bigl[\int^T_S\Bigl(\int^t_S|\Xi(t,s)||\xi(s)|\rd s\Bigr)^2\rd t\Bigr]^{1/2}&\leq\int^T_S\bE\Bigl[\int^T_s|\Xi(t,s)|^2|\xi(s)|^2\rd t\Bigr]^{1/2}\rd s\\
	&=\int^T_S\bE\Bigl[\int^T_s|\Xi(t,s)|^2\rd t\,|\xi(s)|^2\Bigr]^{1/2}\rd s\\
	&=\int^T_S\bE\Bigl[\int^T_s|\Xi(t,s)|^2\rd t\Bigr]^{1/2}\bE\bigl[|\xi(s)|^2\bigr]^{1/2}\rd s\\
	&\leq\Bigl(\int^T_S\bE\Bigl[\int^T_s|\Xi(t,s)|^2\rd t\Bigr]\rd s\Bigr)^{1/2}\Bigl(\int^T_S\bE\bigl[|\xi(s)|^2\bigr]\rd s\Bigr)^{1/2}\\
	&<\infty.
\end{align*}
Thus, the $\ast$-product $\Xi\ast\xi\in L^2_\cF(S,T;\bR^{d_1\times d_3})$ is well-defined, and the operations $\Xi\mapsto \Xi\ast\xi$ and $\xi\mapsto \Xi\ast\xi$ are continuous. Noting the adaptedness, we have $\Xi\ast\xi\in L^2_\bF(S,T;\bR^{d_1\times d_3})$. Furthermore, by using the stochastic Fubini's theorem, we have
\begin{align*}
	&(\Xi\ast\xi)(t)=\Bigl\{\Bigl(\sum^\infty_{n=0}\bfW_n[\Xi]\Bigr)\ast\Bigl(\sum^\infty_{m=0}\fW_m[\xi]\Bigr)\Bigr\}(t)=\sum^\infty_{n=0}\sum^n_{k=0}(\bfW_{n-k}[\Xi]\ast\fW_k[\xi])(t)\\
	&=\sum^\infty_{n=0}\sum^n_{k=0}\int^t_S\int^t_s\int^{t_1}_s\mathalpha{\cdots}\int^{t_{n-k-1}}_s\bfF_{n-k}[\Xi](t,t_1,\mathalpha{\dots},t_{n-k},s)\rd W(t_{n-k})\mathalpha{\cdots}\rd W(t_1)\\
	&\hspace{2cm}\times\int^s_S\int^{t_{n-k+1}}_S\mathalpha{\cdots}\int^{t_{n-1}}_S\fF_k[\xi](s,t_{n-k+1},\mathalpha{\dots},t_n)\rd W(t_n)\mathalpha{\cdots}\rd W(t_{n-k+1})\rd s\\
	&=\sum^\infty_{n=0}\sum^n_{k=0}\int^t_S\int^t_s\int^{t_1}_s\mathalpha{\cdots}\int^{t_{n-k-1}}_s\int^s_S\int^{t_{n-k+1}}_S\mathalpha{\cdots}\int^{t_{n-1}}_S\bfF_{n-k}[\Xi](t,t_1,\mathalpha{\dots},t_{n-k},s)\fF_k[\xi](s,t_{n-k+1},\mathalpha{\dots},t_n)\\
	&\hspace{5cm}\rd W(t_n)\mathalpha{\cdots}\rd W(t_{n-k+1})\rd W(t_{n-k})\mathalpha{\cdots}\rd W(t_1)\rd s\\
	&=\sum^\infty_{n=0}\int^t_S\int^{t_1}_S\mathalpha{\cdots}\int^{t_{n-1}}_S\sum^n_{k=0}\int^{t_{n-k}}_{t_{n-k+1}}\bfF_{n-k}[\Xi](t,t_1,\mathalpha{\dots},t_{n-k},s)\fF_k[\xi](s,t_{n-k+1},\mathalpha{\dots},t_n)\rd s\rd W(t_n)\mathalpha{\cdots}\rd W(t_1)\\
	&=\sum^\infty_{n=0}\int^t_S\int^{t_1}_S\mathalpha{\cdots}\int^{t_{n-1}}_S\sum^n_{k=0}(\bfF_{n-k}[\Xi]\ast\fF_k[\xi])(t,t_1,\mathalpha{\dots},t_n)\rd W(t_n)\mathalpha{\cdots}\rd W(t_1),
\end{align*}
which implies that $\fF_n[\Xi\ast\xi]=\sum^n_{k=0}\bfF_{n-k}[\Xi]\ast\fF_k[\xi]$ for any $n\in\bN_0$. This completes the proof.
\end{proof}


\begin{defi}\label{ast_defi_astkernel}
We define the space $\cJ_\bF(S,T;\bR^{d\times d})$ of \emph{$\ast$-Volterra kernels} by
\begin{equation*}
	\cJ_\bF(S,T;\bR^{d\times d}):=\{J\in L^2_{\bF,\ast}(\Delta_2(S,T);\bR^{d\times d})\,|\,\ \knorm J\knorm_{\cJ_\bF(S,T)}<\infty\},
\end{equation*}
where
\begin{equation*}
	\knorm J\knorm_{\cJ_\bF(S,T)}:=\sum^\infty_{n=0}\Bigl(\int_{\Delta_{n+2}(S,T)}|\bfF_n[J](t)|^2_\op\rd t\Bigr)^{1/2}.
\end{equation*}
\end{defi}

The following proposition shows algebraic properties of the space $\cJ_\bF(S,T;\bR^{d\times d})$ and the $\ast$-product.


\begin{prop}\label{ast_prop_Banachalg}
$(\cJ_\bF(S,T;\bR^{d\times d}),\knorm\cdot\knorm_{\cJ_\bF(S,T)},\ast)$ is a (real) Banach algebra without unit. Furthermore, $(L^2_\bF(S,T;\bR^d),\|\cdot\|_{L^2_\bF(S,T)})$ and $(\cK_\bF(S,T;\bR^{d\times d}),\knorm\cdot\knorm_{\cK_\bF(S,T)})$ are left Banach modules over $\cJ_\bF(S,T;\bR^{d\times d})$ with respect to the $\ast$-product. In other words, $(\cJ_\bF(S,T;\bR^{d\times d}),\knorm\cdot\knorm_{\cJ_\bF(S,T)})$ is a Banach space, and for each $J,J_1,J_2,J_3\in\cJ_\bF(S,T;\bR^{d\times d})$, $\xi,\xi_1,\xi_2\in L^2_\bF(S,T;\bR^d)$, $K,K_1,K_2\in\cK_\bF(S,T;\bR^{d\times d})$ and $\alpha\in\bR$, the following hold:
\begin{gather}
	J_1\ast J_2\in\cJ_\bF(S,T;\bR^{d\times d})\ \text{and}\ \knorm J_1\ast J_2\knorm_{\cJ_\bF(S,T)}\leq\knorm J_1\knorm_{\cJ_\bF(S,T)}\knorm J_2\knorm_{\cJ_\bF(S,T)},\label{ast_eq_JJ}\\
	(J_1\ast J_2)\ast J_3=J_1\ast (J_2\ast J_3),\nonumber\\
	J_1\ast(J_2+J_3)=J_1\ast J_2+J_1\ast J_3,\nonumber\\
	(J_1+J_2)\ast J_3=J_1\ast J_3+J_2\ast J_3,\nonumber\\
	\alpha(J_1\ast J_2)=(\alpha J_1)\ast J_2=J_1\ast(\alpha J_2),\nonumber\\
	J\ast\xi\in L^2_\bF(S,T;\bR^d)\ \text{and}\ \|J\ast\xi\|_{L^2_\bF(S,T)}\leq\knorm J\knorm_{\cJ_\bF(S,T)}\|\xi\|_{L^2_\bF(S,T)},\label{ast_eq_Jxi}\\
	(J_1\ast J_2)\ast\xi=J_1\ast(J_2\ast\xi),\nonumber\\
	J\ast(\xi_1+\xi_2)=J\ast\xi_1+J\ast\xi_2,\nonumber\\
	(J_1+J_2)\ast\xi=J_1\ast\xi+J_2\ast\xi,\nonumber\\
	\alpha(J\ast\xi)=(\alpha J)\ast\xi=J\ast(\alpha\xi),\nonumber\\
	J\ast K\in \cK_\bF(S,T;\bR^{d\times d})\ \text{and}\ \knorm J\ast K\knorm_{\cK_\bF(S,T)}\leq\knorm J\knorm_{\cJ_\bF(S,T)}\knorm K\knorm_{\cK_\bF(S,T)},\label{ast_eq_JK}\\
	(J_1\ast J_2)\ast K=J_1\ast(J_2\ast K),\nonumber\\
	J\ast(K_1+K_2)=J\ast K_1+J\ast K_2,\nonumber\\
	(J_1+J_2)\ast K=J_1\ast K+J_2\ast K,\nonumber\\
	\alpha(J\ast K)=(\alpha J)\ast K=J\ast(\alpha K).\nonumber
\end{gather}
\end{prop}


\begin{proof}
We prove \eqref{ast_eq_JJ}, \eqref{ast_eq_Jxi} and \eqref{ast_eq_JK}. In this proof, we use the notation
\begin{equation*}
	\|f\|_{L^2(\Delta_n(S,T)),\op}:=\Bigl(\int_{\Delta_n(S,T)}|f(t)|^2_\op\rd t\Bigr)^{1/2}
\end{equation*}
for each $f\in L^2(\Delta_n(S,T);\bR^{d\times d})$. First, we observe that, for each $j\in L^2(\Delta_{n+2}(S,T);\bR^{d\times d})$ and $f\in L^2(\Delta_{m+2}(S,T);\bR^d)$ with $n,m\in\bN_0$,
\begin{align*}
	&\|j\ast f\|^2_{L^2(\Delta_{n+m+2}(S,T))}\\
	&\leq\int^T_S\int^{t_0}_S\int^{t_1}_S\mathalpha{\cdots}\int^{t_{n+m}}_S\Bigl(\int^{t_n}_{t_{n+1}}|j(t_0,t_1,\mathalpha{\dots},t_n,s)|_\op|f(s,t_{n+1},\mathalpha{\dots},t_{n+m+1})|\rd s\Bigr)^2\rd t_{n+m+1}\mathalpha{\cdots}\rd t_1\rd t_0\\
	&=\int_{\Delta_{m+1}(S,T)}\Bigl\{\int_{\Delta_{n+1}(\max r,T)}\Bigl(\int^{\min t}_{\max r}|j(t,s)|_\op|f(s,r)|\rd s\Bigr)^2\rd t\Bigr\}\rd r\\
	&\leq\int_{\Delta_{m+1}(S,T)}\Bigl\{\int^T_{\max r}\Bigl(\int_{\Delta_{n+1}(s,T)}|j(t,s)|^2_\op|f(s,r)|^2\rd t\Bigr)^{1/2}\rd s\Bigr\}^2\rd r\\
	&=\int_{\Delta_{m+1}(S,T)}\Bigl\{\int^T_{\max r}\Bigl(\int_{\Delta_{n+1}(s,T)}|j(t,s)|^2_\op\rd t\Bigr)^{1/2}|f(s,r)|\rd s\Bigr\}^2\rd r\\
	&\leq\int_{\Delta_{m+1}(S,T)}\Bigl(\int^T_{\max r}\int_{\Delta_{n+1}(s,T)}|j(t,s)|^2_\op\rd t\rd s\Bigr)\,\Bigl(\int^T_{\max r}|f(s,r)|^2\rd s\Bigr)\rd r\\
	&\leq\|j\|^2_{L^2(\Delta_{n+2}(S,T)),\op}\|f\|^2_{L^2(\Delta_{m+2}(S,T))},
\end{align*}
where we used Minkowski's inequality in the second inequality and H\"{o}lder's inequality in the third inequality. This implies that
\begin{equation*}
	\|j\ast f\|_{L^2(\Delta_{n+m+2}(S,T))}\leq\|j\|_{L^2(\Delta_{n+2}(S,T)),\op}\|f\|_{L^2(\Delta_{m+2}(S,T))}.
\end{equation*}
Similarly, for each $f\in L^2(S,T;\bR^d)$, it holds that
\begin{equation*}
	\|j\ast f\|_{L^2(\Delta_{n+1}(S,T))}\leq\|j\|_{L^2(\Delta_{n+2}(S,T)),\op}\|f\|_{L^2(S,T)}.
\end{equation*}
Also, for each $j_1\in L^2(\Delta_{n+2}(S,T);\bR^{d\times d})$ and $j_2\in L^2(\Delta_{m+2}(S,T);\bR^{d\times d})$ with $n,m\in\bN_0$,
\begin{equation*}
	\|j_1\ast j_2\|_{L^2(\Delta_{n+m+2}(S,T)),\op}\leq\|j_1\|_{L^2(\Delta_{n+2}(S,T)),\op}\|j_2\|_{L^2(\Delta_{m+2}(S,T)),\op}.
\end{equation*}
By the same way, we can show that, for each $j\in L^2(\Delta_{n+2}(S,T);\bR^{d\times d})$ and $k\in\cV_{m+1}(S,T;\bR^{d\times d})$ with $n,m\in\bN_0$,
\begin{equation*}
	\knorm j\ast k\knorm_{\cV_{n+m+1}(S,T)}\leq\|j\|_{L^2(\Delta_{n+2}(S,T)),\op}\knorm k\knorm_{\cV_{m+1}(S,T)}.
\end{equation*}

Fix $J_1,J_2\in\cJ_\bF(S,T;\bR^{d\times d})$. Noting \cref{ast_lemm_chaos}, the above observations and Young's convolution inequality yield that
\begin{align*}
	\knorm J_1\ast J_2\knorm_{\cJ_\bF(S,T)}&=\sum^\infty_{n=0}\|\bfF_n[J_1\ast J_2]\|_{L^2(\Delta_{n+2}(S,T)),\op}\\
	&=\sum^\infty_{n=0}\Bigl\|\sum^n_{k=0}\bfF_{n-k}[J_1]\ast\bfF_k[J_2]\Bigr\|_{L^2(\Delta_{n+2}(S,T)),\op}\\
	&\leq\sum^\infty_{n=0}\sum^n_{k=0}\|\bfF_{n-k}[J_1]\ast\bfF_k[J_2]\|_{L^2(\Delta_{n+2}(S,T)),\op}\\
	&\leq\sum^\infty_{n=0}\sum^n_{k=0}\|\bfF_{n-k}[J_1]\|_{L^2(\Delta_{n-k+2}(S,T)),\op}\|\bfF_k[J_2]\|_{L^2(\Delta_{k+2}(S,T)),\op}\\
	&\leq\sum^\infty_{n=0}\|\bfF_n[J_1]\|_{L^2(\Delta_{n+2}(S,T)),\op}\sum^\infty_{n=0}\|\bfF_n[J_2]\|_{L^2(\Delta_{n+2}(S,T)),\op}\\
	&=\knorm J_1\knorm_{\cJ_\bF(S,T)}\knorm J_2\knorm_{\cJ_\bF(S,T)}.
\end{align*}
Thus, \eqref{ast_eq_JJ} holds. For each $J\in\cJ_\bF(S,T;\bR^{d\times d})$ and $\xi\in L^2_\bF(S,T;\bR^d)$, the isometry yields that
\begin{align*}
	\|J\ast\xi\|^2_{L^2_\bF(S,T)}&=\sum^\infty_{n=0}\|\fF_n[J\ast\xi]\|^2_{L^2(\Delta_{n+1}(S,T))}\\
	&=\sum^\infty_{n=0}\Bigl\|\sum^n_{k=0}\bfF_{n-k}[J]\ast\fF_k[\xi]\Bigr\|^2_{L^2(\Delta_{n+1}(S,T))}\\
	&\leq\sum^\infty_{n=0}\Bigl(\sum^n_{k=0}\|\bfF_{n-k}[J]\ast\fF_k[\xi]\|_{L^2(\Delta_{n+1}(S,T))}\Bigr)^2\\
	&\leq\sum^\infty_{n=0}\Bigl(\sum^n_{k=0}\|\bfF_{n-k}[J]\|_{L^2(\Delta_{n-k+2}(S,T)),\op}\|\fF_k[\xi]\|_{L^2(\Delta_{k+1}(S,T))}\Bigr)^2\\
	&\leq\Bigl(\sum^\infty_{n=0}\|\bfF_n[J]\|_{L^2(\Delta_{n+2}(S,T)),\op}\Bigr)^2\sum^\infty_{n=0}\|\fF_n[\xi]\|^2_{L^2(\Delta_{n+1}(S,T))}\\
	&=\knorm J\knorm^2_{\cJ_\bF(S,T)}\|\xi\|^2_{L^2_\bF(S,T)},
\end{align*}
which implies \eqref{ast_eq_Jxi}. Lastly, for each $J\in\cJ_\bF(S,T;\bR^{d\times d})$ and $K\in\cK_\bF(S,T;\bR^{d\times d})$,
\begin{align*}
	\knorm J\ast K\knorm_{\cK_\bF(S,T)}&=\sum^\infty_{n=0}\knorm\fF_n[J\ast K]\knorm_{\cV_{n+1}(S,T)}\\
	&=\sum^\infty_{n=0}\Bknorm\sum^n_{k=0}\bfF_{n-k}[J]\ast\fF_k[K]\Bknorm_{\cV_{n+1}(S,T)}\\
	&\leq\sum^\infty_{n=0}\sum^n_{k=0}\knorm\bfF_{n-k}[J]\ast\bfF_k[K]\knorm_{\cV_{n+1}(S,T)}\\
	&\leq\sum^\infty_{n=0}\sum^n_{k=0}\|\bfF_{n-k}[J]\|_{L^2(\Delta_{n-k+2}(S,T)),\op}\knorm\fF_k[K]\knorm_{\cV_{k+1}(S,T)}\\
	&\leq\sum^\infty_{n=0}\|\bfF_n[J]\|_{L^2(\Delta_{n+2}(S,T)),\op}\sum^\infty_{n=0}\knorm\fF_n[K]\knorm_{\cV_{n+1}(S,T)}\\
	&=\knorm J\knorm_{\cJ_\bF(S,T)}\knorm K\knorm_{\cK_\bF(S,T)}.
\end{align*}
Thus, \eqref{ast_eq_JK} holds.

For each $J_1,J_2,J_3\in\cJ_\bF(S,T;\bR^{d\times d})$, Fubini's theorem yields that
\begin{align*}
	((J_1\ast J_2)\ast J_3)(t,s)&=\int^t_s(J_1\ast J_2)(t,r)J_3(r,s)\rd r\\
	&=\int^t_s\int^t_rJ_1(t,u)J_2(u,r)\rd uJ_3(r,s)\rd r\\
	&=\int^t_sJ_1(t,u)\int^u_sJ_2(u,r)J_3(r,s)\rd r\rd u\\
	&=\int^t_sJ_1(t,u)(J_2\ast J_3)(u,s)\rd u\\
	&=(J_1\ast(J_2\ast J_3))(t,s)
\end{align*}
for $(t,s)\in\Delta_2(S,T)$. Thus, we have $(J_1\ast J_2)\ast J_3=J_1\ast (J_2\ast J_3)$. Similarly, we can show that $(J_1\ast J_2)\ast \xi=J_1\ast(J_2\ast \xi)$ and $(J_1\ast J_2)\ast K=J_1\ast(J_2\ast K)$ for $\xi\in L^2_\bF(S,T;\bR^d)$ and $K\in\cK_\bF(S,T;\bR^{d\times d})$. The other assertions are clear from the definition. We complete the proof.
\end{proof}

Next, we investigate algebraic relationships between the $\ast$- and the $\star$-products. As before, for each $\star$-Volterra kernel $K\in\cK_\bF(S,T;\bR^{d\times d})$ and $\Xi\in L^2_{\bF,\ast}(\Delta_2(S,T);\bR^{d\times d_1})$ with $d_1\in\bN$, we define the $\star$-product $K\star\Xi\in L^2_{\bF,\ast}(\Delta_2(S,T);\bR^{d\times d_1})$ by the Wiener--It\^{o} chaos expansion
\begin{equation*}
	\bfF_n[K\star \Xi]:=\sum^n_{k=0}\fF_{n-k}[K]\triangleright\bfF_k[\Xi],\ n\in\bN_0.
\end{equation*}
By \cref{star_prop_integral}, we see that
\begin{equation*}
	(K\star\Xi)(t,s)=\fF_0[K](t)\Xi(t,s)+\sum^\infty_{n=1}\int^t_s\int^{t_1}_s\mathalpha{\cdots}\int^{t_{n-1}}_s\fF_n[K](t,t_1,\mathalpha{\dots},t_n)\Xi(t_n,s)\rd W(t_n)\mathalpha{\cdots}\rd W(t_1)
\end{equation*}
for $(t,s)\in\Delta_2(S,T)$. Similar to \cref{star_prop_Banachalg}, $(\cJ_\bF(S,T;\bR^{d\times d}),\knorm\cdot\knorm_{\cJ_\bF(S,T)})$ is a left Banach module over $(\cK_\bF(S,T;\bR^{d\times d}),\knorm\cdot\knorm_{\cK_\bF(S,T)},\star)$. Furthermore, the following proposition shows that the order of the $\star$- and the $\ast$-products is exchangeable.


\begin{prop}\label{ast_prop_aststar}
For each $K\in\cK_\bF(S,T;\bR^{d\times d})$, $J\in\cJ_\bF(S,T;\bR^{d\times d})$ and $\xi\in L^2_\bF(S,T;\bR^{d\times d_1})$ with $d_1\in\bN$, it holds that
\begin{equation}\label{ast_eq_KJxi}
	K\star(J\ast\xi)=(K\star J)\ast\xi
\end{equation}
and
\begin{equation}\label{ast_eq_JKxi}
	J\ast(K\star\xi)=(J\ast K)\star\xi.
\end{equation}
\end{prop}


\begin{proof}
By the definition of the $\star$-product and \cref{ast_lemm_chaos}, we have, for each $n\in\bN_0$,
\begin{align*}
	\fF_n[K\star(J\ast\xi)]&=\sum^n_{k=0}\fF_{n-k}[K]\triangleright\fF_k[J\ast\xi]\\
	&=\sum^n_{k=0}\fF_{n-k}[K]\triangleright\Bigl(\sum^k_{\ell=0}\bfF_{k-\ell}[J]\ast\fF_\ell[\xi]\Bigr)\\
	&=\sum^n_{k=0}\sum^k_{\ell=0}(\fF_{n-k}[K]\triangleright\bfF_{k-\ell}[J])\ast\fF_\ell[\xi]\\
	&=\sum^n_{\ell=0}\Bigl(\sum^n_{k=\ell}\fF_{n-k}[K]\triangleright\bfF_{k-\ell}[J]\Bigr)\ast\fF_\ell[\xi]\\
	&=\sum^n_{\ell=0}\Bigl(\sum^{n-\ell}_{k=0}\fF_{n-\ell-k}[K]\triangleright\bfF_k[J]\Bigr)\ast\fF_\ell[\xi]\\
	&=\sum^n_{\ell=0}\bfF_{n-\ell}[K\star J]\ast\fF_\ell[\xi]\\
	&=\fF_n[(K\star J)\ast\xi].
\end{align*}
Thus, \eqref{ast_eq_KJxi} holds. Similarly,
\begin{align*}
	\fF_n[J\ast(K\star\xi)]&=\sum^n_{k=0}\bfF_{n-k}[J]\ast\fF_k[K\star\xi]\\
	&=\sum^n_{k=0}\bfF_{n-k}[J]\ast\Bigl(\sum^k_{\ell=0}\fF_{k-\ell}[K]\triangleright\fF_\ell[\xi]\Bigr)\\
	&=\sum^n_{k=0}\sum^k_{\ell=0}(\bfF_{n-k}[J]\ast\fF_{k-\ell}[K])\triangleright\fF_\ell[\xi]\\
	&=\sum^n_{\ell=0}\Bigl(\sum^n_{k=\ell}\bfF_{n-k}[J]\ast\fF_{k-\ell}[K]\Bigr)\triangleright\fF_\ell[\xi]\\
	&=\sum^n_{\ell=0}\Bigl(\sum^{n-\ell}_{k=0}\bfF_{n-\ell-k}[J]\ast\fF_k[K]\Bigr)\triangleright\fF_\ell[\xi]\\
	&=\sum^n_{\ell=0}\fF_{n-\ell}[J\ast K]\triangleright\fF_\ell[\xi]\\
	&=\fF_n[(J\ast K)\star\xi].
\end{align*}
Thus, \eqref{ast_eq_JKxi} holds.
\end{proof}


\begin{rem}
By the associative properties of the $\star$- and the $\ast$-products, we may use the notations $K_1\star K_2\star K_3$, $J_1\ast J_2\ast J_3$, $J\ast K\star\xi$, $K\star J\ast\xi$, and so on.
\end{rem}


\subsection{$\star$-resolvent, $\ast$-resolvent and $(\ast,\star)$-resolvent}\label{subsection_res}

As in the classical theory on deterministic Volterra equations (see the textbook~\cite{GrLoSt90}), the resolvent with respect to the $\star$- and the $\ast$-products play central roles in the study of linear SVIEs.


\begin{defi}\label{res_defi_res}
\begin{itemize}
\item[(i)]
For each $\star$-Volterra kernel $K\in\cK_\bF(S,T;\bR^{d\times d})$, we say that a $\star$-Volterra kernel $R\in\cK_\bF(S,T;\bR^{d\times d})$ is a \emph{$\star$-resolvent} of $K$ if it satisfies the following resolvent equation:
\begin{equation*}
	R=K+K\star R=K+R\star K.
\end{equation*}
\item[(ii)]
For each $\ast$-Volterra kernel $J\in\cJ_\bF(S,T;\bR^{d\times d})$, we say that a $\ast$-Volterra kernel $Q\in\cJ_\bF(S,T;\bR^{d\times d})$ is a \emph{$\ast$-resolvent} of $J$ if it satisfies the following resolvent equation:
\begin{equation*}
	Q=J+J\ast Q=J+Q\ast J.
\end{equation*}
\end{itemize}
\end{defi}

First, we show that the $\star$- and the $\ast$-resolvents are unique if they exist. More precisely, the following proposition holds.


\begin{prop}\label{res_prop_unique}
\begin{itemize}
\item[(i)]
Let $\star$-Volterra kernels $K,R_1,R_2\in\cK_\bF(S,T;\bR^{d\times d})$ be given. Suppose that $R_1=K+K\star R_1$ and $R_2=K+R_2\star K$. Then $R_1=R_2$. In particular, each $\star$-Volterra kernel has at most one $\star$-resolvent.
\item[(ii)]
Let $\ast$-Volterra kernels $J,Q_1,Q_2\in\cJ_\bF(S,T;\bR^{d\times d})$ be given. Suppose that $Q_1=J+J\ast Q_1$ and $Q_2=J+Q_2\ast J$. Then $Q_1=Q_2$. In particular, each $\ast$-Volterra kernel has at most one $\ast$-resolvent.
\end{itemize}
\end{prop}

The above proposition holds for any algebra (see Lemma~3.3 in \cite{GrLoSt90}). Here, we give a proof for readers' convenience.


\begin{proof}[Proof of \cref{res_prop_unique}]
We prove (i). The assertion (ii) is proved by the same way. By the assumption, we have $R_2\star(R_1-K)=R_2\star K\star R_1=(R_2-K)\star R_1$. Cancelling out $R_2\star R_1$, one gets $R_2\star K=K\star R_1$. Thus, $R_2=K+R_2\star K=K+K\star R_1=R_1$.
\end{proof}


\begin{rem}
Let $J\in L^2(\Delta_2(S,T);\bR^{d\times d})$ be a deterministic $\ast$-Volterra kernel, and suppose that $Q\in\cJ_\bF(S,T;\bR^{d\times d})$ is a $\ast$-resolvent of $J$. Then we see that $\bE[Q]$ is also a $\ast$-resolvent of $J$. By the uniqueness of the $\ast$-resolvent, we have $Q=\bE[Q]$, and thus $Q$ is deterministic.
\end{rem}

Actually, the $\ast$-resolvent always exists.


\begin{prop}\label{res_prop_astresexist}
Every $\ast$-Volterra kernel $J\in\cJ_\bF(S,T;\bR^{d\times d})$ has a unique $\ast$-resolvent $Q\in\cJ_\bF(S,T;\bR^{d\times d})$.
\end{prop}


\begin{proof}
The uniqueness follows from \cref{res_prop_unique}. We show the existence. First, we assume that $\knorm J\knorm_{\cJ_\bF(S,T)}<1$. By \cref{ast_prop_Banachalg}, we have $\knorm J^{\ast n}\knorm_{\cJ_\bF(S,T)}\leq\knorm J\knorm_{\cJ_\bF(S,T)}^n$ for any $n\in\bN$, where $J^{\ast n}\in\cJ_\bF(S,T;\bR^{d\times d})$ denotes the $(n-1)$-fold $\ast$-product of $J$ by itself. By the assumption, $Q:=\sum^\infty_{n=1}J^{\ast n}$ converges in $\cJ_\bF(S,T;\bR^{d\times d})$. Observe that
\begin{equation*}
	J+J\ast Q=J+J\ast\Bigl(\sum^\infty_{n=1}J^{\ast n}\Bigr)=J+\sum^\infty_{n=1}J\ast J^{\ast n}=\sum^\infty_{n=1}J^{\ast n}=Q.
\end{equation*}
Similarly, we have $J+Q\ast J=Q$. Thus, $Q$ is the $\ast$-resolvent of $J$.

Now we consider the general case. For each $\sigma>0$, define $J_\sigma\in\cJ_\bF(S,T;\bR^{d\times d})$ by $J_\sigma(t,s):=e^{-\sigma(t-s)}J(t,s)$ for $(t,s)\in\Delta_2(S,T)$. It is easy to see that
\begin{equation*}
	\bfF_n[J_\sigma](t_0,t_1,\mathalpha{\dots},t_{n+1})=e^{-\sigma(t_0-t_{n+1})}\bfF_n[J](t_0,t_1,\mathalpha{\dots},t_{n+1})
\end{equation*}
for $(t_0,t_1,\mathalpha{\dots},t_{n+1})\in\Delta_{n+2}(S,T)$ and $n\in\bN_0$. Noting the definition of $\knorm\cdot\knorm_{\cJ_\bF(S,T)}$ (see \cref{ast_defi_astkernel}), by the dominated convergence theorem, we see that $\knorm J_\sigma\knorm_{\cJ_\bF(S,T)}$ tends to zero as $\sigma\to\infty$. Let $\sigma>0$ be such that $\knorm J_\sigma\knorm_{\cJ_\bF(S,T)}<1$. Then there exists a unique $\ast$-resolvent $Q_\sigma\in\cJ_\bF(S,T;\bR^{d\times d})$ of $J_\sigma$. For $(t,s)\in\Delta_2(S,T)$, we have
\begin{align*}
	(J_\sigma\ast Q_\sigma)(t,s)&=\int^t_s\!J_\sigma(t,r)Q_\sigma(r,s)\rd r=\int^t_s\!e^{-\sigma(t-r)}J(t,r)Q_\sigma(r,s)\rd r\\
	&=e^{-\sigma(t-s)}\int^t_s\!J(t,r)e^{\sigma(r-s)}Q_\sigma(r,s)\rd r
\end{align*}
and
\begin{align*}
	(Q_\sigma\ast J_\sigma)(t,s)&=\int^t_s\!Q_\sigma(t,r)J_\sigma(r,s)\rd r=\int^t_s\!Q_\sigma(t,r)e^{-\sigma(r-s)}J(r,s)\rd r\\
	&=e^{-\sigma(t-s)}\int^t_s\!e^{\sigma(t-r)}Q_\sigma(t,r)J(r,s)\rd r.
\end{align*}
Therefore, the $\ast$-Volterra kernel $Q\in\cJ_\bF(S,T;\bR^{d\times d})$ defined by $Q(t,s):=e^{\sigma(t-s)}Q_\sigma(t,s)$ for $(t,s)\in\Delta_2(S,T)$ satisfies
\begin{equation*}
	Q(t,s)=J(t,s)+\int^t_sJ(t,r)Q(r,s)\rd r=J(t,s)+\int^t_sQ(t,r)J(r,s)\rd r
\end{equation*}
for $(t,s)\in\Delta_2(S,T)$, and thus it is the $\ast$-resolvent of $J$. This completes the proof.
\end{proof}

The case of the $\star$-resolvent is more difficult. We provide useful sufficient conditions for the existence of the $\star$-resolvent in \cref{section_exist}.

The next lemma provides an explicit expression of the Wiener--It\^{o} chaos expansion of the $\star$-resolvent.


\begin{lemm}\label{res_lemm_starresexplicit}
Suppose that a $\star$-Volterra kernel $K\in\cK_\bF(S,T;\bR^{d\times d})$ satisfies $\fF_0[K]=0$ (which is equivalent to $\bE[K]=0$). If $R\in\cK_\bF(0,T;\bR^{d\times d})$ is a $\star$-resolvent of $K$, then the Wiener--It\^{o} chaos expansion is given by $\fF_0[R]=0$ and
\begin{equation*}
	\fF_n[R]=\sum^n_{i=1}\sum_{\substack{1\leq m_1,m_2,\dots,m_i\leq n\\m_1+m_2+\cdots+m_i=n}}\fF_{m_1}[K]\triangleright\fF_{m_2}[K]\triangleright\cdots\triangleright\fF_{m_i}[K],\ n\in\bN.
\end{equation*}
\end{lemm}


\begin{proof}
By the definition, the Wiener--It\^{o} chaos expansion of $R$ satisfies the following equations:
\begin{equation*}
	\fF_n[R]=\fF_n[K]+\sum^n_{k=0}\fF_{n-k}[K]\triangleright\fF_k[R],\ n\in\bN_0.
\end{equation*}
Since $\fF_0[K]=0$, we have $\fF_0[R]=0$, $\fF_1[R]=\fF_1[K]$, and $\fF_n[R]=\fF_n[K]+\sum^{n-1}_{k=1}\fF_{n-k}[K]\triangleright\fF_k[R]$ for $n\geq2$. By the induction, we get the assertion.
\end{proof}

Next, we introduce the notion of the $(\ast,\star$)-resolvent.


\begin{defi}\label{res_defi_aststarres}
For each $\ast$-Volterra kernel $J\in\cJ_\bF(S,T;\bR^{d\times d})$ and $\star$-Volterra kernel $K\in\cK_\bF(S,T;\bR^{d\times d})$, we say that a pair $(Q,R)\in\cJ_\bF(S,T;\bR^{d\times d})\times\cK_\bF(S,T;\bR^{d\times d})$ is a \emph{$(\ast,\star)$-resolvent} of $(J,K)$ if it satisfies the following resolvent equations:
\begin{equation*}
	\begin{cases}
	Q=J+J\ast Q+K\star Q=J+Q\ast J+R\star J,\\
	R=K+J\ast R+K\star R=K+Q\ast K+R\star K.
	\end{cases}
\end{equation*}
\end{defi}

The following is concerned with the uniqueness of the $(\ast,\star)$-resolvent.


\begin{lemm}\label{ast_prop_unique2}
Let $(J,K)\in\cJ_\bF(S,T;\bR^{d\times d})\times\cK_\bF(S,T;\bR^{d\times d})$ be fixed. Suppose that $(Q_1,R_1),(Q_2,R_2)\in\cJ_\bF(S,T;\bR^{d\times d})\times\cK_\bF(S,T;\bR^{d\times d})$ satisfy
\begin{equation*}
	\begin{cases}
	Q_1=J+J\ast Q_1+K\star Q_1,\\
	R_1=K+J\ast R_1+K\star R_1,
	\end{cases}
	\text{and}\ 
	\begin{cases}
	Q_2=J+Q_2\ast J+R_2\star J,\\
	R_2=K+Q_2\ast K+R_2\star K.
	\end{cases}
\end{equation*}
Then $Q_1=Q_2$ and $R_1=R_2$. In particular, the pair $(J,K)$ has at most one $(\ast,\star)$-resolvent.
\end{lemm}


\begin{proof}
By \cref{ast_prop_Banachalg} and \cref{ast_prop_aststar}, we have
\begin{equation*}
	Q_2\ast(Q_1-J)=Q_2\ast(J\ast Q_1+K\star Q_1)=Q_2\ast J\ast Q_1+Q_2\ast K\star Q_1=(Q_2-J-R_2\star J)\ast Q_1+Q_2\ast K\star Q_1.
\end{equation*}
Cancelling out $Q_2\ast Q_1$, we get
\begin{equation}\label{ast_eq_unique1}
	Q_2\ast J+Q_2\ast K\star Q_1=J\ast Q_1+R_2\star J\ast Q_1.
\end{equation}
Similarly,
\begin{equation*}
	Q_2\ast(R_1-K)=Q_2\ast(J\ast R_1+K\star R_1)=Q_2\ast J\ast R_1+Q_2\ast K\star R_1=(Q_2-J-R_2\star J)\ast R_1+Q_2\ast K\star R_1.
\end{equation*}
Cancelling out $Q_2\ast R_1$, we get
\begin{equation}\label{ast_eq_unique2}
	Q_2\ast K+Q_2\ast K\star R_1=J\ast R_1+R_2\star J\ast R_1.
\end{equation}
Also,
\begin{equation*}
	R_2\star(Q_1-J)=R_2\star(J\ast Q_1+K\star Q_1)=R_2\star J\ast Q_1+R_2\star K\star Q_1=R_2\star J\ast Q_1+(R_2-K-Q_2\ast K)\star Q_1.
\end{equation*}
Cancelling out $R_2\star Q_1$, we get
\begin{equation}\label{ast_eq_unique3}
	R_2\star J+R_2\star J\ast Q_1=K\star Q_1+Q_2\ast K\star Q_1.
\end{equation}
Lastly,
\begin{equation*}
	R_2\star(R_1-K)=R_2\star(J\ast R_1+K\star R_1)=R_2\star J\ast R_1+R_2\star K\star R_1=R_2\star J\ast R_1+(R_2-K-Q_2\ast K)\star R_1.
\end{equation*}
Cancelling out $R_2\star R_1$, we get
\begin{equation}\label{ast_eq_unique4}
	R_2\star K+R_2\star J\ast R_1=K\star R_1+Q_2\ast K\star R_1.
\end{equation}
By \eqref{ast_eq_unique1} and \eqref{ast_eq_unique3}, we see that
\begin{equation*}
	Q_2\ast J+R_2\star J=J\ast Q_1+K\star Q_1.
\end{equation*}
Thus, we obtain
\begin{equation*}
	Q_2=J+Q_2\ast J+R_2\star J=J+J\ast Q_1+K\star Q_1=Q_1.
\end{equation*}
Furthermore, by \eqref{ast_eq_unique2} and \eqref{ast_eq_unique4}, we see that
\begin{equation*}
	Q_2\ast K+R_2\star K=J\ast R_1+K\star R_1.
\end{equation*}
Thus, we obtain
\begin{equation*}
	R_2=K+Q_2\ast K+R_2\star K=K+J\ast R_1+K\star R_1=R_1.
\end{equation*}
This completes the proof.
\end{proof}

We note that the $(\ast,\star)$-resolvent is defined as a pair of kernels. The next proposition shows that the $(\ast,\star)$-resolvent is constructed by the $\ast$-resolvent and the $\star$-resolvent.


\begin{prop}\label{ast_prop_aststarres}
Let a $\ast$-Volterra kernel $J\in\cJ_\bF(S,T;\bR^{d\times d})$ and a $\star$-Volterra kernel $K\in\cK_\bF(S,T;\bR^{d\times d})$ be given.
\begin{itemize}
\item[(i)]
Let $Q_1\in\cJ_\bF(S,T;\bR^{d\times d})$ be the $\ast$-resolvent of $J$, and suppose that $K+Q_1\ast K\in\cK_\bF(S,T;\bR^{d\times d})$ has a $\star$-resolvent $R_1\in\cK_\bF(S,T;\bR^{d\times d})$. Then $(Q,R)\in\cJ_\bF(S,T;\bR^{d\times d})\times\cK_\bF(S,T;\bR^{d\times d})$ defined by $Q:=Q_1+R_1\star Q_1$ and $R:=R_1$ is the $(\ast,\star)$-resolvent of $(J,K)$.
\item[(ii)]
Suppose that $K\in\cK_\bF(S,T;\bR^{d\times d})$ has a $\star$-resolvent $R_2\in\cK_\bF(S,T;\bR^{d\times d})$, and let $Q_2\in\cJ_\bF(S,T;\bR^{d\times d})$ be the $\ast$-resolvent of $J+R_2\star J\in\cJ_\bF(S,T;\bR^{d\times d})$. Then $(Q,R)\in\cJ_\bF(S,T;\bR^{d\times d})\times\cK_\bF(S,T;\bR^{d\times d})$ defined by $Q:=Q_2$ and $R:=R_2+Q_2\ast R_2$ is the $(\ast,\star)$-resolvent of $(J,K)$.
\end{itemize}
\end{prop}


\begin{proof}
\begin{itemize}
\item[(i)]
Note that $Q_1=J+J\ast Q_1$ and $R_1=K+Q_1\ast K+(K+Q_1\ast K)\star R_1$. By \cref{star_prop_Banachalg}, \cref{ast_prop_Banachalg} and \cref{ast_prop_aststar}, we have
\begin{align*}
	J\ast R=J\ast R_1&=J\ast K+J\ast Q_1\ast K+J\ast K\star R_1+J\ast Q_1\ast K\star R_1\\
	&=(J+J\ast Q_1)\ast K+(J+J\ast Q_1)\ast(K\star R_1)\\
	&=Q_1\ast K+Q_1\ast K\star R_1,
\end{align*}
and hence
\begin{align*}
	K+J\ast R+K\star R&=K+Q_1\ast K+Q_1\ast K\star R_1+K\star R_1\\
	&=K+Q_1\ast K+(K+Q_1\ast K)\star R_1\\
	&=R_1=R.
\end{align*}
Consequently, we get
\begin{equation}\label{ast_eq_aststarres1}
	R=K+J\ast R+K\star R.
\end{equation}
Furthermore, we have
\begin{align*}
	J+J\ast Q+K\star Q&=J+J\ast(Q_1+R_1\star Q_1)+K\star(Q_1+R_1\star Q_1)\\
	&=J+J\ast Q_1+(K+J\ast R_1+K\star R_1)\star Q_1\\
	&=Q_1+R_1\star Q_1=Q,
\end{align*}
and thus
\begin{equation}\label{ast_eq_aststarres2}
	Q=J+J\ast Q+K\star Q.
\end{equation}
Similarly, from the equalities $Q_1=J+Q_1\ast J$ and $R_1=K+Q_1\ast K+R_1\star(K+Q_1\ast K)$, we have
\begin{align*}
	K+Q\ast K+R\star K&=K+(Q_1+R_1\star Q_1)\ast K+R_1\star K\\
	&=K+Q_1\ast K+R_1\star(K+Q_1\ast K)\\
	&=R_1=R,
\end{align*}
and thus
\begin{equation}\label{ast_eq_aststarres3}
	R=K+Q\ast K+R\star K.
\end{equation}
Furthermore, we have
\begin{align*}
	J+Q\ast J+R\star J&=J+(Q_1+R_1\star Q_1)\ast J+R_1\star J\\
	&=J+Q_1\ast J+R_1\star(J+Q_1\ast J)\\
	&=Q_1+R_1\star Q_1=Q,
\end{align*}
and thus
\begin{equation}\label{ast_eq_aststarres4}
	Q=J+Q\ast J+R\star J.
\end{equation}
By \eqref{ast_eq_aststarres1}, \eqref{ast_eq_aststarres2}, \eqref{ast_eq_aststarres3} and \eqref{ast_eq_aststarres4}, we see that $(Q,R)$ is the $(\ast,\star)$-resolvent of $(J,K)$.
\item[(ii)]
The claim (ii) can be proved by the same way as (i) by inverting the roles of $J$ and $K$, $Q$ and $R$, and the $\ast$-product and the $\star$-product.
\end{itemize}
\end{proof}


\begin{rem}
From \cref{ast_prop_aststarres} (ii) and \cref{res_prop_astresexist}, if $K\in\cK_\bF(S,T;\bR^{d\times d})$ has a $\star$-resolvent, then $(J,K)$ has a $(\ast,\star)$-resolvent for any $J\in\cJ_\bF(S,T;\bR^{d\times d})$. The existence of the $\star$-resolvent will be discussed in \cref{section_exist}.
\end{rem}


\subsection{A variation of constants formula for generalized SVIEs}\label{subsection_fVCF}

Now we are ready to show the variation of constants formula for generalized SVIEs.


\begin{theo}\label{fVCF_theo_VCF}
Let a $\ast$-Volterra kernel $J\in\cJ_\bF(S,T;\bR^{d\times d})$ and a $\star$-Volterra kernel $K\in\cK_\bF(S,T;\bR^{d\times d})$ be fixed, and suppose that $(J,K)$ has a $(\ast,\star)$-resolvent $(Q,R)\in\cJ_\bF(S,T;\bR^{d\times d})\times\cK_\bF(S,T;\bR^{d\times d})$. Then for any free term $\varphi\in L^2_\bF(S,T;\bR^d)$, the generalized SVIE
\begin{equation*}
	X=\varphi+J\ast X+K\star X
\end{equation*}
has a unique solution $X\in L^2_\bF(S,T;\bR^d)$. This solution is given by the variation of constants formula:
\begin{equation*}
	X=\varphi+Q\ast\varphi+R\star\varphi.
\end{equation*}
\end{theo}


\begin{proof}
Define $X\in L^2_\bF(S,T;\bR^d)$ by $X=\varphi+Q\ast\varphi+R\star\varphi$. By \cref{star_prop_Banachalg}, \cref{ast_prop_Banachalg} and \cref{ast_prop_aststar}, we have
\begin{equation*}
	J\ast X=J\ast(\varphi+Q\ast\varphi+R\star\varphi)=(J+J\ast Q)\ast\varphi+(J\ast R)\star\varphi
\end{equation*}
and
\begin{equation*}
	K\star X=K\star(\varphi+Q\ast\varphi+R\star\varphi)=(K\star Q)\ast\varphi+(K+K\star R)\star\varphi.
\end{equation*}
Combining the above two equalities and using the resolvent equations, we get
\begin{equation*}
	J\ast X+K\star X=(J+J\ast Q+K\star Q)\ast\varphi+(K+J\ast R+K\star R)\star\varphi=Q\ast\varphi+R\star\varphi.
\end{equation*}
This implies that $X=\varphi+J\ast X+K\star X$.

Conversely, if $X\in L^2_\bF(S,T;\bR^d)$ satisfies $X=\varphi+J\ast X+K\star X$, then we have
\begin{equation*}
	Q\ast\varphi=Q\ast(X-J\ast X-K\star X)=(Q-Q\ast J)\ast X-(Q\ast K)\star X
\end{equation*}
and
\begin{equation*}
	R\star\varphi=R\star(X-J\ast X-K\star X)=-(R\star J)\ast X+(R-R\star K)\star X.
\end{equation*}
Combining the above two equalities and using the resolvent equations, we get
\begin{equation*}
	Q\ast\varphi+R\star\varphi=(Q-Q\ast J-R\star J)\ast X+(R-Q\ast K-R\star K)\star X=J\ast X+K\star X.
\end{equation*}
This implies that $X=\varphi+Q\ast\varphi+R\star\varphi$. This completes the proof.
\end{proof}

Noting the stochastic integral representation of the $\star$-product (see \cref{star_prop_integral}), we immediately get the following corollary.


\begin{cor}\label{fVCF_cor_VCF1}
Let $J\in \cJ_\bF(S,T;\bR^{d\times d})$ and $K=\sum^\infty_{n=1}\fW_n[k_n]\in\cK_\bF(S,T;\bR^{d\times d})$ with $k_n\in\cV_{n+1}(S,T;\bR^{d\times d})$, $n\in\bN$, be fixed, and suppose that $(J,K)$ has a $(\ast,\star)$-resolvent $(Q,R)\in\cJ_\bF(S,T;\bR^{d\times d})\times\cK_\bF(S,T;\bR^{d\times d})$. Then for any free term $\varphi\in L^2_\bF(S,T;\bR^d)$, the generalized SVIE
\begin{equation*}
	X(t)=\varphi(t)+\int^t_SJ(t,s)X(s)\rd s+\sum^\infty_{n=1}\int^t_S\int^{t_1}_S\mathalpha{\cdots}\int^{t_{n-1}}_Sk_n(t,t_1,\mathalpha{\dots},t_n)X(t_n)\rd W(t_n)\mathalpha{\cdots}\rd W(t_1),\ t\in(S,T),
\end{equation*}
has a unique solution $X\in L^2_\bF(S,T;\bR^d)$. This solution is given by the variation of constants formula:
\begin{equation*}
	X(t)=\varphi(t)+\int^t_SQ(t,s)\varphi(s)\rd s+\sum^\infty_{n=1}\int^t_S\int^{t_1}_S\mathalpha{\cdots}\int^{t_{n-1}}_Sr_n(t,t_1,\mathalpha{\dots},t_n)\varphi(t_n)\rd W(t_n)\mathalpha{\cdots}\rd W(t_1),\ t\in(S,T),
\end{equation*}
where $r_n:=\fF_n[R]\in\cV_{n+1}(S,T;\bR^{d\times d})$ for $n\in\bN$. Here, the infinite sum in the right-hand side converges in $L^2_\bF(S,T;\bR^d)$.
\end{cor}


\begin{cor}\label{fVCF_cor_VCF2}
Let deterministic kernels $j\in L^2(\Delta_2(S,T);\bR^{d\times d})$ and $k\in\cV_2(S,T;\bR^{d\times d})$ be fixed. Let $q\in L^2(\Delta_2(S,T);\bR^{d\times d})$ be the $\ast$-resolvent of $j$, and assume that $\fW_1[k+q\ast k]$ has a $\star$-resolvent. Then for any free term $\varphi\in\ L^2_\bF(S,T;\bR^d)$, the SVIE
\begin{equation*}
	X(t)=\varphi(t)+\int^t_Sj(t,s)X(s)\rd s+\int^t_Sk(t,s)X(s)\rd W(s),\ t\in(S,T),
\end{equation*}
has a unique solution $X\in L^2_\bF(S,T;\bR^d)$. This solution is given by the variation of constants formula:
\begin{align*}
	X(t)&=\varphi(t)+\int^t_Sq(t,s)\varphi(s)\rd s\\
	&\hspace{0.3cm}+\sum^\infty_{n=1}\int^t_S\int^{t_1}_S\mathalpha{\cdots}\int^{t_{n-1}}_S(k+q\ast k)^{\triangleright n}(t,t_1,\mathalpha{\dots},t_n)\Bigl\{\varphi(t_n)+\int^{t_n}_Sq(t_n,s)\varphi(s)\rd s\Bigr\}\rd W(t_n)\mathalpha{\cdots}\rd W(t_1),\\
	&\hspace{7cm}t\in(S,T).
\end{align*}
Here, $(k+q\ast k)^{\triangleright n}\in\cV_{n+1}(S,T;\bR^{d\times d})$ denotes the $(n-1)$-fold $\triangleright$-product of $k+q\ast k\in\cV_2(S,T;\bR^{d\times d})$ by itself, and the infinite sum in the right-hand side converges in $L^2_\bF(S,T;\bR^d)$.
\end{cor}


\begin{proof}
Note that $\fW_1[k]+q\ast\fW_1[k]=\fW_1[k+q\ast k]$. Thus, by \cref{res_lemm_starresexplicit}, the $\star$-resolvent of $\fW_1[k]+q\ast\fW_1[k]$ is given by $R=\sum^\infty_{n=1}\fW_n[(k+q\ast k)^{\triangleright n}]\in\cK_\bF(S,T;\bR^{d\times d})$. Defining $Q\in\cJ_\bF(S,T;\bR^{d\times d})$ by $Q:=q+R\star q$, by \cref{ast_prop_aststarres}, we see that $(Q,R)$ is the $(\ast,\star)$-resolvent of $(j,\fW_1[k])$. Thus, by \cref{fVCF_theo_VCF}, the SVIE $X=\varphi+j\ast X+\fW_1[k]\star X$ has a unique solution $X\in L^2_\bF(S,T;\bR^d)$, and the solution is given by
\begin{align*}
	X&=\varphi+Q\ast\varphi+R\star\varphi\\
	&=\varphi+(q+R\star q)\ast\varphi+R\star\varphi\\
	&=\varphi+q\ast\varphi+R\star(\varphi+q\ast\varphi).
\end{align*}
Noting \cref{star_prop_integral}, we get the assertion.
\end{proof}


\begin{exam}\label{fVCF_exam_FBS}
Consider the following one-dimensional linear fractional SDE of order $\alpha\in(\frac{1}{2},1]$:
\begin{equation}\label{fVCF_eq_FBS}
	\begin{cases}\displaystyle
	\CD X(t)=\mu X(t)+\sigma X(t)\frac{\mathrm{d}W(t)}{\mathrm{d}t},\ t\in(0,T),\\
	X(0)=x_0,
	\end{cases}
\end{equation}
where $\mu,\sigma,x_0\in\bR$ are constants. When $\alpha=1$, the above equation becomes the Black--Scholes SDE:
\begin{equation*}
	\begin{cases}
	\mathrm{d}X(t)=\mu X(t)\rd t+\sigma X(t)\rd W(t),\ t\in(0,T),\\
	X(0)=x_0,
	\end{cases}
\end{equation*}
and the solution is the geometric Brownian motion $X(t)=x_0\exp((\mu-\sigma^2/2)t+\sigma W(t))$. Thus, \eqref{fVCF_eq_FBS} can be seen as a time-fractional Black--Scholes SDE. Note that \eqref{fVCF_eq_FBS} is equivalent to the following SVIE:
\begin{equation*}
	X(t)=x_0+\int^t_0j(t,s)X(s)\rd s+\int^t_0k(t,s)X(s)\rd W(s),\ t\in(0,T),
\end{equation*}
where $j(t,s)=\frac{\mu}{\Gamma(\alpha)}(t-s)^{\alpha-1}$ and $k(t,s)=\frac{\sigma}{\Gamma(\alpha)}(t-s)^{\alpha-1}$. Here we use some fundamental calculus for the fractional kernel and the Mittag--Leffler function (see Section 1.2 in the textbook \cite{Po99}). The $\ast$-resolvent $q\in L^2(\Delta_2(0,T);\bR)$ of $j$ is given by $q(t,s)=\mu f(t-s)$, where $f(t)=t^{\alpha-1}E_{\alpha,\alpha}(\mu t^\alpha)$, and $E_{\beta,\gamma}(z):=\sum^\infty_{k=0}\frac{z^k}{\Gamma(\beta k+\gamma)}$ is the Mittag--Leffler function. Furthermore, it holds that $(k+q\ast k)(t,s)=\sigma f(t-s)$ for $(t,s)\in\Delta_2(0,T)$, and thus the $\star$-resolvent $R\in\cK_\bF(0,T;\bR)$ of $\fW_1[k+q\ast k]$ is given by
\begin{equation*}
	R(t)=\sum^\infty_{n=1}\sigma^n\int^t_0\int^{t_1}_0\mathalpha{\cdots}\int^{t_{n-1}}_0f(t-t_1)f(t_1-t_2)\mathalpha{\cdots} f(t_{n-1}-t_n)\rd W(t_n)\mathalpha{\cdots}\rd W(t_1),\ t\in(0,T).
\end{equation*}
Here, the existence of the $\star$-resolvent $R$ follows from the results in \cref{section_exist}. Observe that $1+\int^t_0q(t,s)\rd s=1+\mu\int^t_0f(t-s)\rd s=E_{\alpha,1}(\mu t^\alpha)=:E_\alpha(\mu t^\alpha)$. Consequently, by \cref{fVCF_cor_VCF2}, the solution $X\in L^2_\bF(S,T;\bR)$ to the fractional Black--Scholes SDE \eqref{fVCF_eq_FBS} is given by
\begin{equation*}
	X(t)=x_0\Bigl\{E_\alpha(\mu t^\alpha)+\sum^\infty_{n=1}\sigma^n\int^t_0\int^{t_1}_0\mathalpha{\cdots}\int^{t_{n-1}}_0f(t-t_1)f(t_1-t_2)\mathalpha{\cdots} f(t_{n-1}-t_n)E_\alpha(\mu t^\alpha_n)\rd W(t_n)\mathalpha{\cdots}\rd W(t_1)\Bigr\}
\end{equation*}
for $t\in(0,T)$. In particular, if $\mu=0$, we have
\begin{equation*}
	X(t)=x_0\Bigl\{1+\sum^\infty_{n=1}\Bigl(\frac{\sigma}{\Gamma(\alpha)}\Bigr)^n\int^t_0\int^{t_1}_0\mathalpha{\cdots}\int^{t_{n-1}}_0(t-t_1)^{\alpha-1}(t_1-t_2)^{\alpha-1}\mathalpha{\cdots}(t_{n-1}-t_n)^{\alpha-1}\rd W(t_n)\mathalpha{\cdots}\rd W(t_1)\Bigr\}
\end{equation*}
for $t\in(0,T)$.
\end{exam}


\section{Backward stochastic Volterra integral equations}\label{section_BSVIE}

In this section, we investigate linear Type-II BSVIEs. First of all, consider the following Type-II BSVIE:
\begin{equation}\label{BSVIE_eq_classical}
	Y(t)=\tilde{\psi}(t)+\int^T_t\{j(s,t)^\top Y(s)+k(s,t)^\top Z(s,t)\}\rd s-\int^T_tZ(t,s)\rd W(s),\ t\in(S,T),
\end{equation}
where $\tilde{\psi}(\cdot)\in L^2_\cF(S,T;\bR^d)$ is a given free term such that $\tilde{\psi}(t)$ is $\cF^S_T$-measurable for each $t\in(S,T)$, and $j\in L^2(\Delta_2(S,T);\bR^{d\times d})$ and $k\in\cV_2(S,T;\bR^{d\times d})$ are given deterministic kernels. We say that a pair $(Y(\cdot),Z(\cdot,\cdot))$ is an adapted M-solution of BSVIE \eqref{BSVIE_eq_classical} if $Y\in L^2_\bF(S,T;\bR^d)$, $Z:\Omega\times(S,T)^2\to\bR^d$ are measurable, $Z(t,\cdot)$ is $\bF^S$-adapted for a.e.\ $t\in(S,T)$, $\bE[\int^T_S\int^T_S|Z(t,s)|^2\rd t\rd s]<\infty$, and the equation, together with the relation $Y(t)=\bE_S[Y(t)]+\int^t_SZ(t,s)\rd W(s)$ hold for a.e.\ $t\in(S,T)$, a.s.\ (see \cite{Yo08}). Note that the term $Z(t,s)$ for $(t,s)\in\Delta_2(S,T)$ is determined by $Y$ via the martingale representation theorem. Thus, we can write $Z=\cM_1[Y]$ in $\Delta_2(S,T)$ by means of the martingale representation operator $\cM_1:L^2_\bF(S,T;\bR^d)\to L^2_\bF(\Delta_2(S,T);\bR^d)$, which will be defined below. Therefore, taking the conditional expectation $\bE_t$ in \eqref{BSVIE_eq_classical}, we see that the above equation is equivalent to the following equations:
\begin{equation*}
	\begin{cases}
	\displaystyle Y(t)=\bE_t[\tilde{\psi}(t)]+\bE_t\Bigl[\int^T_tj(s,t)^\top Y(s)\rd s\Bigr]+\int^T_tk(s,t)^\top\cM_1[Y](s,t)\rd s,\ t\in(S,T),\\
	\displaystyle \int^T_tZ(t,s)\rd W(s)=\Theta^Y(t)-\bE_t[\Theta^Y(t)],\ t\in(S,T),
	\end{cases}
\end{equation*}
where the process $\Theta^Y(t):=\tilde{\psi}(t)+\int^T_tj(s,t)^\top Y(s)\rd s$, $t\in(S,T)$, is in $L^2_\cF(S,T;\bR^d)$ and depends on $Y$. Note that the above two equations are decoupled in the sense that if $Y\in L^2_\bF(S,T;\bR^d)$ solves the first equation, then the term $Z(t,s)$ for $(t,s)\in(S,T)^2\setminus\Delta_2(S,T)$ is determined by $Y$ as the integrand of the martingale representation theorem for $\Theta^Y(t)$. Consequently, Type-II BSVIE \eqref{BSVIE_eq_classical} can be seen as an integral equation for $Y\in L^2_\bF(S,T;\bR^d)$ (not for the pair $(Y(\cdot),Z(\cdot,\cdot))$) including the martingale representation operator:
\begin{equation*}
	Y(t)=\psi(t)+\bE_t\Bigl[\int^T_tj(s,t)^\top Y(s)\rd s\Bigr]+\int^T_tk(s,t)^\top\cM_1[Y](s,t)\rd s,\ t\in(S,T),
\end{equation*}
where $\psi\in L^2_\bF(S,T;\bR^d)$ is defined by $\psi(t):=\bE_t[\tilde{\psi}(t)]$ for $t\in(S,T)$.


\begin{defi}\label{BSVIE_defi_martop}
Let $d_1,d_2\in\bN$. For each $n\in\bN_0$, we define the \emph{martingale representation operator} $\cM_n:L^2_\bF(S,T;\bR^{d_1\times d_2})\to L^2_\bF(\Delta_{n+1}(S,T);\bR^{d_1\times d_2})$ via the martingale representation theorem inductively as follows: for each $\xi\in L^2_\bF(S,T;\bR^{d_1\times d_2})$,
\begin{equation*}
	\cM_0[\xi](t_0):=\xi(t_0)
\end{equation*}
for $t_0\in(S,T)$, and
\begin{equation*}
	\cM_n[\xi](t_0,t_1,\mathalpha{\dots},t_n)=\bE_S[\cM_n[\xi](t_0,t_1,\mathalpha{\dots},t_n)]+\int^{t_n}_S\cM_{n+1}[\xi](t_0,t_1,\mathalpha{\dots},t_n,t_{n+1})\rd W(t_{n+1})
\end{equation*}
for $(t_0,t_1,\mathalpha{\dots},t_n)\in\Delta_{n+1}(S,T)$ and $n\in\bN$.
\end{defi}


\begin{rem}
\begin{itemize}
\item[(i)]
For each $\xi\in L^2_\bF(S,T;\bR^{d_1\times d_2})$ and $n\in\bN_0$, it holds that
\begin{equation*}
	\bE_S[\cM_n[\xi](t_0,t_1,\mathalpha{\dots},t_n)]=\bE[\cM_n[\xi](t_0,t_1,\mathalpha{\dots},t_n)]=\fF_n[\xi](t_0,t_1,\mathalpha{\dots},t_n),\ (t_0,t_1,\mathalpha{\dots},t_n)\in\Delta_{n+1}(S,T).
\end{equation*}
See \eqref{bstar_eq_n+k} below.
\item[(ii)]
Noting the isometry, we have $\|\cM_{n+1}[\xi]\|_{L^2_\bF(\Delta_{n+2}(S,T))}\leq\|\cM_n[\xi]\|_{L^2_\bF(\Delta_{n+1}(S,T))}$ for any $n\in\bN_0$ and $\xi\in L^2_\bF(S,T;\bR^{d_1\times d_2})$. By the induction, we see that $\cM_n:L^2_\bF(S,T;\bR^{d_1\times d_2})\to L^2_\bF(\Delta_{n+1}(S,T);\bR^{d_1\times d_2})$ is a bounded linear operator with the operator norm $\|\cM_n\|_\op\leq1$.
\end{itemize}
\end{rem}

We will consider the following generalized equation including (infinitely many) martingale representation operators:
\begin{equation}\label{BSVIE_eq_generalized}
	Y(t)=\psi(t)+\bE_t\Bigl[\int^T_tJ(s,t)^\top Y(s)\rd s\Bigr]+\sum^\infty_{n=1}\int_{\Delta_n(t,T)}k_n(s,t)^\top\cM_n[Y](s,t)\rd s,\ t\in(S,T),
\end{equation}
where $\psi\in L^2_\bF(S,T;\bR^d)$, $J\in\cJ_\bF(S,T;\bR^{d\times d})$ and $K:=\sum^\infty_{n=1}\fW_n[k_n]\in\cK_\bF(S,T;\bR^{d\times d})$. We call this equation a \emph{generalized BSVIE}. The solution we are looking for is the adapted process $Y\in L^2_\bF(S,T;\bR^d)$ satisfying \eqref{BSVIE_eq_generalized}. In order to investigate generalized BSVIE \eqref{BSVIE_eq_generalized}, we write it as an algebraic equation on $L^2_\bF(S,T;\bR^d)$ as follows:
\begin{equation}\label{BSVIE_eq_BSVIEalg}
	Y=\psi+J^\top\bast Y+K^\top\bstar Y,
\end{equation}
where $J^\top\bast Y$ corresponds to the term $\bE_t[\int^T_tJ(s,t)^\top Y(s)\rd s]$, and $K^\top\bstar Y$ corresponds to the infinite sum of the integrals including martingale representation operators. Then, we provide a variation of constants formula for \eqref{BSVIE_eq_BSVIEalg}. Furthermore, we show the duality principle between a generalized SVIE and a generalized BSVIE.


\subsection{Backward $\star$-product}\label{subsection_bstar}

In this subsection, we investigate the term $K^\top\bstar Y$ appearing in equation \eqref{BSVIE_eq_BSVIEalg}.


\begin{defi}\label{bstar_defi_bstarprod}
For each $\star$-Volterra kernel $K\in\cK_\bF(S,T;\bR^{d\times d})$ and $\xi\in L^2_\bF(S,T;\bR^{d\times d_1})$ with $d_1\in\bN$, we define the \emph{backward $\star$-product} $K\bstar\xi\in L^2_\bF(S,T;\bR^{d\times d_1})$ by the following Wiener--It\^{o} chaos expansion:
\begin{equation*}
	\fF_n[K\bstar\xi](t_0,t):=\sum^\infty_{k=n}\int_{\Delta_{k-n}(t_0,T)}\fF_{k-n}[K](s,t_0)\fF_k[\xi](s,t_0,t)\rd s
\end{equation*}
for $(t_0,t)=(t_0,t_1,\mathalpha{\dots},t_n)\in\Delta_{n+1}(S,T)$ and $n\in\bN_0$, where $\int_{\Delta_0(t_0,T)}\alpha\rd s:=\alpha$ for each vector $\alpha$.
\end{defi}

The following lemma ensures the well-definedness of the backward $\star$-product.


\begin{lemm}\label{bstar_lemm_bstarwelldef}
For each $\star$-Volterra kernel $K\in\cK_\bF(S,T;\bR^{d\times d})$ and $\xi\in L^2_\bF(S,T;\bR^{d\times d_1})$ with $d_1\in\bN$, the backward $\star$-product $K\bstar\xi\in L^2_\bF(S,T;\bR^{d\times d_1})$ is well-defined and satisfies
\begin{equation*}
	\|K\bstar\xi\|_{L^2_\bF(S,T)}\leq\knorm K\knorm_{\cK_\bF(S,T)}\|\xi\|_{L^2_\bF(S,T)}.
\end{equation*}
Furthermore, if $K_1,K_2\in\cK_\bF(S,T;\bR^{d\times d})$, then $K_1\bstar K_2$ is in $\cK_\bF(S,T;\bR^{d\times d})$ and satisfies
\begin{equation*}
	\knorm K_1\bstar K_2\knorm_{\cK_\bF(S,T)}\leq\knorm K_1\knorm_{\cK_\bF(S,T)}\knorm K_2\knorm_{\cK_\bF(S,T)}.
\end{equation*}
\end{lemm}


\begin{proof}
By using Minkowski's inequality and H\"{o}lder's inequality, we have, for each $n\in\bN_0$,
\begin{align*}
	&\Bigl\{\int_{\Delta_{n+1}(S,T)}\Bigl(\sum^\infty_{k=n}\int_{\Delta_{k-n}(\max t,T)}|\fF_{k-n}[K](s,\max t)\fF_k[\xi](s,t)|\rd s\Bigr)^2\rd t\Bigr\}^{1/2}\\
	&\leq\sum^\infty_{k=n}\Bigl\{\int_{\Delta_{n+1}(S,T)}\Bigl(\int_{\Delta_{k-n}(\max t,T)}|\fF_{k-n}[K](s,\max t)\fF_k[\xi](s,t)|\rd s\Bigr)^2\rd t\Bigr\}^{1/2}\\
	&\leq\sum^\infty_{k=n}\Bigl\{\int_{\Delta_{n+1}(S,T)}\Bigl(\int_{\Delta_{k-n}(\max t,T)}|\fF_{k-n}[K](s,\max t)|_\op\,|\fF_k[\xi](s,t)|\rd s\Bigr)^2\rd t\Bigr\}^{1/2}\\
	&\leq\sum^\infty_{k=n}\Bigl\{\int_{\Delta_{n+1}(S,T)}\int_{\Delta_{k-n}(\max t,T)}|\fF_{k-n}[K](s,\max t)|^2_\op\rd s\int_{\Delta_{k-n}(\max t,T)}|\fF_k[\xi](s,t)|^2\rd s\rd t\Bigr\}^{1/2}\\
	&\leq\sum^\infty_{k=n}\knorm\fF_{k-n}[K]\knorm_{\cV_{k-n+1}(S,T)}\|\fF_k[\xi]\|_{L^2(\Delta_{k+1}(S,T))}.
\end{align*}
This implies that
\begin{equation*}
	\|\fF_n[K\bstar\xi]\|_{L^2(\Delta_{n+1}(S,T))}\leq\sum^\infty_{k=n}\knorm\fF_{k-n}[K]\knorm_{\cV_{k-n+1}(S,T)}\|\fF_k[\xi]\|_{L^2(\Delta_{k+1}(S,T))}.
\end{equation*}
Therefore, by Young's convolution inequality and the isometry,
\begin{align*}
	\sum^\infty_{n=0}\|\fF_n[K\bstar\xi]\|^2_{L^2(\Delta_{n+1}(S,T))}&\leq\sum^\infty_{n=0}\Bigl(\sum^\infty_{k=n}\knorm\fF_{k-n}[K]\knorm_{\cV_{k-n+1}(S,T)}\|\fF_k[\xi]\|_{L^2(\Delta_{k+1}(S,T))}\Bigr)^2\\
	&\leq\Bigl(\sum^\infty_{n=0}\knorm\fF_n[K]\knorm_{\cV_{n+1}(S,T)}\Bigr)^2\sum^\infty_{n=0}\|\fF_n[\xi]\|^2_{L^2(\Delta_{n+1}(S,T))}\\
	&=\knorm K\knorm^2_{\cK_\bF(S,T)}\|\xi\|^2_{L^2_\bF(S,T)}<\infty.
\end{align*}
This implies the first assertion. Similarly, we can show that
\begin{equation*}
	\knorm\fF_n[K_1\bstar K_2]\knorm_{\cV_{n+1}(S,T)}\leq\sum^\infty_{k=n}\knorm\fF_{k-n+1}[K_1]\knorm_{\cV_{k-n+1}(S,T)}\knorm\fF_k[K_2]\knorm_{\cV_{k+1}(S,T)}
\end{equation*}
for each $n\in\bN_0$, and thus
\begin{align*}
	\knorm K_1\bstar K_2 \knorm_{\cK_\bF(S,T)}&=\sum^\infty_{n=0}\knorm\fF_n[K_1\bstar K_2]\knorm_{\cV_{n+1}(S,T)}\\
	&\leq\sum^\infty_{n=0}\sum^\infty_{k=n}\knorm\fF_{k-n}[K_1]\knorm_{\cV_{k-n+1}(S,T)}\knorm\fF_k[K_2]\knorm_{\cV_{k+1}(S,T)}\\
	&\leq\sum^\infty_{n=0}\knorm\fF_n[K_1]\knorm_{\cV_{n+1}(S,T)}\sum^\infty_{n=0}\knorm\fF_n[K_2]\knorm_{\cV_{n+1}(S,T)}\\
	&=\knorm K_1\knorm_{\cK_\bF(S,T)}\knorm K_2\knorm_{\cK_\bF(S,T)}.
\end{align*}
Hence, the last assertion holds.
\end{proof}

The following proposition shows fundamental algebraic properties of the backward $\star$-product.


\begin{prop}\label{bstar_prop_bstaralg}
For each $K,K_1,K_2\in\cK_\bF(S,T;\bR^{d\times d})$, $\xi,\xi_1,\xi_2\in L^2_\bF(S,T;\bR^{d\times d_1})$ with $d_1\in\bN$, and $\alpha\in\bR$, the following hold:
\begin{gather*}
	I_d\bstar\xi=\xi,\\
	(K_1+K_2)\bstar\xi=K_1\bstar\xi+K_2\bstar\xi,\\
	K\bstar(\xi_1+\xi_2)=K\bstar\xi_1+K\bstar\xi_2,\\
	\alpha(K\bstar\xi)=(\alpha K)\bstar\xi=K\bstar(\alpha\xi),\\
	K_1\bstar(K_2\bstar\xi)=(K^\top_2\star K^\top_1)^\top\bstar\xi.
\end{gather*}
\end{prop}


\begin{proof}
We prove $K_1\bstar(K_2\bstar\xi)=(K^\top_2\star K^\top_1)^\top\bstar\xi$. The others are trivial from the definition. Note that $\fF_n[K]^\top=\fF_n[K^\top]$ for any $n\in\bN_0$ and $K\in\cK_\bF(S,T;\bR^{d\times d})$. For each $n\in\bN_0$ and $t\in\Delta_{n+1}(S,T)$, we have
\begin{align*}
	&\fF_n[K_1\bstar(K_2\bstar\xi)](t)\\
	&=\sum^\infty_{k=n}\int_{\Delta_{k-n}(\max t,T)}\fF_{k-n}[K_1](s,\max t)\fF_k[K_2\bstar\xi](s,t)\rd s\\
	&=\sum^\infty_{k=n}\int_{\Delta_{k-n}(\max t,T)}\fF_{k-n}[K_1](s,\max t)\Bigl\{\sum^\infty_{\ell=k}\int_{\Delta_{\ell-k}(\max s,T)}\fF_{\ell-k}[K_2](r,\max s)\fF_\ell[\xi](r,s,t)\rd r\Bigr\}\rd s\\
	&=\sum^\infty_{\ell=n}\sum^\ell_{k=n}\int_{\Delta_{k-n}(\max t,T)}\int_{\Delta_{\ell-k}(\max s,T)}\fF_{k-n}[K_1](s,\max t)\fF_{\ell-k}[K_2](r,\max s)\fF_\ell[\xi](r,s,t)\rd r\rd s\\
	&=\sum^\infty_{\ell=n}\sum^\ell_{k=n}\int_{\Delta_{\ell-n}(\max t,T)}(\fF_{\ell-k}[K_2]^\top\triangleright\fF_{k-n}[K_1]^\top)(s,\max t)^\top\fF_\ell[\xi](s,t)\rd s\\
	&=\sum^\infty_{\ell=n}\int_{\Delta_{\ell-n}(\max t,T)}\Bigl\{\sum^{\ell-n}_{k=0}(\fF_{\ell-n-k}[K^\top_2]\triangleright\fF_k[K^\top_1])(s,\max t)\Bigr\}^\top\fF_\ell[\xi](s,t)\rd s\\
	&=\sum^\infty_{\ell=n}\int_{\Delta_{\ell-n}(\max t,T)}\fF_{\ell-n}[(K^\top_2\star K^\top_1)^\top](s,\max t)\fF_\ell[\xi](s,t)\rd s\\
	&=\fF_n[(K^\top_2\star K^\top_1)^\top\bstar\xi](t).
\end{align*}
Hence, we have $K_1\bstar(K_2\bstar\xi)=(K^\top_2\star K^\top_1)^\top\bstar\xi$.
\end{proof}

Next, we prove a representation of the backward $\star$-product in terms of martingale representation operators.


\begin{prop}\label{bstar_prop_integral}
For each $\star$-Volterra kernel $K\in\cK_\bF(S,T;\bR^{d\times d})$ and $\xi\in L^2_\bF(S,T;\bR^{d\times d_1})$ with $d_1\in\bN$, it holds that
\begin{equation}\label{bstar_eq_integral}
	(K\bstar\xi)(t)=\fF_0[K](t)\xi(t)+\sum^\infty_{n=1}\int_{\Delta_n(t,T)}\fF_n[K](s,t)\cM_n[\xi](s,t)\rd s,\ t\in(S,T),
\end{equation}
where the infinite sum in the right-hand side converges in $L^2_\bF(S,T;\bR^{d\times d_1})$.
\end{prop}


\begin{proof}
First, we show that, for any $k,n\in\bN_0$, and $(s_1,\mathalpha{\dots},s_k,t_0,t)=(s_1,\mathalpha{\dots},s_k,t_0,t_1,\mathalpha{\dots},t_n)\in\Delta_{n+k+1}(S,T)$,
\begin{equation}\label{bstar_eq_n+k}
	\fF_n[\cM_k[\xi](s_1,\mathalpha{\dots},s_k,\cdot)](t_0,t)=\fF_{n+k}[\xi](s_1,\mathalpha{\dots},s_k,t_0,t).
\end{equation}
Clearly, for $k=0$, \eqref{bstar_eq_n+k} holds for any $n\in\bN_0$. Assume that \eqref{bstar_eq_n+k} holds for any $n\in\bN_0$ for some $k\in\bN_0$. By the definition of the martingale representation operator $\cM_{k+1}$, we have, for $(s,s_k)=(s_0,\mathalpha{\dots},s_{k-1},s_k)\in\Delta_{k+1}(S,T)$,
\begin{align*}
	\cM_k[\xi](s,s_k)&=\bE_S[\cM_k[\xi](s,s_k)]+\int^{s_k}_S\cM_{k+1}[\xi](s,s_k,t)\rd W(t)\\
	&=\bE_S[\cM_k[\xi](s,s_k)]+\int^{s_k}_S\sum^\infty_{n=0}\fW_n[\cM_{k+1}[\xi](s,s_k,\cdot)](t)\rd W(t)\\
	&=\bE_S[\cM_k[\xi](s,s_k)]+\sum^\infty_{n=0}\int^{s_k}_S\fW_n[\cM_{k+1}[\xi](s,s_k,\cdot)](t)\rd W(t)\\
	&=\bE_S[\cM_k[\xi](s,s_k)]+\int^{s_k}_S\fF_0[\cM_{k+1}[\xi](s,s_k,\cdot)](t)\rd W(t)\\
	&\hspace{0.5cm}+\sum^\infty_{n=1}\int^{s_k}_S\int^t_S\int^{t_1}_S\mathalpha{\cdots}\int^{t_{n-1}}_S\fF_n[\cM_{k+1}[\xi](s,s_k,\cdot)](t,t_1,\mathalpha{\dots},t_n)\rd W(t_n)\mathalpha{\cdots}\rd W(t_1)\rd W(t).
\end{align*}
From this, together with the assumption of the induction, we have, for any $n\in\bN_0$ and $(s,s_k,t_0,t)=(s_0,\mathalpha{\dots},s_{k-1},s_k,t_0,t_1,\mathalpha{\dots},t_n)\in\Delta_{n+k+2}(S,T)$,
\begin{equation*}
	\fF_n[\cM_{k+1}[\xi](s,s_k,\cdot)](t_0,t)=\fF_{n+1}[\cM_k[\xi](s,\cdot)](s_k,t_0,t)=\fF_{n+k+1}[\xi](s,s_k,t_0,t).
\end{equation*}
By the induction, we see that \eqref{bstar_eq_n+k} holds for any $k,n\in\bN_0$.

We write the right-hand side of \eqref{bstar_eq_integral} by $\sum^\infty_{k=0}\Psi^K_k[\xi](t)$, where $\Psi^K_0[\xi](t):=\fF_0[K](t)\xi(t)$, $t\in(S,T)$, and
\begin{equation*}
	\Psi^K_k[\xi](t):=\int_{\Delta_k(t,T)}\fF_k[K](s,t)\cM_k[\xi](s,t)\rd s,\ t\in(S,T),
\end{equation*}
for $k\in\bN$. Noting that $\|\cM_k[\xi]\|_{L^2_\bF(\Delta_{k+1}(S,T))}\leq\|\xi\|_{L^2_\bF(S,T)}$, we see that
\begin{align*}
	\|\Psi^K_k[\xi]\|^2_{L^2_\bF(S,T)}&=\bE\Bigl[\int^T_S\Bigl|\int_{\Delta_k(t,T)}\fF_k[K](s,t)\cM_k[\xi](s,t)\rd s\Bigr|^2\rd t\Bigr]\\
	&\leq\bE\Bigl[\int^T_S\Bigl(\int_{\Delta_k(t,T)}|\fF_k[K](s,t)|^2_\op\rd s\Bigr)\Bigl(\int_{\Delta_k(t,T)}|\cM_k[\xi](s,t)|^2\rd s\Bigr)\rd t\Bigr]\\
	&\leq\knorm \fF_k[K]\knorm^2_{\cV_{k+1}(S,T)}\|\cM_k[\xi]\|^2_{L^2_\bF(\Delta_{k+1}(S,T))}\\
	&\leq\knorm \fF_k[K]\knorm^2_{\cV_{k+1}(S,T)}\|\xi\|^2_{L^2_\bF(S,T)}.
\end{align*}
This implies that, for each $k\in\bN_0$, $\Psi^K_k$ is a bounded linear operator on $L^2_\bF(S,T;\bR^{d\times d_1})$ with the operator norm $\|\Psi^K_k\|_\op\leq\knorm\fF_k[K]\knorm_{\cV_{k+1}(S,T)}$. Noting that $\sum^\infty_{k=0}\knorm\fF_k[K]\knorm_{\cV_{k+1}(S,T)}=\knorm K\knorm_{\cK_\bF(S,T)}<\infty$, the infinite sum $\sum^\infty_{k=0}\Psi^K_k[\xi]$ converges in $L^2_\bF(S,T;\bR^{d\times d_1})$. Furthermore, by the stochastic Fubini theorem and \eqref{bstar_eq_n+k}, we have
\begin{align*}
	\fF_n[\Psi^K_k[\xi]](t_0,t)&=\int_{\Delta_k(t_0,T)}\fF_k[K](s,t_0)\fF_n[\cM_k[\xi](s,\cdot)](t_0,t)\rd s\\
	&=\int_{\Delta_k(t_0,T)}\fF_k[K](s,t_0)\fF_{n+k}[\xi](s,t_0,t)\rd s
\end{align*}
for $(t_0,t)=(t_0,t_1,\mathalpha{\dots},t_n)\in\Delta_{n+1}(S,T)$ and $k,n\in\bN_0$. Thus, we get
\begin{align*}
	\fF_n\Bigl[\sum^\infty_{k=0}\Psi^K_k[\xi]\Bigr](t_0,t)&=\sum^\infty_{k=0}\fF_n[\Psi^K_k[\xi]](t_0,t)\\
	&=\sum^\infty_{k=0}\int_{\Delta_k(t_0,T)}\fF_k[K](s,t_0)\fF_{n+k}[\xi](s,t_0,t)\rd s\\
	&=\sum^\infty_{k=n}\int_{\Delta_{k-n}(t_0,T)}\fF_{k-n}[K](s,t_0)\fF_k[\xi](s,t_0,t)\rd s\\
	&=\fF_n[K\bstar\xi](t_0,t)
\end{align*}
for $(t_0,t)=(t_0,t_1,\mathalpha{\dots},t_n)\in\Delta_{n+1}(S,T)$ and $n\in\bN_0$. This implies that the equality $K\bstar\xi=\sum^\infty_{k=0}\Psi^K_k[\xi]$ holds in $L^2_\bF(S,T;\bR^d)$, and we complete the proof.
\end{proof}

The following is a \emph{duality principle} with respect to the $\star$-product and the backward $\star$-product.


\begin{prop}\label{bstar_prop_duality}
For each $\star$-Volterra kernel $K\in\cK_\bF(S,T;\bR^{d\times d})$ and $\varphi,\psi\in L^2_\bF(S,T;\bR^d)$, it holds that
\begin{equation*}
	\langle K\star\varphi,\psi\rangle_{L^2_\bF(S,T)}=\langle\varphi,K^\top\bstar\psi\rangle_{L^2_\bF(S,T)}.
\end{equation*}
Here, for each $\xi_1,\xi_2\in L^2_\bF(S,T;\bR^d)$, $\langle\xi_1,\xi_2\rangle_{L^2_\bF(S,T)}:=\bE[\int^T_S\langle\xi_1(t),\xi_2(t)\rangle\rd t]$ denotes the inner product in the Hilbert space $L^2_\bF(S,T;\bR^d)$.
\end{prop}


\begin{proof}
By the isometry, together with the Fubini theorem, we have
\begin{align*}
	\langle K\star\varphi,\psi\rangle_{L^2_\bF(S,T)}&=\sum^\infty_{n=0}\langle\fF_n[K\star\varphi],\fF_n[\psi]\rangle_{L^2(\Delta_{n+1}(S,T))}\\
	&=\sum^\infty_{n=0}\Bigl\langle\sum^n_{k=0}\fF_{n-k}[K]\triangleright\fF_k[\varphi],\fF_n[\psi]\Bigr\rangle_{L^2(\Delta_{n+1}(S,T))}\\
	&=\sum^\infty_{k=0}\sum^\infty_{n=k}\langle\fF_{n-k}[K]\triangleright\fF_k[\varphi],\fF_n[\psi]\rangle_{L^2(\Delta_{n+1}(S,T))}\\
	&=\sum^\infty_{k=0}\sum^\infty_{n=k}\int_{\Delta_{k+1}(S,T)}\int_{\Delta_{n-k}(\max t,T)}\langle\fF_{n-k}[K](s,\max t)\fF_k[\varphi](t),\fF_n[\psi](s,t)\rangle\rd s\rd t\\
	&=\sum^\infty_{k=0}\int_{\Delta_{k+1}(S,T)}\Bigl\langle\fF_k[\varphi](t),\sum^\infty_{n=k}\int_{\Delta_{n-k}(\max t,T)}\fF_{n-k}[K](s,\max t)^\top\fF_n[\psi](s,t)\rd s\Bigr\rangle\rd t\\
	&=\sum^\infty_{k=0}\int_{\Delta_{k+1}(S,T)}\langle\fF_k[\varphi](t),\fF_k[K^\top\bstar\psi](t)\rangle\rd t\\
	&=\sum^\infty_{k=0}\langle\fF_k[\varphi],\fF_k[K^\top\bstar\psi]\rangle_{L^2(\Delta_{k+1}(S,T))}\\
	&=\langle\varphi,K^\top\bstar\psi\rangle_{L^2_\bF(S,T)}.
\end{align*}
Thus, we get the assertion.
\end{proof}


\subsection{Backward $\ast$-product}\label{subsection_bast}

Next, we investigate the term $J^\top\bast Y$ appearing in equation \eqref{BSVIE_eq_BSVIEalg}.


\begin{defi}\label{bast_defi_bastprod}
For each $\Xi\in L^2_{\bF,\ast}(\Delta_2(S,T);\bR^{d_1\times d_2})$ and $\xi\in L^2_\bF(S,T;\bR^{d_2\times d_3})$ with $d_1,d_2,d_3\in\bN$, we define the \emph{backward $\ast$-product} $\Xi\bast\xi$ by
\begin{equation*}
	(\Xi\bast\xi)(t):=\bE_t\Bigl[\int^T_t\Xi(s,t)\xi(s)\rd s\Bigr],\ t\in(S,T).
\end{equation*}
Also, for each $\Xi_1\in L^2_{\bF,\ast}(\Delta_2(S,T);\bR^{d_1\times d_2})$ and $\Xi_2\in  L^2_{\bF,\ast}(\Delta_2(S,T);\bR^{d_2\times d_3})$, we define
\begin{equation*}
	(\Xi_1\bast \Xi_2)(t,s):=\bE_t\Bigl[\int^T_t\Xi_1(r,t)\Xi_2(r,s)\rd r\Bigr],\ (t,s)\in\Delta_2(S,T).
\end{equation*}
\end{defi}


\begin{lemm}\label{bast_lemm_chaos}
For each $\Xi\in  L^2_{\bF,\ast}(\Delta_2(S,T);\bR^{d_1\times d_2})$ and $\xi\in L^2_\bF(S,T;\bR^{d_2\times d_3})$ with $d_1,d_2,d_3\in\bN$, the backward $\ast$-product $\Xi\bast\xi$ is in $L^2_\bF(S,T;\bR^{d_1\times d_3})$, and the Wiener--It\^{o} chaos expansion satisfies
\begin{equation*}
	\fF_n[\Xi\bast\xi](t_0,t)=\sum^\infty_{k=n}\int_{\Delta_{k-n+1}(t_0,T)}\bfF_{k-n}[\Xi](s,t_0)\fF_k[\xi](s,t)\rd s,
\end{equation*}
for each $n\in\bN_0$ and $(t_0,t)=(t_0,t_1,\mathalpha{\dots},t_n)\in\Delta_{n+1}(S,T)$. Furthermore, for each $\Xi_1\in  L^2_{\bF,\ast}(\Delta_2(S,T);\bR^{d_1\times d_2})$ and $\Xi_2\in  L^2_{\bF,\ast}(\Delta_2(S,T);\bR^{d_2\times d_3})$, the backward $\ast$-product $\Xi_1\bast\Xi_2$ is in $L^2_{\bF,\ast}(\Delta_2(S,T);\bR^{d_1\times d_3})$, and the Wiener--It\^{o} chaos expansion satisfies
\begin{equation*}
	\bfF_n[\Xi_1\bast\Xi_2](t_0,t)=\sum^\infty_{k=n}\int_{\Delta_{k-n+1}(t_0,T)}\bfF_{k-n}[\Xi_1](s,t_0)\bfF_k[\Xi_2](s,t)\rd s,
\end{equation*}
for each $n\in\bN_0$ and $(t_0,t)=(t_0,t_1,\mathalpha{\dots},t_n,t_{n+1})\in\Delta_{n+2}(S,T)$.
\end{lemm}


\begin{proof}
We prove the first assertion. The second can be proved similarly. By using the conditional H\"{o}lder's inequality and noting that $\Xi(s,t)$ is independent of $\cF_t$, we have
\begin{align*}
	\bE\Bigl[\int^T_S\Bigl|\bE_t\Bigl[\int^T_t\Xi(s,t)\xi(s)\rd s\Bigr]\Bigr|^2\rd t\Bigr]&\leq\bE\Bigl[\int^T_S\bE_t\Bigl[\int^T_t|\Xi(s,t)||\xi(s)|\rd s\Bigr]^2\rd t\Bigr]\\
	&\leq\bE\Bigl[\int^T_S\bE_t\Bigl[\int^T_t|\Xi(s,t)|^2\rd s\Bigr]\bE_t\Bigl[\int^T_t|\xi(s)|^2\rd s\Bigr]\rd t\Bigr]\\
	&=\int^T_S\bE\Bigl[\int^T_t|\Xi(s,t)|^2\rd s\Bigr]\bE\Bigl[\int^T_t|\xi(s)|^2\rd s\Bigr]\rd t\\
	&\leq\bE\Bigl[\int^T_S\int^T_t|\Xi(s,t)|^2\rd s\rd t\Bigr]\bE\Bigl[\int^T_S|\xi(s)|^2\rd s\Bigr]\\
	&<\infty.
\end{align*}
This implies that $\Xi\bast\xi\in L^2_\bF(S,T;\bR^{d_1\times d_3})$, and that the operations $\Xi\mapsto\Xi\bast\xi$ and $\xi\mapsto\Xi\bast\xi$ are continuous. Observe that
\begin{equation*}
	\Xi\bast\xi=\Bigl(\sum^\infty_{m=0}\bfW_m[\Xi]\Bigr)\bast\Bigl(\sum^\infty_{n=0}\fW_n[\xi]\Bigr)=\sum^\infty_{m,n=0}\bfW_m[\Xi]\bast\fW_n[\xi].
\end{equation*}
For each $m,n\in\bN_0$ and $t\in(S,T)$, we have
\begin{align*}
	(\bfW_m[\Xi]\bast\fW_n[\xi])(t)&=\bE_t\Bigl[\int^T_t\bfW_m[\Xi](s,t)\fW_n[\xi](s)\rd s\Bigr]\\
	&=\int^T_t\bE_t\Bigl[\Bigl(\int^s_t\int^{t_1}_t\mathalpha{\cdots}\int^{t_{m-1}}_t\bfF_m[\Xi](s,t_1,\mathalpha{\dots},t_m,t)\rd W(t_m)\mathalpha{\cdots}\rd W(t_1)\Bigr)\\
	&\hspace{1.5cm}\times\Bigl(\int^s_S\int^{t_1}_S\mathalpha{\cdots}\int^{t_{n-1}}_S\fF_n[\xi](s,t_1,\mathalpha{\dots},t_n)\rd W(t_n)\mathalpha{\cdots}\rd W(t_1)\Bigr)\Bigr]\rd s.
\end{align*}
If $m>n$, we get
\begin{align*}
	&(\bfW_m[\Xi]\bast\fW_n[\xi])(t)\\
	&=\int^T_t\int^s_t\int^{t_1}_t\mathalpha{\cdots}\int^{t_{n-1}}_t\bE_t\Bigl[\int^{t_n}_t\int^{t_{n+1}}_t\mathalpha{\cdots}\int^{t_{m-1}}_t\bfF_m[\Xi](s,t_1,\mathalpha{\dots},t_n,t_{n+1},\mathalpha{\dots},t_m,t)\rd W(t_m)\mathalpha{\cdots}\rd W(t_{n+1})\Bigr]\\
	&\hspace{4cm}\times\fF_n[\xi](s,t_1,\mathalpha{\dots},t_n)\rd t_n\mathalpha{\cdots}\rd t_1\rd s\\
	&=0.
\end{align*}
On the other hand, if $m\leq n$, we have
\begin{align*}
	&(\bfW_m[\Xi]\bast\fW_n[\xi])(t)\\
	&=\int^T_t\int^s_t\int^{t_1}_t\mathalpha{\cdots}\int^{t_{m-1}}_t\bfF_m[\Xi](s,t_1,\mathalpha{\dots},t_m,t)\\
	&\hspace{1cm}\times\bE_t\Bigl[\int^{t_m}_S\int^{t_{m+1}}_S\mathalpha{\cdots}\int^{t_{n-1}}_S\fF_n[\xi](s,t_1,\mathalpha{\dots},t_m,t_{m+1},\mathalpha{\dots},t_n)\rd W(t_n)\mathalpha{\cdots}\rd W(t_{m+1})\Bigr]\rd t_m\mathalpha{\cdots}\rd t_1\rd s\\
	&=\int^T_t\int^s_t\int^{t_1}_t\mathalpha{\cdots}\int^{t_{m-1}}_t\bfF_m[\Xi](s,t_1,\mathalpha{\dots},t_m,t)\\
	&\hspace{1cm}\times\int^t_S\int^{t_{m+1}}_S\mathalpha{\cdots}\int^{t_{n-1}}_S\fF_n[\xi](s,t_1,\mathalpha{\dots},t_m,t_{m+1},\mathalpha{\dots},t_n)\rd W(t_n)\mathalpha{\cdots}\rd W(t_{m+1})\rd t_m\mathalpha{\cdots}\rd t_1\rd s\\
	&=\int^t_S\int^{t_1}_S\mathalpha{\cdots}\int^{t_{n-m-1}}_S\int_{\Delta_{m+1}(t,T)}\bfF_m[\Xi](s,t)\fF_n[\xi](s,t_1,\mathalpha{\dots},t_{n-m})\rd s\rd W(t_{n-m})\mathalpha{\cdots}\rd W(t_1),
\end{align*}
where we used the stochastic Fubini theorem in the last equality. Consequently, we obtain
\begin{align*}
	&(\Xi\bast\xi)(t)\\
	&=\sum_{m\leq n}\int^t_S\int^{t_1}_S\mathalpha{\cdots}\int^{t_{n-m-1}}_S\int_{\Delta_{m+1}(t,T)}\bfF_m[\Xi](s,t)\fF_n[\xi](s,t_1,\mathalpha{\dots},t_{n-m})\rd s\rd W(t_{n-m})\mathalpha{\cdots}\rd W(t_1)\\
	&=\sum^\infty_{n=0}\int^t_S\int^{t_1}_S\mathalpha{\cdots}\int^{t_{n-1}}_S\Bigl\{\sum^\infty_{k=n}\int_{\Delta_{k-n+1}(t,T)}\bfF_{k-n}[\Xi](s,t)\fF_k[\xi](s,t_1,\mathalpha{\dots},t_n)\rd s\Bigr\}\rd W(t_n)\mathalpha{\cdots}\rd W(t_1).
\end{align*}
This implies that
\begin{equation*}
	\fF_n[\Xi\bast\xi](t_0,t_1,\mathalpha{\dots},t_n)=\sum^\infty_{k=n}\int_{\Delta_{k-n+1}(t_0,T)}\bfF_{k-n}[\Xi](s,t_0)\fF_k[\xi](s,t_1,\mathalpha{\dots},t_n)\rd s
\end{equation*}
for each $n\in\bN_0$ and $(t_0,t_1,\mathalpha{\dots},t_n)\in\Delta_{n+1}(S,T)$. This completes the proof.
\end{proof}


\begin{prop}\label{bast_prop_alg}
For each $J,J_1,J_2,J_3\in\cJ_\bF(S,T;\bR^{d\times d})$, $\xi,\xi_1,\xi_2\in L^2_\bF(S,T;\bR^{d\times d_1})$ with $d_1\in\bN$, $K\in\cK_\bF(S,T;\bR^{d\times d})$ and $\alpha\in\bR$, the following hold:
\begin{gather}
	J_1\bast J_2\in\cJ_\bF(S,T;\bR^{d\times d})\ \text{and}\ \knorm J_1\bast J_2\knorm_{\cJ_\bF(S,T)}\leq\knorm J_1\knorm_{\cJ_\bF(S,T)}\knorm J_2\knorm_{\cJ_\bF(S,T)},\label{bast_eq_JJ}\\
	J\bast\xi\in L^2_\bF(S,T;\bR^{d\times d_1})\ \text{and}\ \|J\bast\xi\|_{L^2_\bF(S,T)}\leq\knorm J\knorm_{\cJ_\bF(S,T)}\|\xi\|_{L^2_\bF(S,T)},\label{bast_eq_Jxi}\\
	J\bast K\in \cK_\bF(S,T;\bR^{d\times d})\ \text{and}\ \knorm J\bast K\knorm_{\cK_\bF(S,T)}\leq\knorm J\knorm_{\cJ_\bF(S,T)}\knorm K\knorm_{\cK_\bF(S,T)},\label{bast_eq_JK}
\end{gather}
and
\begin{gather}
	J_1\bast(J_2+J_3)=J_1\bast J_2+J_1\bast J_3,\nonumber\\
	(J_1+J_2)\bast J_3=J_1\bast J_3+J_2\bast J_3,\nonumber\\
	\alpha(J_1\bast J_2)=(\alpha J_1)\bast J_2=J_1\bast(\alpha J_2),\nonumber\\
	J\bast(\xi_1+\xi_2)=J\bast\xi_1+J\bast\xi_2,\nonumber\\
	(J_1+J_2)\bast\xi=J_1\bast\xi+J_2\bast\xi,\nonumber\\
	\alpha(J\bast\xi)=(\alpha J)\bast\xi=J\bast(\alpha\xi),\nonumber\\
	J_1\bast(J_2\bast J_3)=(J^\top_2\ast J^\top_1)^\top\bast J_3,\label{bast_eq_JJJ}\\
	J_1\bast(J_2\bast\xi)=(J^\top_2\ast J^\top_1)^\top\bast\xi,\label{bast_eq_JJxi}\\
	J\bast(K\bstar\xi)=(K^\top\star J^\top)^\top\bast\xi,\label{bast_eq_JKxi}\\
	K\bstar(J\bast\xi)=(J^\top\ast K^\top)^\top\bstar\xi.\label{bast_eq_KJxi}
\end{gather}
\end{prop}


\begin{proof}
Noting \cref{bast_lemm_chaos}, the assertions \eqref{bast_eq_JJ}, \eqref{bast_eq_Jxi} and \eqref{bast_eq_JK} can be proved by the same way as in \cref{bstar_lemm_bstarwelldef}, and thus we omit the proofs. We prove \eqref{bast_eq_JJxi}, \eqref{bast_eq_JKxi} and \eqref{bast_eq_KJxi}. The equality \eqref{bast_eq_JJJ} can be proved by the same way as \eqref{bast_eq_JJxi}, and the other assertions are clear from the definition.

First, we prove \eqref{bast_eq_JJxi}. Noting the measurability of $J_1$ and using Fubini's theorem, we see that, for each $t\in(S,T)$,
\begin{align*}
	(J_1\bast(J_2\bast\xi))(t)&=\bE_t\Bigl[\int^T_tJ_1(s,t)(J_2\bast\xi)(s)\rd s\Bigr]\\
	&=\bE_t\Bigl[\int^T_tJ_1(s,t)\bE_s\Bigl[\int^T_sJ_2(r,s)\xi(r)\rd r\Bigr]\rd s\Bigr]\\
	&=\bE_t\Bigl[\int^T_tJ_1(s,t)\int^T_sJ_2(r,s)\xi(r)\rd r\rd s\Bigr]\\
	&=\bE_t\Bigl[\int^T_t\Bigl(\int^r_tJ_2(r,s)^\top J_1(s,t)^\top\rd s\Bigr)^\top\xi(r)\rd r\Bigr]\\
	&=\bE_t\Bigl[\int^T_t(J^\top_2\ast J^\top_1)(r,t)^\top\xi(r)\rd r\Bigr]\\
	&=((J^\top_2\ast J^\top_1)^\top\bast\xi)(t).
\end{align*}
Thus, \eqref{bast_eq_JJxi} holds.

Next, we prove \eqref{bast_eq_JKxi}. By \cref{bast_lemm_chaos}, for each $n\in\bN_0$ and $(t_0,t)=(t_0,t_1,\mathalpha{\dots},t_n)\in\Delta_{n+1}(S,T)$,
\begin{align*}
	&\fF_n[J\bast(K\bstar\xi)](t_0,t)\\
	&=\sum^\infty_{k=n}\int_{\Delta_{k-n+1}(t_0,T)}\bfF_{k-n}[J](s,t_0)\fF_k[K\bstar\xi](s,t)\rd s\\
	&=\sum^\infty_{k=n}\int_{\Delta_{k-n+1}(t_0,T)}\bfF_{k-n}[J](s,t_0)\Bigl\{\sum^\infty_{\ell=k}\int_{\Delta_{\ell-k}(\max s,T)}\fF_{\ell-k}[K](r,\max s)\fF_\ell[\xi](r,s,t)\rd r\Bigr\}\rd s\\
	&=\sum^\infty_{k=n}\sum^\infty_{\ell=k}\int_{\Delta_{\ell-n+1}(t_0,T)}(\fF_{\ell-k}[K^\top]\triangleright\bfF_{k-n}[J^\top])(s,t_0)^\top\fF_\ell[\xi](s,t)\rd s\\
	&=\sum^\infty_{\ell=n}\int_{\Delta_{\ell-n+1}(t_0,T)}\Bigl\{\sum^{\ell-n}_{k=0}(\fF_{\ell-n-k}[K^\top]\triangleright\bfF_k[J^\top])(s,t_0)\Bigr\}^\top\fF_\ell[\xi](s,t)\rd s\\
	&=\sum^\infty_{\ell=n}\int_{\Delta_{\ell-n+1}(t_0,T)}\bfF_{\ell-n}[(K^\top\star J^\top)^\top](s,t_0)\fF_\ell[\xi](s,t)\rd s\\
	&=\fF_n[(K^\top\star J^\top)^\top\bast\xi](t_0,t).
\end{align*}
Thus, \eqref{bast_eq_JKxi} holds.

Lastly, we prove \eqref{bast_eq_KJxi}. By \cref{bast_lemm_chaos} and \cref{ast_lemm_chaos}, for each $n\in\bN_0$ and $t\in\Delta_{n+1}(S,T)$,
\begin{align*}
	&\fF_n[K\bstar(J\bast\xi)](t)\\
	&=\sum^\infty_{k=n}\int_{\Delta_{k-n}(\max t,T)}\fF_{k-n}[K](s,\max t)\fF_k[J\bast\xi](s,t)\rd s\\
	&=\sum^\infty_{k=n}\int_{\Delta_{k-n}(\max t,T)}\fF_{k-n}[K](s_0,s,\max t)\Bigl\{\sum^\infty_{\ell=k}\int_{\Delta_{\ell-k+1}(s_0,T)}\bfF_{\ell-k}[J](r,s_0)\fF_\ell[\xi](r,s,t)\rd r\Bigr\}\rd (s_0,s)\\
	&=\sum^\infty_{\ell=n}\int_{\Delta_{\ell-n}(\max t,T)}\Bigl\{\sum^\ell_{k=n}\int^{\min r}_{\max s}\bfF_{\ell-k}[J](r,s_0)^\top\fF_{k-n}[K](s_0,s,\max t)^\top\rd s_0\Bigr\}^\top\fF_\ell[\xi](r,s,t)\rd (r,s)\\
	&=\sum^\infty_{\ell=n}\int_{\Delta_{\ell-n}(\max t,T)}\Bigl\{\sum^\ell_{k=n}(\bfF_{\ell-k}[J^\top]\ast\fF_{k-n}[K^\top])(s,\max t)\Bigr\}^\top\fF_\ell[\xi](s,t)\rd s\\
	&=\sum^\infty_{\ell=n}\int_{\Delta_{\ell-n}(\max t,T)}\Bigl\{\sum^{\ell-n}_{k=0}(\bfF_{\ell-n-k}[J^\top]\ast\fF_k[K^\top])(s,\max t)\Bigr\}^\top\fF_\ell[\xi](s,t)\rd s\\
	&=\sum^\infty_{\ell=n}\int_{\Delta_{\ell-n}(\max t,T)}\fF_{\ell-n}[(J^\top\ast K^\top)^\top](s,\max t)\fF_\ell[\xi](s,t)\rd s\\
	&=\fF_n[(J^\top\ast K^\top)^\top\bstar\xi](t).
\end{align*}
Thus, \eqref{bast_eq_KJxi} holds, and we finish the proof.
\end{proof}

The following is a \emph{duality principle} with respect to the $\ast$-product and the backward $\ast$-product.


\begin{prop}\label{bast_prop_duality}
For each $\ast$-Volterra kernel $J\in\cJ_\bF(S,T;\bR^{d\times d})$ and $\varphi,\psi\in L^2_\bF(S,T;\bR^d)$, it holds that
\begin{equation*}
	\langle J\ast\varphi,\psi\rangle_{L^2_\bF(S,T)}=\langle\varphi,J^\top\bast\psi\rangle_{L^2_\bF(S,T)}.
\end{equation*}
\end{prop}


\begin{proof}
By using Fubini's theorem, we have
\begin{align*}
	\langle J\ast\varphi,\psi\rangle_{L^2_\bF(S,T)}&=\bE\Bigl[\int^T_S\langle(J\ast\varphi)(t),\psi(t)\rangle\rd t\Bigr]\\
	&=\bE\Bigl[\int^T_S\Bigl\langle\int^t_SJ(t,s)\varphi(s)\rd s,\psi(t)\Bigr\rangle\rd t\Bigr]\\
	&=\bE\Bigl[\int^T_S\Bigl\langle\varphi(s),\int^T_sJ(t,s)^\top\psi(t)\rd t\Bigr\rangle\rd s\Bigr]\\
	&=\bE\Bigl[\int^T_S\Bigl\langle\varphi(s),\bE_s\Bigl[\int^T_sJ(t,s)^\top\psi(t)\rd t\Bigr]\Bigr\rangle\rd s\Bigr]\\
	&=\bE\Bigl[\int^T_S\langle\varphi(s),(J^\top\bast\psi)(s)\rangle\rd s\Bigr]\\
	&=\langle\varphi,J^\top\bast\psi\rangle_{L^2_\bF(S,T)}.
\end{align*}
Thus, we get the assertion.
\end{proof}


\subsection{A variation of constants formula for generalized BSVIEs}\label{subsection_bVCF}

Now we show the variation of constants formula for generalized BSVIEs.


\begin{theo}\label{bVCF_theo_VCF}
Let a $\ast$-Volterra kernel $J\in\cJ_\bF(S,T;\bR^{d\times d})$ and a $\star$-Volterra kernel $K\in\cK_\bF(S,T;\bR^{d\times d})$ be fixed, and suppose that $(J,K)$ has a $(\ast,\star)$-resolvent $(Q,R)\in\cJ_\bF(S,T;\bR^{d\times d})\times\cK_\bF(S,T;\bR^{d\times d})$. Then for any free term $\psi\in L^2_\bF(S,T;\bR^d)$, the generalized BSVIE
\begin{equation*}
	Y=\psi+J^\top\bast Y+K^\top\bstar Y
\end{equation*}
has a unique solution $Y\in L^2_\bF(S,T;\bR^d)$. This solution is given by the variation of constants formula:
\begin{equation*}
	Y=\psi+Q^\top\bast\psi+R^\top\bstar\psi.
\end{equation*}
\end{theo}


\begin{proof}
Define $Y\in L^2_\bF(S,T;\bR^d)$ by $Y=\psi+Q^\top\bast\psi+R^\top\bstar\psi$. By \cref{bstar_prop_bstaralg}, \cref{bast_prop_alg} and the resolvent equations, we have
\begin{align*}
	J^\top\bast Y&=J^\top\bast(\psi+Q^\top\bast\psi+R^\top\bstar\psi)\\
	&=J^\top\bast\psi+(Q\ast J)^\top\bast\psi+(R\star J)^\top\bast\psi\\
	&=(J+Q\ast J+R\star J)^\top\bast\psi\\
	&=Q^\top\bast\psi
\end{align*}
and
\begin{align*}
	K^\top\bstar Y&=K^\top\bstar(\psi+Q^\top\bast\psi+R^\top\bstar\psi)\\
	&=K^\top\bstar\psi+(Q\ast K)^\top\bstar\psi+(R\star K)^\top\bstar\psi\\
	&=(K+Q\ast K+R\star K)^\top\bstar\psi\\
	&=R^\top\bstar\psi.
\end{align*}
Thus, we get $Y=\psi+J^\top\bast Y+K^\top\bstar Y$.

Conversely, if $Y\in L^2_\bF(S,T;\bR^d)$ satisfies $Y=\psi+J^\top\bast Y+K^\top\bstar Y$, then we have
\begin{align*}
	Q^\top\bast\psi&=Q^\top\bast(Y-J^\top\bast Y-K^\top\bstar Y)\\
	&=Q^\top\bast Y-(J\ast Q)^\top\bast Y-(K\star Q)^\top\bast Y\\
	&=(Q-J\ast Q-K\star Q)^\top\bast Y\\
	&=J^\top\bast Y
\end{align*}
and
\begin{align*}
	R^\top\bstar\psi&=R^\top\bstar(Y-J^\top\bast Y-K^\top\bstar Y)\\
	&=R^\top\bstar Y-(J\ast R)^\top\bstar Y-(K\star R)^\top\bstar Y\\
	&=(R-J\ast R-K\star R)^\top\bstar Y\\
	&=K^\top\bstar Y.
\end{align*}
Thus, we get $Y=\psi+Q^\top\bast\psi+R^\top\bstar\psi$. This completes the proof.
\end{proof}

From the above result, together with the stochastic integral representation of the backward $\star$-product (see \cref{bstar_prop_integral}), we immediately get the following corollary.


\begin{cor}\label{bVCF_cor_VCF1}
Let $J\in \cJ_\bF(S,T;\bR^{d\times d})$ and $K=\sum^\infty_{n=1}\fW_n[k_n]\in\cK_\bF(S,T;\bR^{d\times d})$ with $k_n\in\cV_{n+1}(S,T;\bR^{d\times d})$, $n\in\bN$, be fixed, and suppose that $(J,K)$ has a $(\ast,\star)$-resolvent $(Q,R)\in\cJ_\bF(S,T;\bR^{d\times d})\times\cK_\bF(S,T;\bR^{d\times d})$. Then for any free term $\psi\in L^2_\bF(S,T;\bR^d)$, the generalized BSVIE
\begin{equation*}
	Y(t)=\psi(t)+\bE_t\Bigl[\int^T_tJ(s,t)^\top Y(s)\rd s\Bigr]+\sum^\infty_{n=1}\int_{\Delta_n(t,T)}k_n(s,t)^\top\cM_n[Y](s,t)\rd s,\ t\in(S,T),
\end{equation*}
has a unique solution $Y\in L^2_\bF(S,T;\bR^d)$. This solution is given by the variation of constants formula:
\begin{equation*}
	Y(t)=\psi(t)+\bE_t\Bigl[\int^T_tQ(s,t)^\top \psi(s)\rd s\Bigr]+\sum^\infty_{n=1}\int_{\Delta_n(t,T)}r_n(s,t)^\top\cM_n[\psi](s,t)\rd s,\ t\in(S,T),
\end{equation*}
where $r_n:=\fF_n[R]\in\cV_{n+1}(S,T;\bR^{d\times d})$ for $n\in\bN$. Here, the infinite sum in the right-hand side converges in $L^2_\bF(S,T;\bR^d)$.
\end{cor}


\begin{cor}\label{bVCF_cor_VCF2}
Let deterministic kernels $j\in L^2(\Delta_2(S,T);\bR^{d\times d})$ and $k\in\cV_2(S,T;\bR^{d\times d})$ be fixed. Let $q\in L^2(\Delta_2(S,T);\bR^{d\times d})$ be the $\ast$-resolvent of $j$, and assume that $\fW_1[k+q\ast k]$ has a $\star$-resolvent. Then for any free term $\psi\in\ L^2_\bF(S,T;\bR^d)$, the Type-II BSVIE
\begin{equation*}
	Y(t)=\psi(t)+\bE_t\Bigl[\int^T_tj(s,t)^\top Y(s)\rd s\Bigr]+\int^T_tk(s,t)^\top\cM_1[Y](s,t)\rd s,\ t\in(S,T),
\end{equation*}
has a unique solution $Y\in L^2_\bF(S,T;\bR^d)$. This solution is given by the variation of constants formula:
\begin{align*}
	Y(t)&=\psi(t)+\sum^\infty_{n=1}\int_{\Delta_n(t,T)}(k+q\ast k)^{\triangleright n}(r,t)^\top\cM_n[\psi](r,t)\rd r\\
	&\hspace{0.3cm}+\int^T_tq(s,t)^\top\bE_t\Bigl[\psi(s)+\sum^\infty_{n=1}\int_{\Delta_n(s,T)}(k+q\ast k)^{\triangleright n}(r,s)^\top\cM_n[\psi](r,s)\rd r\Bigr]\rd s,\ t\in(S,T).
\end{align*}
\end{cor}


\begin{proof}
By the arguments in the proof of \cref{fVCF_cor_VCF2}, we see that $(Q,R)$ defined by $R=\sum^\infty_{n=1}\fW_n[(k+q\ast k)^{\triangleright n}]\in\cK_\bF(S,T;\bR^{d\times d})$ and $Q=q+R\star q\in\cJ_\bF(S,T;\bR^{d\times d})$ is the $(\ast,\star)$-resolvent of $(j,\fW_1[k])$. Thus, by \cref{bVCF_theo_VCF}, the BSVIE $Y=\psi+j^\top\bast Y+\fW_1[k]^\top\bstar Y$ has a unique solution $Y\in L^2_\bF(S,T;\bR^d)$, and the solution is given by
\begin{align*}
	Y&=\psi+Q^\top\bast\psi+R^\top\bstar\psi\\
	&=\psi+(q+R\star q)^\top\bast\psi+R^\top\bstar\psi\\
	&=\psi+q^\top\bast\psi+q^\top\bast(R^\top\bstar\psi)+R^\top\bstar\psi\\
	&=\psi+R^\top\bstar\psi+q^\top\bast(\psi+R^\top\bstar\psi).
\end{align*}
Here, we used \cref{bast_prop_alg}. Noting \cref{bstar_prop_integral}, we get the assertion.
\end{proof}

By the variation of constants formulae, together with the duality principles for the $\star$- and the $\ast$-products (see \cref{bstar_prop_duality} and \cref{bast_prop_duality}), we get the following duality principle between generalized SVIEs and generalized BSVIEs.


\begin{theo}\label{bVCF_theo_duality}
Let a $\ast$-Volterra kernel $J\in\cJ_\bF(S,T;\bR^{d\times d})$ and a $\star$-Volterra kernel $K\in\cK_\bF(S,T;\bR^{d\times d})$ be fixed, and suppose that $(J,K)$ has a $(\ast,\star)$-resolvent $(Q,R)\in\cJ_\bF(S,T;\bR^{d\times d})\times\cK_\bF(S,T;\bR^{d\times d})$. For each $\varphi,\psi\in L^2_\bF(S,T;\bR^d)$, let $X,Y\in L^2_\bF(S,T;\bR^d)$ be the solutions of the generalized SVIE
\begin{equation*}
	X=\varphi+J\ast X+K\star X
\end{equation*}
and the generalized BSVIE
\begin{equation*}
	Y=\psi+J^\top\bast Y+K^\top\bstar Y,
\end{equation*}
respectively. Then it holds that
\begin{equation*}
	\langle X,\psi\rangle_{L^2_\bF(S,T)}=\langle\varphi,Y\rangle_{L^2_\bF(S,T)}.
\end{equation*}
\end{theo}


\begin{proof}
By \cref{fVCF_theo_VCF} and \cref{bVCF_theo_VCF}, $X$ and $Y$ are given by $X=\varphi+Q\ast\varphi+R\star\varphi$ and $Y=\psi+Q^\top\bast\psi+R^\top\bstar\psi$, respectively. Therefore, by \cref{bstar_prop_duality} and \cref{bast_prop_duality}, we get
\begin{align*}
	\langle X,\psi\rangle_{L^2_\bF(S,T)}=\langle\varphi+Q\ast\varphi+R\star\varphi,\psi\rangle_{L^2_\bF(S,T)}=\langle\varphi,\psi+Q^\top\bast\psi+R^\top\bstar\psi\rangle_{L^2_\bF(S,T)}=\langle\varphi,Y\rangle_{L^2_\bF(S,T)}.
\end{align*}
Thus, the assertion holds.
\end{proof}


\section{Existence of the $\star$-resolvent}\label{section_exist}

As we have seen in the above sections, the $(\ast,\star)$-resolvent plays crucial roles for solving SVIEs and BSVIEs. Recall that every $\ast$-Volterra kernel has a unique $\ast$-resolvent (see \cref{res_prop_astresexist}). Also, the $(\ast,\star)$-resolvent can be constructed by the $\ast$-resolvent and the $\star$-resolvent (see \cref{ast_prop_aststarres}). In this section, we focus on the existence of the $\star$-resolvent. We borrow some ideas from Section 9.3 of the textbook \cite{GrLoSt90}, where the general theory on deterministic Volterra kernels is discussed.

First, we show the following lemma.


\begin{lemm}\label{exist_lemm_<1}
Every $\star$-Volterra kernel $K\in\cK_\bF(S,T;\bR^{d\times d})$ with $\knorm K\knorm_{\cK_\bF(S,T)}<1$ has a $\star$-resolvent $R\in\cK_\bF(S,T;\bR^{d\times d})$. Furthermore, it holds that $R=\sum^\infty_{n=1}K^{\star n}$, where $K^{\star n}$ denotes the $(n-1)$-fold $\star$-product of $K$ by itself, and the infinite sum converges in $\cK_\bF(S,T;\bR^{d\times d})$.
\end{lemm}


\begin{proof}
By \cref{star_prop_Banachalg}, we have $\knorm K^{\star n}\knorm_{\cK_\bF(S,T)}\leq\knorm K\knorm^n_{\cK_\bF(S,T)}$ for any $n\in\bN$. Thus, by the assumption, the infinite sum $R:=\sum^\infty_{n=1}K^{\star n}$ converges in $\cK_\bF(S,T;\bR^{d\times d})$. Observe that
\begin{equation*}
	K+K\star R=K+K\star\Bigl(\sum^\infty_{n=1}K^{\star n}\Bigr)=K+\sum^\infty_{n=1}K\star K^{\star n}=\sum^\infty_{n=1}K^{\star n}=R.
\end{equation*}
Similarly, we have $K+R\star K=R$. Thus, $R$ is the $\star$-resolvent of $K$.
\end{proof}

The above lemma is not so useful by its own right for the existence of the $\star$-resolvent on the whole interval $(S,T)$ with arbitrary length. In order to ensure the existence of the $\star$-resolvent in more general settings, we introduce the notion of the restriction and the concatenation of $\star$-Volterra kernels with respect to subintervals of $(S,T)$.


\begin{defi}\label{exist_defi_restrict}
For each $\star$-Volterra kernel $K\in\cK_\bF(S,T;\bR^{d\times d})$ and a subinterval $(U,V)$ of $(S,T)$, we define the \emph{restricted $\star$-Volterra kernel} $K_{(U,V)}\in\cK_\bF(U,V;\bR^{d\times d})$ of $K$ on $(U,V)$ by the Wiener--It\^{o} chaos expansion
\begin{equation*}
	\fF_n[K_{(U,V)}]:=\fF_n[K]|_{\Delta_{n+1}(U,V)},\ n\in\bN_0,
\end{equation*}
where, for each function $f:\Delta_{n+1}(S,T)\to\bR^{d\times d}$ with $n\in\bN_0$, the function $f|_{\Delta_{n+1}(U,V)}:\Delta_{n+1}(U,V)\to\bR^{d\times d}$ is the restriction of $f$ on $\Delta_{n+1}(U,V)$.
\end{defi}


\begin{rem}
\begin{itemize}
\item[(i)]
Note that it does not hold $K_{(U,V)}(t)=K(t)$ for $t\in(U,V)$ in general.
\item[(ii)]
Let $0\leq S\leq U<V\leq T<\infty$. Suppose that a $\star$-Volterra kernel $K\in\cK_\bF(S,T;\bR^{d\times d})$ has a $\star$-resolvent $R\in\cK_\bF(S,T;\bR^{d\times d})$ on $(S,T)$. By the definition, it is easy to see that the restricted $\star$-Volterra kernel $K_{(U,V)}\in\cK_\bF(U,V;\bR^{d\times d})$ has a $\star$-resolvent on $(U,V)$, which is given by the restricted $\star$-Volterra kernel $R_{(U,V)}\in\cK_\bF(U,V;\bR^{d\times d})$ of $R$.
\end{itemize}
\end{rem}

We prove the opposite direction of (ii) in the above remark. In other words, we provide a (non-trivial) concatenation procedure for the $\star$-resolvent. The following is the main result of this section.


\begin{theo}\label{exist_theo_concate}
Let $K\in\cK_\bF(S,T;\bR^{d\times d})$ be a $\star$-Volterra kernel on $(S,T)$ with $0\leq S<T<\infty$. Suppose that there exists $\{U_k\}^m_{k=0}$ with $m\in\bN$ such that $S=U_0<U_1<\mathalpha{\cdots}<U_m=T$ and that, for every $k\in\{0,1,\mathalpha{\dots},m-1\}$, the restricted $\star$-Volterra kernel $K_{(U_k,U_{k+1})}\in\cK_\bF(U_k,U_{k+1};\bR^{d\times d})$ has a $\star$-resolvent on $(U_k,U_{k+1})$. Then $K$ has a $\star$-resolvent on $(S,T)$.
\end{theo}


\begin{proof}
It suffices to show the theorem when $m=2$. Let $S<U<T$ be fixed, and assume that $K_{(S,U)}\in\cK_\bF(S,U;\bR^{d\times d})$ (resp.\ $K_{(U,T)}\in\cK_\bF(U,T;\bR^{d\times d})$) has a $\star$-resolvent $R_{(S,U)}\in\cK_\bF(S,U;\bR^{d\times d})$ on $(S,U)$ (resp.\ $R_{(U,T)}\in\cK_\bF(U,T;\bR^{d\times d})$ on $(U,T)$). Define $\star$-Volterra kernels $K_i,R_i\in\cK_\bF(S,T;\bR^{d\times d})$, $i=1,2$, on the whole interval $(S,T)$ by the Wiener--It\^{o} chaos expansions
\begin{align*}
	&\fF_n[K_1]:=\fF_n[K_{(S,U)}]\ \text{on}\ \Delta_{n+1}(S,U),\ \fF_n[K_1]:=0\ \text{on}\ \Delta_{n+1}(S,T)\setminus\Delta_{n+1}(S,U),\ n\in\bN_0,\\
	&\fF_n[K_2]:=\fF_n[K_{(U,T)}]\ \text{on}\ \Delta_{n+1}(U,T),\ \fF_n[K_2]:=0\ \text{on}\ \Delta_{n+1}(S,T)\setminus\Delta_{n+1}(U,T),\ n\in\bN_0,\\
	&\fF_n[R_1]:=\fF_n[R_{(S,U)}]\ \text{on}\ \Delta_{n+1}(S,U),\ \fF_n[R_1]:=0\ \text{on}\ \Delta_{n+1}(S,T)\setminus\Delta_{n+1}(S,U),\ n\in\bN_0,\\
	&\fF_n[R_2]:=\fF_n[R_{(U,T)}]\ \text{on}\ \Delta_{n+1}(U,T),\ \fF_n[R_2]:=0\ \text{on}\ \Delta_{n+1}(S,T)\setminus\Delta_{n+1}(U,T),\ n\in\bN_0.
\end{align*}
Define $R\in\cK_\bF(S,T;\bR^{d\times d})$ by
\begin{equation*}
	R:=K+K\star R_1+R_2\star K+R_2\star K\star R_1.
\end{equation*}
We will show that $R$ is the $\star$-resolvent of $K$ on $(S,T)$.

First, we show that $R_i$ is the $\star$-resolvent of $K_i$ on $(S,T)$ for $i=1,2$. Observe that, on $\Delta_{n+1}(S,T)$ with $n\in\bN_0$,
\begin{align*}
	\fF_n[K_1\star R_1]&=\sum^n_{k=0}\fF_{n-k}[K_1]\triangleright\fF_k[R_1]=\Bigl(\sum^n_{k=0}\fF_{n-k}[K_{(S,U)}]\triangleright\fF_k[R_{(S,U)}]\Bigr)\1_{\Delta_{n+1}(S,U)}\\
	&=\fF_n[K_{(S,U)}\star R_{(S,U)}]\1_{\Delta_{n+1}(S,U)}=\fF_n[R_{(S,U)}-K_{(S,U)}]\1_{\Delta_{n+1}(S,U)}\\
	&=\fF_n[R_1-K_1],
\end{align*}
and hence $R_1=K_1+K_1\star R_1$ on $(S,T)$. Similarly, we can show that $R_1=K_1+R_1\star K_1$ and $R_2=K_2+K_2\star R_2=K_2+R_2\star K_2$ on $(S,T)$. Thus, $R_i$ is the $\star$-resolvent of $K_i$ on $(S,T)$ for $i=1,2$.

Second, we show the following claim:
\begin{equation}\label{exist_eq_concateclaim1}
	R=K+K\star R_1+K_2\star R=K+R_2\star K+R\star K_1.
\end{equation}
Since $R_2$ is the $\star$-resolvent of $K_2$, we have
\begin{align*}
	K_2\star R&=K_2\star K+K_2\star K\star R_1+\underbrace{K_2\star R_2}_{=R_2-K_2}\star K+\underbrace{K_2\star R_2}_{=R_2-K_2}\star K\star R_1\\
	&=R_2\star K+R_2\star K\star R_1.
\end{align*}
By inserting this formula into the definition of $R$, we get the first equality in \eqref{exist_eq_concateclaim1}. On the other hand, since $R_1$ is the $\star$-resolvent of $K_1$, we have
\begin{align*}
	R\star K_1&=K\star K_1+K\star \underbrace{R_1\star K_1}_{=R_1-K_1}+R_2\star K\star K_1+R_2\star K\star \underbrace{R_1\star K_1}_{=R_1-K_1}\\
	&=K\star R_1+R_2\star K\star R_1.
\end{align*}
By inserting this formula into the definition of $R$, we get the second equality in \eqref{exist_eq_concateclaim1}.

Third, let us show the following claim:
\begin{equation}\label{exist_eq_concateclaim2}
\begin{split}
	&\fF_n[K]=\fF_n[K_1]\ \text{on}\ \Delta_{n+1}(S,U)\ \text{and}\ \fF_n[K]=\fF_n[K_2]\ \text{on}\ \Delta_{n+1}(U,T),\\
	&\fF_n[R]=\fF_n[R_1]\ \text{on}\ \Delta_{n+1}(S,U)\ \text{and}\ \fF_n[R]=\fF_n[R_2]\ \text{on}\ \Delta_{n+1}(U,T),\\
\end{split}
\end{equation}
for each $n\in\bN_0$. The equalities for $K$ are trivial from the definition. We show the equalities for $R$. Let $n\in\bN_0$. Observe that, on $\Delta_{n+1}(S,U)$,
\begin{align*}
	&\fF_n[K]=\fF_n[K_1],\\
	&\fF_n[K\star R_1]=\sum^n_{k=0}\fF_{n-k}[K]\triangleright\fF_k[R_1]=\sum^n_{k=0}\fF_{n-k}[K_1]\triangleright\fF_k[R_1]=\fF_n[K_1\star R_1],\\
	&\fF_n[R_2\star K]=\sum^n_{k=0}\fF_{n-k}[R_2]\triangleright\fF_k[K]=0,\\
	&\fF_n[R_2\star K\star R_1]=\sum^n_{k=0}\fF_{n-k}[R_2]\triangleright\fF_k[K\star R_1]=0,
\end{align*}
and hence
\begin{equation*}
	\fF_n[R]=\fF_n[K_1]+\fF_n[K_1\star R_1]=\fF_n[K_1+K_1\star R_1]=\fF_n[R_1].
\end{equation*}
On the other hand, on $\Delta_{n+1}(U,T)$,
\begin{align*}
	&\fF_n[K]=\fF_n[K_2],\\
	&\fF_n[K\star R_1]=\sum^n_{k=0}\fF_{n-k}[K]\triangleright\fF_k[R_1]=0,\\
	&\fF_n[R_2\star K]=\sum^n_{k=0}\fF_{n-k}[R_2]\triangleright\fF_k[K]=\sum^n_{k=0}\fF_{n-k}[R_2]\triangleright\fF_k[K_2]=\fF_n[R_2\star K_2],\\
	&\fF_n[R_2\star K\star R_1]=\sum^n_{k=0}\fF_{n-k}[R_2\star K]\triangleright\fF_k[R_1]=0,
\end{align*}
and hence
\begin{equation*}
	\fF_n[R]=\fF_n[K_2]+\fF_n[R_2\star K_2]=\fF_n[K_2+R_2\star K_2]=\fF_n[R_2].
\end{equation*}
Thus, the claim \eqref{exist_eq_concateclaim2} holds.

Lastly, we show the following claim:
\begin{equation}\label{exist_eq_concateclaim3}
	K\star R_1+K_2\star R=K\star R\ \text{and}\ R_2\star K+R\star K_1=R\star K.
\end{equation}
By using \eqref{exist_eq_concateclaim2}, for each $n\in\bN_0$, we have
\begin{align*}
	\fF_n[K\star R_1]&=\sum^n_{k=0}\fF_{n-k}[K]\triangleright\fF_k[R_1]=\sum^n_{k=0}\fF_{n-k}[K]\triangleright(\fF_k[R_1]\1_{\Delta_{k+1}(S,U)})\\
	&=\sum^n_{k=0}\fF_{n-k}[K]\triangleright(\fF_k[R]\1_{\Delta_{k+1}(S,U)})
\end{align*}
and, for a.e.\ on $\Delta_{n+1}(S,T)$,
\begin{align*}
	\fF_n[K_2\star R]&=\sum^n_{k=0}\fF_{n-k}[K_2]\triangleright\fF_k[R]=\sum^n_{k=0}(\fF_{n-k}[K_2]\1_{\Delta_{n-k+1}(U,T)})\triangleright\fF_k[R]\\
	&=\sum^n_{k=0}(\fF_{n-k}[K]\1_{\Delta_{n-k+1}(U,T)})\triangleright\fF_k[R]\\
	&=\sum^n_{k=0}\fF_{n-k}[K]\triangleright(\fF_k[R]\1_{\Delta_{k+1}(S,T)\setminus\Delta_{k+1}(S,U)}).
\end{align*}
Note that the last equality holds except on the set $\bigcup^n_{k=0}\{(t_0,t_1,\mathalpha{\dots},t_n)\in\Delta_{n+1}(S,T)\,|\,t_k=U\}$ whose Lebesgue measure is zero on $\Delta_{n+1}(S,T)$. Thus, we get
\begin{equation*}
	\fF_n[K\star R_1+K_2\star R]=\sum^n_{k=0}\fF_{n-k}[K]\triangleright\fF_k[R]=\fF_n[K\star R]\ \text{a.e.\ on $\Delta_{n+1}(S,T)$}
\end{equation*}
for each $n\in\bN_0$. Therefore, the first equality in \eqref{exist_eq_concateclaim3} holds. Similarly, for each $n\in\bN_0$, we have
\begin{align*}
	\fF_n[R_2\star K]&=\sum^n_{k=0}\fF_{n-k}[R_2]\triangleright\fF_k[K]=\sum^n_{k=0}(\fF_{n-k}[R_2]\1_{\Delta_{n-k+1}(U,T)})\triangleright\fF_k[K]\\
	&=\sum^n_{k=0}(\fF_{n-k}[R]\1_{\Delta_{n-k+1}(U,T)})\triangleright\fF_k[K]
\end{align*}
and, for a.e.\ on $\Delta_{n+1}(S,T)$,
\begin{align*}
	\fF_n[R\star K_1]&=\sum^n_{k=0}\fF_{n-k}[R]\triangleright\fF_k[K_1]=\sum^n_{k=0}\fF_{n-k}[R]\triangleright(\fF_k[K_1]\1_{\Delta_{k+1}(S,U)})\\
	&=\sum^n_{k=0}\fF_{n-k}[R]\triangleright(\fF_k[K]\1_{\Delta_{k+1}(S,U)})\\
	&=\sum^n_{k=0}(\fF_{n-k}[R]\1_{\Delta_{n-k+1}(S,T)\setminus\Delta_{n-k+1}(U,T)})\triangleright\fF_k[K].
\end{align*}
Thus, we get
\begin{equation*}
	\fF_n[R_2\star K+R\star K_1]=\sum^n_{k=0}\fF_{n-k}[R]\triangleright\fF_k[K]=\fF_n[R\star K]\ \text{a.e.\ on $\Delta_{n+1}(S,T)$}
\end{equation*}
for each $n\in\bN_0$. Therefore, the second equality in \eqref{exist_eq_concateclaim3} holds.

By \eqref{exist_eq_concateclaim1} and \eqref{exist_eq_concateclaim3}, we get $R=K+K\star R=K+R\star K$, which implies that $R$ is the $\star$-resolvent of $K$ on $(S,T)$. This completes the proof.
\end{proof}

By \cref{exist_lemm_<1} and \cref{exist_theo_concate}, we get the following important corollaries.


\begin{cor}\label{exist_cor_concate}
Let $K\in\cK_\bF(S,T;\bR^{d\times d})$ be a $\star$-Volterra kernel. Suppose that there exists $\{U_k\}^m_{k=0}$ with $m\in\bN$ such that $S=U_0<U_1<\mathalpha{\cdots}<U_m=T$ and that, for every $k\in\{0,1,\mathalpha{\dots},m-1\}$, the restricted $\star$-Volterra kernel $K_{(U_k,U_{k+1})}\in\cK_\bF(U_k,U_{k+1};\bR^{d\times d})$ satisfies $\knorm K_{(U_k,U_{k+1})}\knorm_{\cK_\bF(U_k,U_{k+1})}<1$. Then $K$ has a $\star$-resolvent $R\in\cK_\bF(S,T;\bR^{d\times d})$ on $(S,T)$.
\end{cor}


\begin{cor}\label{exist_cor_concateestimate}
Let $K\in\cK_\bF(S,T;\bR^{d\times d})$ be a $\star$-Volterra kernel. Assume that
\begin{equation*}
	\underset{t\in(S,T)}{\mathrm{ess\,sup}}|\fF_0[K](t)|_\op<1
\end{equation*}
and that there exists a nonnegative function $a\in L^2(0,T-S;\bR)$ satisfying
\begin{equation*}
	|\fF_n[K](t_0,t_1,\mathalpha{\dots},t_n)|_\op\leq\prod^n_{i=1}a(t_{i-1}-t_i),\ (t_0,t_1,\mathalpha{\dots},t_n)\in\Delta_{n+1}(S,T),
\end{equation*}
for any $n\in\bN$. Then $K$ has a $\star$-resolvent $R\in\cK_\bF(S,T;\bR^{d\times d})$ on $(S,T)$.
\end{cor}


\begin{proof}
We denote $\delta:=\underset{t\in(S,T)}{\mathrm{ess\,sup}}|\fF_0[K](t)|_\op<1$. Let $\ep\in(0,1)$ be a number such that $\delta+\frac{\ep}{1-\ep}<1$. Then there exists $\{U_k\}^m_{k=0}$ with $m\in\bN$ such that $S=U_0<U_1<\mathalpha{\cdots}<U_m=T$ and that $\|a\|_{L^2(0,U_{k+1}-U_k)}\leq\ep$ for every $k\in\{0,1,\mathalpha{\dots},m-1\}$. Fix an arbitrary $k\in\{0,1,\mathalpha{\dots},m-1\}$. Note that
\begin{equation*}
	\knorm \fF_0[K_{(U_k,U_{k+1})}]\knorm_{\cV_1(U_k,U_{k+1})}=\underset{t\in(U_k,U_{k+1})}{\mathrm{ess\,sup}}|\fF_0[K](t)|_\op\leq\delta.
\end{equation*}
Furthermore, noting \cref{detkernel_lemm_kernel}, we see that
\begin{equation*}
	\knorm \fF_n[K_{(U_k,U_{k+1})}]\knorm_{\cV_{n+1}(U_k,U_{k+1})}\leq\|a\|^n_{L^2(0,U_{k+1}-U_k)}\leq\ep^n
\end{equation*}
for any $n\in\bN$. Thus, we have
\begin{equation*}
	\knorm K_{(U_k,U_{k+1})}\knorm_{\cK_\bF(U_k,U_{k+1})}=\sum^\infty_{n=0}\knorm\fF_n[K_{(U_k,U_{k+1})}]\knorm_{\cV_{n+1}(U_k,U_{k+1})}\leq\delta+\frac{\ep}{1-\ep}<1.
\end{equation*}
By \cref{exist_cor_concate}, we get the assertion.
\end{proof}


\begin{cor}\label{exist_cor_concateestimate2}
Let $K\in\cK_\bF(S,T;\bR^{d\times d})$ be a $\star$-Volterra kernel. Assume that
\begin{equation*}
	\underset{t\in(S,T)}{\mathrm{ess\,sup}}|\fF_0[K](t)|_\op<1
\end{equation*}
and that
\begin{equation*}
	\sup_{n\in\bN}\Bigl\{(n!)^{-\frac{p-2}{2p}}\underset{t\in(S,T)}{\mathrm{ess\,sup}}\Bigl(\int_{\Delta_n(t,T)}|\fF_n[K](s,t)|^p_\op\rd s\Bigr)^{1/p}\Bigr\}^{1/n}<\infty
\end{equation*}
for some $p>2$. Then $K$ has a $\star$-resolvent $R\in\cK_\bF(S,T;\bR^{d\times d})$ on $(S,T)$.
\end{cor}


\begin{proof}
We denote $\delta:=\underset{t\in(S,T)}{\mathrm{ess\,sup}}|\fF_0[K](t)|_\op<1$ and
\begin{equation*}
	C_p:=\sup_{n\in\bN}\Bigl\{(n!)^{-\frac{p-2}{2p}}\underset{t\in(S,T)}{\mathrm{ess\,sup}}\Bigl(\int_{\Delta_n(t,T)}|\fF_n[K](s,t)|^p_\op\rd s\Bigr)^{1/p}\Bigr\}^{1/n}<\infty.
\end{equation*}
Let $\ep>0$ be a number such that
\begin{equation*}
	C_p\ep^{\frac{(p-2)}{2p}}<1\ \text{and}\ \delta+\frac{C_p\ep^{\frac{(p-2)}{2p}}}{1-C_p\ep^{\frac{(p-2)}{2p}}}<1.
\end{equation*}
Take $\{U_k\}^m_{k=0}$ with $m\in\bN$ such that $S=U_0<U_1<\mathalpha{\cdots}<U_m=T$ and that $U_{k+1}-U_k<\ep$ for every $k\in\{0,\dots,m-1\}$. Let $k\in\{0,\dots,m-1\}$ be fixed. The restricted $\star$-Volterra kernel $K_{(U_k,U_{k+1})}\in\cK_\bF(U_k,U_{k+1};\bR^{d\times d})$ satisfies
\begin{equation*}
	\knorm \fF_0[K_{(U_k,U_{k+1})}]\knorm_{\cV_1(U_k,U_{k+1})}=\underset{t\in(U_k,U_{k+1})}{\mathrm{ess\,sup}}|\fF_0[K](t)|_\op\leq\delta
\end{equation*}
and, by H\"{o}lder's inequality,
\begin{align*}
	&\knorm\fF_n[K_{(U_k,U_{k+1})}]\knorm_{\cV_{n+1}(U_k,U_{k+1})}\\
	&=\underset{t\in(U_k,U_{k+1})}{\mathrm{ess\,sup}}\Bigl(\int_{\Delta_n(t,U_{k+1})}|\fF_n[K](s,t)|^2_\op\rd s\Bigr)^{1/2}\\
	&\leq\Bigl(\frac{(U_{k+1}-U_k)^n}{n!}\Bigr)^{\frac{(p-2)}{2p}}\underset{t\in(S,T)}{\mathrm{ess\,sup}}\Bigl(\int_{\Delta_n(t,T)}|\fF_n[K](s,t)|^p_\op\rd s\Bigr)^{1/p}\\
	&\leq (C_p\ep^{\frac{(p-2)}{2p}})^n.
\end{align*}
for each $n\in\bN$. Thus, we have
\begin{equation*}
	\knorm K_{(U_k,U_{k+1})}\knorm_{\cK_\bF(U_k,U_{k+1})}=\sum^\infty_{n=0}\knorm\fF_n[K_{(U_k,U_{k+1})}]\knorm_{\cV_{n+1}(U_k,U_{k+1})}\leq\delta+\frac{C_p\ep^{\frac{(p-2)}{2p}}}{1-C_p\ep^{\frac{(p-2)}{2p}}}<1.
\end{equation*}
By \cref{exist_cor_concate}, we get the assertion.
\end{proof}


\section*{Acknowledgments}
The author was supported by JSPS KAKENHI Grant Number 22K13958.



\begin{thebibliography}{99}

%
%
\bibitem{Ag19}
{N. Agram}.
{Dynamic risk measure for BSVIE with jumps and semimartingale issues}.
{\it Stoch. Anal. Appl.},
{\bf 37}(3),
361--376,
2019.


\bibitem{AlNu97}
{E. Al\`{o}s} and {D. Nualart}.
{Anticipating stochastic Volterra equations}.
{\it Stochastic Process. Appl.},
{\bf 72},
73--95,
1997.
%
%
%
%
%
%
\bibitem{BeMi80}
{M. A. Berger} and {V. J. Mizel}.
{Volterra equations with It\^o integrals---I}.
{\it J. Integral Equations},
{\bf 2}(3),
187--245,
1980.

\bibitem{BeMi80+}
{M. A. Berger} and {V. J. Mizel}.
{Volterra equations with It\^o integrals---II}.
{\it J. Integral Equations},
{\bf 2}(4),
319--337,
1980.
%
\bibitem{BeMi82}
{M. A. Berger} and {V. J. Mizel}.
{An extension of the stochastic integral}.
{\it Ann. Probab.},
{\bf 10}(2),
435--450,
1982.

%

%

%
\bibitem{ChYo07}
{S. Chen} and {J. Yong}.
{A linear quadratic optimal control problem for stochastic Volterra integral equations}.
{\it Control theory and related topics – in memory of professor Xunjing Li}, Fudan university, China,
44--66,
2007.

%

\bibitem{DaMoOkRo16}
{K. Dahl}, {S. E. A. Mohammed}, {B. {\O}ksendal}, and {E. E. R{\o}se}.
{Optimal control of systems with noisy memory and BSDEs with Malliavin derivatives}.
{\it J. Funct. Anal.}
{\bf 271},
289--329,
2016.

\bibitem{DaBa10}
{M. Dalir} and {M. Bashour}.
{Applications of Fractional Calculus}.
{\it Appl. Math. Sci.},
{\bf 4},
1021--1032,
2010.

\bibitem{DiGi84}
{O. Diekmann} and {S. A. Van Gils}.
{Invariant manifolds for Volterra integral equations of convolution type}.
{J. Differential Equations},
{\bf 54}
139--180,
1984.

\bibitem{Di07}
{K. Diethelm}.
{\it The Analysis of Fractional Differential Equations}.
Springer, New York,
2007.
%
%
%
%
%
%

\bibitem{GrLoSt90}
{G. Gripenberg}, {S. O. Londen}, and {O. Staffans}.
{\it Volterra Integral and Functional Equations},
volume 34 of {\it Encyclopedia of Mathematics and its Applications}.
Cambridge University Press, Cambridge,
1990.
%
%
\bibitem{Ha21}
{Y. Hamaguchi}.
{Extended backward stochastic Volterra integral equations and their applications to time-inconsistent stochastic recursive control problems}.
{\it Math. Control Relat. Fields},
{\bf 11},
197--242,
2021.

\bibitem{Ha21+}
{Y. Hamaguchi}.
{Infinite horizon backward stochastic Volterra integral equations and discounted control problems}.
{\it ESAIM Control Optim. Calc. Var.},
{\bf 27}(101),
47 pages,
2021.

\bibitem{Ha21++}
{Y. Hamaguchi}.
{On the maximum principle for optimal control problems of stochastic Volterra integral equations with delay}.
{\it preprint}.
arXiv:2109.06092.

\bibitem{HaLiYo22}
{S. Han}, {P. Lin}, and {J. Yong}.
{Causal state feedback representation for linear quadratic optimal control problems of singular Volterra integral equations}.
{\it Math. Control Relat. Fields},
doi: 10.3934/mcrf.2022038,
2022.



%
%
%
%
%
%
\bibitem{HuOk19}
{Y. Hu} and {B. {\O}ksendal}.
{Linear Volterra backward stochastic integral equations}.
{\it Stochastic Process. Appl.},
{\bf 129}(2),
626--633,
2019.

\bibitem{It51}
{K. It\^{o}}.
{Multiple Wiener integral}.
{\it J. Math. Soc. Japan},
{\bf 3},
157--169,
1951.
%
%
%
\bibitem{KrOv17}
{E. Kromer} and {L. Overbeck}.
{Differentiability of BSVIEs and dynamic capital allocations}.
{\it Int. J. Theor. Appl. Finance},
{\bf 20}(7),
1--26,
2017.
%
\bibitem{Li02}
{J. Lin}.
{Adapted solution of a backward stochastic nonlinear Volterra integral equation}.
{\it Stoch. Anal. Appl.},
{\bf 20}(1),
165--183,
2002.
%
%
%
%
%



%


\bibitem{OkZh93}
{B.~{\O}ksendal} and {T.~Zhang}.
{The stochastic Volterra equation}.
In: D.\ Nualart and M.\ Sanz Sole (editors), {\it Barcelona Seminar on Stochastic Analysis}, Birkhauser,
168--202,
1993.

\bibitem{OkZh96}
{B.~{\O}ksendal} and {T.~Zhang}.
{The general linear stochastic Volterra equation with anticipating coefficients}.
In {\it Stochastic analysis and applications}, World scientific Publishing,
343--366,
1996.
%
%
%
\bibitem{PaPr90}
{E. Pardoux} and {P. Protter}.
{Stochastic Volterra equations with anticipating coefficients}.
{\it Ann. Probab.},
{\bf 18},
1635--1655,
1990.
%
%

%
\bibitem{Po99}
{I. Podlubny}.
{\it Fractional differential equations}.
{Mathematics in Science and Engineering},
volume 198, Academic Press, Inc., San Diego, CA,
1999.


\bibitem{Ra10}
{M. Rahimy}.
{Applications of fractional differential equations}.
{\it Appl. Math. Sci.},
{\bf 4},
2453--2461,
2010.

\bibitem{ReCoAm21}
{Y. Ren}, {H. Coulibaly}, and {A. Aman}.
{Explicit solution for backward stochastic Volterra integral equations with linear time delayed generators}.
{\it preprint},
arXiv:2110.00753.

%
%

\bibitem{SaKiMa87}
{S. G. Samko}, {A. A. Kilbas}, and {O. I. Marichev}.
{\it Fractional Integrals and Derivatives, Theory and Applications}.
Gordon and Breach Science Publishers, Yverdon, Switzerland,
1987.
%
%
\bibitem{ShWaYo13}
{Y. Shi}, {T. Wang}, and {J. Yong}.
{Mean-field backward stochastic Volterra integral equations}.
{\it Discrete Contin. Dyn. Syst.},
{\bf 18}(7),
1929--1967,
2013.

\bibitem{ShWaYo15}
{Y. Shi}, {T. Wang}, and {J. Yong}.
{Optimal control problems of forward-backward stochastic Volterra integral equations}.
{\it Math. Control Relat. Fields},
{\bf 5}(3),
613--649,
2015.
%
\bibitem{ShWeXi20}
{Y. Shi}, {J. Wen}, and {J. Xiong}.
{Backward doubly stochastic Volterra integral equations and their applications}.
{\it J. Differential Equations},
{\bf 269}(9),
6492--6528,
2020.

%

\bibitem{WaSuYo21}
{H. Wang}, {J. Sun}, and {J. Yong}.
{Recursive utility processes, dynamic risk measures and quadratic backward stochastic Volterra integral equations}.
{\it Appl. Math. Optim.},
{\bf 84},
145--190,
2021.
%
%
\bibitem{WaYo21}
{H. Wang} and {J. Yong}.
{Time-inconsistent stochastic optimal control problems and backward stochastic Volterra integral equations}.
{\it ESAIM Control Optim. Calc. Var.},
{\bf 27}(22),
40 pages,
2021.
%
\bibitem{WaYoZh20}
{H. Wang}, {J. Yong}, and {J. Zhang}.
{Path dependent Feynman--Kac formula for forward backward stochastic Volterra integral equations}.
{\it Ann. Inst. Henri Poincar\`{e} Probab. Stat.},
{\bf 58}(2),
603--638,
2022.
%
%
%
%
\bibitem{WaT18}
{T. Wang}.
{Linear quadratic control problems of stochastic Volterra integral equations}.
{\it ESAIM Control Optim. Calc. Var.},
{\bf 24}(4),
1849--1879,
2018.
%
%
\bibitem{WaT20}
{T. Wang}.
{Necessary conditions of Pontraygin's type for general controlled stochastic Volterra integral equations}.
{\it ESAIM Control Optim. Calc. Var.},
{\bf 26}(16),
29 pages,
2020.

%
%
\bibitem{WaTZh17}
{T. Wang} and {H. Zhang}.
{Optimal control problems of forward-backward stochastic Volterra integral equations with closed control regions}.
{\it SIAM J. Control Optim.},
{\bf 55}(4),
2574--2602,
2017.

\bibitem{WaXuKl16}
{Y. Wang}, {J. Xu}, and {P. E. Kloeden}.
{Asymptotic behavior of stochastic lattice systems with a Caputo fractional time derivative}.
{\it Nonlinear Anal.},
{\bf 135},
205--222,
2016.
%
%
%


%
%
%
%
%
\bibitem{Yo06}
{J. Yong}.
{Backward stochastic Volterra integral equations and some related problems}.
{\it Stoch. Anal. Appl.},
{\bf 116}(5),
779--795,
2006.
%
\bibitem{Yo07}
{J. Yong}.
{Continuous-time dynamic risk measures by backward stochastic Volterra integral equations}.
{\it Appl. Anal.},
{\bf 86}(11),
1429--1442,
2007.
%
\bibitem{Yo08}
{J. Yong}.
{Well-posedness and regularity of backward stochastic Volterra integral equations}.
{\it Probab. Theory Related Fields},
{\bf 142}(1-2),
2--77,
2008.
%


%


\end{thebibliography}
\end{document}